\documentclass[leqno]{amsart}%
\usepackage{hyperref}
\usepackage{pdfsync}
\usepackage{dsfont}
\usepackage{color}
\usepackage{amsthm}
\usepackage{algorithm}
\usepackage{algorithmic}
\usepackage{braket}
\usepackage{bm}
\usepackage{amsmath}
\usepackage{graphicx}
\usepackage{caption}
\usepackage{subcaption}
\usepackage{pb-diagram}
\usepackage{amsfonts}
\usepackage{amssymb}%
\providecommand{\U}[1]{\protect\rule{.1in}{.1in}}
\numberwithin{equation}{section}
\newtheorem{theorem}{Theorem}[section]

\newtheorem{corollary}[theorem]{Corollary}

\newtheorem{definition}[theorem]{Definition}

\newtheorem{lemma}[theorem]{Lemma}

\newtheorem{proposition}[theorem]{Proposition}

\def\Xint#1{\mathchoice {\XXint\displaystyle\textstyle{#1}}
{\XXint\textstyle\scriptstyle{#1}}
{\XXint\scriptstyle\scriptscriptstyle{#1}}
{\XXint\scriptscriptstyle\scriptscriptstyle{#1}} \!\int}
\def\XXint#1#2#3{{\setbox0=\hbox{$#1{#2#3}{\int}$} \vcenter{\hbox{$#2#3$}}\kern-.5\wd0}}

\def\dashint{\Xint-}

\theoremstyle{remark}
\newtheorem{remark}[theorem]{Remark}
\newcommand{\R}{\mathbb{R}}
\newcommand{\N}{\mathbb{N}}

\newcommand{\x}{x}
\newcommand{\uu}{\mathbf{u}}
\newcommand{\z}{\mathbf{z}}

\newcommand{\te}{\textrm}

\newcommand{\veps}{\varepsilon}

\DeclareMathOperator{\id}{Id}
\DeclareMathOperator{\supp}{supp}

\DeclareMathOperator{\divergence}{div}

\DeclareMathOperator{\Proj}{Proj}

\DeclareMathOperator{\Lip}{Lip}
\DeclareMathOperator{\Span}{Span}

\definecolor{mygreen}{rgb}{0.1,0.75,0.2}

\newcommand{\nc}{\normalcolor}

\newenvironment{list1}
  {\begin{list}
 {\textsc{\arabic{broj1}.} }
 {\usecounter{broj1}
  \setlength{\itemindent}{-1pt}
  \setlength{\listparindent}{-1pt}
  \setlength{\itemsep}{0pt}}
  \setlength{\labelwidth}{30pt}
  \setlength{\parsep}{0pt}
  }
  {\end{list}}

\title{A variational approach to the consistency of spectral clustering}
\author{Nicol\'as Garc\'ia Trillos and  Dejan Slep\v{c}ev}

\keywords{spectral clustering, graph Laplacian, point cloud, discrete to continuum limit, Gamma-convergence, Dirichlet energy,  random geometric graph }
\subjclass{49J55, 49J45, 60D05, 68R10, 62G20}
\begin{document}
\newcounter{broj1}
\date{\today}
\maketitle

\begin{abstract}
This paper establishes the consistency of spectral approaches to data clustering. 
We consider clustering of point clouds obtained as samples of a ground-truth measure. 
A graph representing the point cloud is obtained by assigning weights to edges based on the distance between the points they connect. 
We investigate the spectral convergence of both unnormalized and normalized graph Laplacians towards the appropriate operators in the continuum domain.
We obtain sharp conditions on how the connectivity radius can be scaled with respect to the number of sample points for the spectral convergence to hold. 
 We also show that the discrete clusters obtained via spectral clustering converge towards a continuum partition  of the ground truth measure. Such continuum partition minimizes a functional describing the continuum analogue of the graph-based spectral partitioning. 
Our approach, based on variational convergence, is general and flexible.
\end{abstract}

\section{Introduction}

Clustering is one of the basic problems of  statistics and machine learning: having a collection of $n$ data points
and a measure of their pairwise similarity the task is to partition the data into $k$ meaningful groups.  
There is a variety of criteria for the quality of partitioning and a plethora of clustering algorithms, overviewed in \cite{Fort10, Scha07, XuWuClusSurv, YanLes15}. 
Among most widely used are centroid based (for example the $k$-means algorithm),  agglomeration based (or hierarchical) and graph based ones.  Many graph partitioning approaches are based on dividing the graph representing the data into clusters of balanced sizes which have as few as possible edges between them \cite{ACPP12, ARV, BalancedCutHein, ShiMalik, SpiTen13, SzlamBresson, WeiChe89}.  
 Spectral clustering is a relaxation of minimizing graph cuts, which in any of its variants, \cite{JordanNg, ShiMalik, vonLux_tutorial},  consists of two steps.  The first step is the embedding step where data points are mapped to a euclidean space by using the spectrum of a \textit{graph Laplacian}. In the second step, the actual clustering is obtained by applying a clustering algorithm like $k$-means to the transformed points.  

The input of a spectral clustering algorithm is a weight matrix $W$ which captures the similarity relation between the data points. Typically, the choice of edge weights depends on the distance between the data points and a parameter $\veps$ which determines the length scale over which points are connected. We assume that the data set is a random sample of an underlying ground-truth measure. We investigate the convergence of spectral clustering as the number of available data points goes to infinity.  
   
For any given clustering procedure, a natural and important question is whether the procedure is consistent. That is, if it is true that as more data is collected, the partitioning of the data into groups obtained converges to some meaningful partitioning in the limit. Despite the abundance of clustering procedures in the literature, not many results establish their consistency in the nonparametric setting, where the data is assumed to be obtained from a unknown general distribution. 
Consistency of $k$-means clustering was established by Pollard \cite{pollard1981strong}.
%\red What else could we cite??\nc
Consistency of $k$-means clustering for paths with regularization was recently studied by Thorpe, Theil and Cade \cite{ThoTheCad15}, using a similar viewpoint to those of this paper.
Consistency for a class of single linkage clustering algorithms was shown by Hartigan \cite{Hartigan1981}.
Arias-Castro and Pelletier have proved the consistency of maximum variance unfolding \cite{ACPell13}.
Pointwise estimates between graph Laplacians and the continuum operators were studied by
Belkin and Niyogi \cite{bel_niy_LB},   % pointwise manifold LB, 
Coifman and Lafon \cite{coifman1},
Gin\'e and Koltchinskii \cite{GK}, Hein, Audibert and von Luxburg \cite{hein_audi_vlux05}, and %on weighted graph Laplacian on manifolds,  
Singer \cite{Singer}.
Spectral convergence was studied in the works of Ting, Huang, and  Jordan \cite{THJ}, Belkin and Niyogi \cite{belkin2007convergence} on the convergence of Laplacian eigenmaps, von Luxburg, Belkin and Bousquet on graph Laplacians, and of  Singer and Wu \cite{SinWu13} on connection graph Laplacian. 
The convergence of the eigenvalues and eigenvectors these works obtain is of great relevance to machine learning. However obtaining practical and rigorous rates at which the connectivity length scale $\veps_n \to 0$ as $n \to \infty$ remained an open problem.
Also relevant to point cloud analysis are studies of Laplacians on discretized manifolds by Burago, Ivanov and Kurylev  \cite{BIK} who obtain precise error estimates for eigenvalues and eigenvectors.
% and the consistency of the spectral clustering  by Belkin and Nyogi \cite{belkin2007convergence} and von Luxburg, Belkin, and Bousquet\cite{vonLuxburg}. 

Recently the authors in \cite{GTS}, and together with Laurent, von Brecht and Bresson in\cite{CheegerRatioConsistency},
introduced a framework for showing the consistency of clustering algorithms based on minimizing an objective functional on graphs. In \cite{CheegerRatioConsistency} they applied the technique to Cheeger and Ratio cuts. 
%
%Some of the consistency results related to graph Laplacians  are due to Belkin and Niyogi \cite{bel_niy_LB},   % pointwise manifold LB, 
%Gin\'e and Koltchinskii \cite{GK}, Hein, Audibert and von Luxburg \cite{hein_audi_vlux05}, and %on weighted graph Laplacian on manifolds,  
%Singer \cite{Singer}. All such works establish consistency of graph Laplacians in a pointwise setting. On the other hand, the works of Ting, Huang, and  Jordan \cite{THJ}, Belkin and Niyogi \cite{belkin2007convergence} on the convergence of Laplacian eigenmaps, and von Luxburg, Belkin and Bousquet on consistency of spectral clustering \cite{vonLuxburg}, consider spectral convergence and thus convergence of eigenvalues and eigenvectors, relevant for machine learning. An important difference between our work and the works previously mentioned is that in them, there is no explicit rate at which $\veps_n$ (the length scale used to construct the graph) is allowed to converge to $0$ as $n \to \infty$. 
%
Here the framework of \cite{GTS, CheegerRatioConsistency} is used to prove new results on consistency of spectral clustering, which establish the (almost) optimal rate at which the connectivity radius $\veps$ can be taken to $0$ as $n \to \infty$.
We prove the convergence of  the spectrum of the graph Laplacian towards the spectrum of a corresponding continuum operator. An important element of our work is that we
establish the convergence of the discrete clusters obtained via spectral clustering to their continuum counterparts. That is, as the number of data points $n \to \infty$ the discrete clusters (obtained via spectral clustering) are show to converge towards continuum objects (measures), which themselves are obtained via a clustering procedure in the continuum setting (performed on the ground truth measure). That is, the discrete clusters are shown to converge to continuum clusters obtained via spectral clustering procedure with full information (ground truth measure) available.  \nc
 We obtain results for unnormalized
(Theorem \ref{ConvergenceSpectrumTheorem}),
 and normalized (Theorems 
\ref{ConvergenceSpectrumTheoremNormalized} and \ref{ConvergenceSpectrumTheoremNormalizedRW}) graph Laplacians. 
 % In order to show that such convergence holds, we use a different approach to the one in \cite{vonLuxburg} where operator perturbation theory was used.
 The bridge connecting the spectrum of the graph Laplacian and the spectrum of a limiting operator in the continuum is built by using the notion of variational convergence known as $\Gamma$-convergence. The setting of $\Gamma$-convergence, combined with techniques of optimal transportation, provides an effective viewpoint to address a range of consistency and stability problems based on minimizing objective functionals on a random sample of a measure.

\subsection{Description of spectral clustering}
\label{sec:DescriptionSpectralClustering}
Let $V= \{\x_1, \dots, \x_n\}$ be a set of vertices  and let $W \in \R^{n \times n}$ be a symmetric matrix with non-negative entries. We define $\mathcal{D}\in \R^{n \times n}$, the \emph{degree matrix} of the weighted graph $(V, W)$, to be the diagonal matrix with $\mathcal{D}_{ii}= \sum_{j}W_{i,j}$ for every $i$. Also, we define $L$, the  \emph{unnormalized graph Laplacian} matrix of the weighted graph $(V,W)$, to be
\begin{equation} \label{GL}
 L :=  \mathcal{D}- W.  
\end{equation}
We also consider the matrices $N^{sym}$ and $N^{rw}$ given by
$$ N^{sym}:= \mathcal{D}^{-1/2} L \mathcal{D}^{-1/2}, \quad N^{rw}:= \mathcal{D}^{-1} L, $$
both of which we refer to as \emph{normalized graph Laplacians}. The superscript $sym$ indicates the fact that $N^{sym}$ is symmetric, whereas the superscript $rw$ indicates the fact that $N^{rw}$ is connected to the transition probabilities of a random walk that can be defined on the graph. Each of the matrices $L, N^{sym}, N^{rw}$ is used in a version of spectral clustering. The so called unnormalized spectral clustering uses the spectrum of the unnormalized graph Laplacian to embed the point cloud into a lower dimensional space, typically a method like $k$-means on the embedded points then provides the desired clusters (see \cite{vonLux_tutorial}). This is Algorithm \ref{alg1} below. 
\begin{algorithm*}                      % enter the algorithm environment
\caption{Unnormalized spectral clustering}          % give the algorithm a caption
\label{alg1}
\begin{itemize}
\item[]\textbf{Input:} Number of clusters $k$ and similarity matrix $W$.
\item[-] Construct the unnormalized graph Laplacian $L$.
\item[-] Compute the eigenvectors $u_1, \dots,u_k$ of $L$ associated to the $k$ smallest (nonzero) eigenvalues of $L$. 
\item[-] Define the matrix $U \in \R^{k\times n}$, where the $i$-th row of $U$ is the vector $u_i$.
\item[-] For $i=1,\dots,n$, let $y_i \in \R^k$ be the $i$-th column of $U$. 
\item[-] Use the $k$-means algorithm to partition the set of points $\left\{y_1,\dots, y_n  \right\}$  into $k$ groups, that we denote by $G_1,\dots, G_k$.  
\item[]\textbf{Output:} Clusters $G_1,\dots, G_k$.  
\end{itemize}
% and a label for \ref{} commands later in the document
%and a label for \ref{} commands later in the document
\end{algorithm*}

In the same spirit, the normalized graph Laplacians are used. An algorithm for normalized spectral clustering using $N^{sym}$ was introduced in \cite{JordanNg} (see Algorithm \ref{alg2}), and an algorithm using $N^{rw}$ was introduced in \cite{ShiMalik} (see Algorithm \ref{alg3}).  

\begin{algorithm*}                      % enter the algorithm environment
\caption{Normalized spectral clustering as defined in \cite{JordanNg}}          % give the algorithm a caption
\label{alg2}
\begin{itemize}
\item[]\textbf{Input:} Number of clusters $k$ and similarity matrix $W$.
\item[-] Construct the normalized graph Laplacian $N^{sym}$.
\item[-] Compute the eigenvectors $u_1, \dots,u_k$ of $N^{sym}$ associated to the $k$ smallest (nonzero) eigenvalues of $N^{sym}$. 
\item[-] Define the matrix $U \in \R^{k\times n}$, where the $i$-th row of $U$ is the vector $u_i$.
\item[-] Construct the matrix $V$ by normalizing the columns of $U$ so that the columns of $V$ have all euclidean norm equal to one.  
\item[-] For $i=1,\dots,n$, let $y_i \in \R^k$ be the $i$-th column of $V$. 
\item[-] Use the $k$-means algorithm to partition the set of points $\left\{y_1,\dots, y_n  \right\}$  into $k$ groups that we denote by $G_1,\dots, G_k$.  
\item[]\textbf{Output:} Clusters $G_1,\dots, G_k$.  
\end{itemize}
% and a label for \ref{} commands later in the document
%and a label for \ref{} commands later in the document
\end{algorithm*}

\begin{algorithm*}                      % enter the algorithm environment
\caption{Normalized spectral clustering as defined in \cite{ShiMalik}}          % give the algorithm a caption
\label{alg3}
Same as Algorithm 1 but using the normalized graph Laplacian $N^{rw}$ instead of $L$.
%\begin{itemize}
%\item[]\textbf{Input:} Number of clusters $k$ and similarity matrix $W$.
%\item[-] Construct the normalized graph Laplacian $N^{rw}$.
%\item[-] Compute the eigenvectors $u_1, \dots,u_k$ of $N^{rw}$ associated to the $k$ smallest (nonzero) eigenvalues of $N^{rw}$. 
%\item[-] Define the matrix $U \in \R^{k\times n}$, where the $i$-th row of $U$ is the vector $u_i$.
%\item[-] For $i=1,\dots,n$, let $y_i \in \R^k$ be the $i$-th column of $U$. 
%\item[-] Use the $k$-means algorithm to partition the set of points $\left\{y_1,\dots, y_n  \right\}$  into $k$ groups that we denote by $G_1,\dots, G_k$.  
%\item[]\textbf{Output:} Clusters $G_1,\dots, G_k$.  
%\end{itemize}
\end{algorithm*}
Spectral properties of graph Laplacians have connections to balanced graph cuts. For example, the spectrum of $N^{rw}$ is shown to be connected to the Ncut problem, whereas the spectrum of $L$ is connected to RatioCut (see \cite{vonLux_tutorial}). A probabilistic interpretation of the spectrum of $N^{rw}$ may be found in \cite{DiffMapsProbabilistic}. In addition, connections between normalized graph Laplacians, data parametrization and dimensionality reduction via diffusion maps are developed in \cite{DiffusionMapsCoarseGraining}.
% We present two different normalized graph clustering algorithms one of which uses $N^{sym}$ and is proposed in \cite{JordanNg} (see Algorithm \ref{alg2} below), and the other one which uses $N^{rw}$ and is proposed in \cite{ShiMalik} (see Algorithm \ref{alg3} below).

We now present some facts about the matrices $L, N^{sym}$ and $N^{rw}$, all of which may be found in \cite{vonLux_tutorial}. First of all $L$ is a positive semidefinite symmetric matrix. In fact for every vector $u \in \R^n$
\begin{equation}
 \langle Lu, u \rangle= \frac{1}{2} \sum_{i,j}W_{i,j}(u_i - u_j)^2,
 \label{LaplacianAndEnergy0}
\end{equation}
where on the left hand side we are using the usual inner product in $\R^n$. The smallest eigenvalue of $L$ is equal to zero, and its multiplicity is equal to the number of connected components of the weighted graph. The matrix $N^{sym}$ is symmetric and positive semidefinite as well. Moreover, for every $u \in \R^n$
\begin{equation}
 \langle N^{sym}u, u \rangle= \frac{1}{2} \sum_{i,j}W_{i,j}\left(\frac{u_i}{\mathcal{D}_{ii}} - \frac{u_j}{\mathcal{D}_{jj}}\right)^2.
 \label{LaplacianAndEnergyN0}
\end{equation}
In addition, $0$ is an eigenvalue of $N^{sym}$, with multiplicity equal to the number of connected components of the weighted graph. The vector $\mathcal{D}^{1/2} \textbf{1}$ (where $\textbf{1}$ is the vector with all entries equal to one) is an eigenvector of $N^{sym}$ with eigenvalue $0$. 

The two forms of normalized graph Laplacians are closely related due to the correspondence between the spectruma of $N^{sym}$ and $N^{rw}$. In fact, it is straightforward to show that
\begin{equation}
N^{rw}u = \lambda u \quad \text{if and only if} \quad N^{sym} w = \lambda w, \quad \text{where } w = \mathcal{D}^{1/2}u.
\label{ConditionNsymNrw}
\end{equation}
That is, $N^{sym}$ and $N^{rw}$ have the same eigenvalues, and there is a simple relation between their corresponding eigenvectors.

\subsection{Spectral clustering of point clouds.}
Let $V= \{\x_1, \dots, \x_n\}$ be a point cloud in $\R^d$. To give a weighted graph structure to the set $V$, we consider a kernel $\eta$, that is, we consider $\eta : \R^d \to [0, \infty)$ a radially symmetric, radially decreasing function decaying to zero sufficiently fast. The kernel is appropriately rescaled to take into account data density. In particular, let $\eta_\veps$ depend on the length scale $\veps$ where we take $\eta_\veps: \R^d \rightarrow \R $ to be defined by
$$\eta_\veps(z) := \frac{1}{\veps^d}\eta\left(\frac{z}{\veps} \right).$$
In this way we impose that significant weight is given to edges connecting points up to distance $\veps$. We consider the similarity matrix $W^\veps$ defined by
\begin{equation}
\label{weights} 
W^\veps_{i,j} = \eta_\veps(\x_i - \x_j).
\end{equation}
%\red Nicolas, we should discuss if we want $W_{i,i} > 0$ or not.  \nc % ??
%and set $W^\veps_{i,i}=0$ for all $i$. \nc
We denote by $\mathcal{L}_{n, \veps}$ the unnormalized graph Laplacian \eqref{GL} of the weighted graph $(V,W^\veps)$, that is
\begin{equation} \label{GLen}
\mathcal{L}_{n, \veps} = \mathcal{D^\veps} - W^\veps 
\end{equation}
where $\mathcal D^\veps$ is the diagonal matrix with $\mathcal D^\veps_{i,i} = \sum_j W^\veps_{i,j}$. 

We define the \emph{Dirichlet energy on the graph} of a function $u : V \rightarrow \R$ to be
\begin{equation}  \sum_{i,j}W^\veps_{i,j}(u(\x_i) - u(\x_j))^2.
\label{Perimeter}
\end{equation}
 The fact that $\eta$ is a symmetric function guarantees that $W$ is symmetric and thus all the facts presented in Subsection \ref{sec:DescriptionSpectralClustering} apply. In particular, \eqref{LaplacianAndEnergy0} can be stated as: for every function $u : V \rightarrow \R$
\begin{equation}
 \langle \mathcal{L}_{n, \veps}u, u \rangle= \frac{1}{2} \sum_{i,j}W^\veps_{i,j}(u(\x_i) - u(\x_j))^2,
 \label{LaplacianAndEnergy}
\end{equation}
where on the left hand side we have identified the function $u$ with the vector $(u(\x_1), \dots, u(\x_n))$ in $\R^n$ and where $\langle \cdot, \cdot \rangle$ denotes the usual inner product in $\R^n$. 
%We use this identification in the remainder of the paper without further mention. 

The symmetric normalized graph Laplacian   $ \mathcal{N}^{sym}_{n , \veps} $  is given by
$$  \mathcal{N}_{n , \veps}^{sym} := \mathcal{D}^{ -1/2} \mathcal{L}_{n , \veps} \mathcal{D} ^{-1/2}.  $$

 Since the kernel $\eta$ is assumed radially symmetric, it can be defined as $\eta(x) := \bm{\eta}(|x|)$ for all $x \in \R^d$, where  $\bm{\eta}: [0, \infty) \rightarrow [0, \infty)$ is the radial profile.  We assume the following properties on $\bm{\eta}$:
\begin{itemize}
\item[\textbf{(K1)}] $\bm{\eta}(0)>0$ and $\bm{\eta}$ is continuous at $0$.  
\item[\textbf{(K2)}] $\bm{\eta}$ is non-increasing.
\item[\textbf{(K3)}] The integral $\int_{0}^{\infty} \bm{\eta}(r) \, r^{d+1 } dr $ is finite.
\end{itemize}
\begin{remark}
We remark that the last assumption on $\bm{\eta}$ is equivalent to imposing that the \emph{surface tension}
 \begin{equation} \label{sigma_eta}
 \sigma_\eta := \int_{\R^d} \eta(h)|h_1|^2dh  
\end{equation}
 is finite, where $h_1$ represents the first component of $h$. The second condition implies that more relevance is given to the interactions between points that are close to each other. We notice that the class of acceptable kernels is quite broad and includes both Gaussian kernels and discontinuous kernels like one defined by a function $\bm{\eta}$ of the form $\bm{\eta}=1$ for $t \leq 1$ and $\bm{\eta}=0$ for $t>1$. 
\end{remark}

\medskip

We focus on point clouds that are obtained as independent samples from a given distribution $\nu$. Specifically, consider an open, bounded, and connected set $D \subset \R^d$ with Lipschitz boundary (i.e. locally the graph of a Lipschitz function) and consider a probability measure $\nu$ supported on $\overline {D}$. We assume $\nu$ has a continuous density $\rho$, which is bounded above and below by positive constants on $D$.
We assume that the points $\x_1, \dots , \x_n$ (i.i.d. random points) are chosen according to the distribution $\nu$. We consider the graph with nodes $V=\{\x_1, \dots, \x_n\}$ and edge weights $\left\{W^\veps_{i,j} \right\}_{i,j}$ defined in \eqref{weights}.
For an appropriate scaling of $\veps:=\veps_n$ with respect to $n$,  we study the limiting behavior of the eigenvalues and eigenvectors of the graph Laplacians as $n \to \infty$. We now describe the continuum problems which characterize the limit.

%%%%%%%%%%%%%%%%%%%%%%%%%%%%%%%%%%%%%%%%%%%%%%%
\subsection{Description of spectral clustering in the continuum setting: the unnormalized case} \label{DescContClustU}

Let domain $D$, "ground-truth" measure $\nu$ with density $\rho$ be as above. 
The object that characterizes the limit of the graph Laplacians $\mathcal{L}_{n, \veps_n}$ as $n \rightarrow \infty$ is the differential operator:
\begin{equation} \label{defL}
\mathcal{L}:u \mapsto - \frac{1}{\rho} \divergence(\rho^2 \nabla u).
\end{equation}
We consider the pairs $\lambda\in \R$ and $u \in H^1(D)$ (the Sobolev space of $L^2(D)$ functions with distributional derivative $\nabla u$ in $L^2(D, \R^d)$), with $u$ not identically equal to zero, such that
\begin{alignat}{2}  \label{PDELimitFunctional0} 
\begin{aligned}
 \mathcal{L}u & =  \lambda u ,  \quad &  &  \te{in } D, \\
  \frac{\partial u}{\partial \bf{n}} & =0, &  & \te{on }  \partial D.
\end{aligned}
\end{alignat}
%
%
%In fact, we consider the values $\lambda \in \R$ for which there exists nontrivial solutions $u: D \rightarrow \R$ to the PDE
%\begin{equation}
%\label{PDELimitFunctional0} 
% \left\{ \begin{array}{l}
% - \frac{1}{\rho}\divergence(\rho^2 \nabla u  )=  \lambda u , \:  \te{in} \:  D 
%\\ \frac{\partial u}{\partial \bf{n}} =0, \: \te{on} \: \partial D
%\end{array} \right.,
%\end{equation}
A function $u$ as above is said to be an eigenfunction of $\mathcal{L}$ with corresponding eigenvalue $\lambda \in \R$. 
In Subsection \ref{SectionSpectrumL} we discuss the precise definition of a solution of 
 \eqref{PDELimitFunctional0} and present some facts about it. In particular  $\mathcal{L}$ is a positive semidefinite self-adjoint operator with respect to the inner product $\langle \cdot, \cdot \rangle_{L^2(D, \nu)}$ and has a discrete spectrum that can be arranged as an increasing sequence converging to infinity
 $$  0 = \lambda_1 \leq \lambda_2 \leq \dots ,$$
where each eigenvalue is repeated according to (finite) multiplicity. Furthermore, there exists a orthonormal basis of $L^2(D)$ (with respect to the inner product $\langle\cdot ,\cdot \rangle_{L^2(D, \nu)}$) consisting of eigenfunctions $u_i$ of $\mathcal{L}$. 

Given a mapping $\Phi: D \longrightarrow \R^k$ by $\Phi_\sharp \nu$ we denote the push forward of the measure $\nu$, namely the measure for which $\Phi_\sharp \nu(A) = \nu(\Phi^{-1}(A))$, for any Borel set $A$.
The continuum spectral clustering analogous to the discrete one of Algorithm 1 is as follows. 
Let $u_1,\dots, u_k : D \to \R$ be the orthonormal set of eigenfunctions corresponding to eigenvalues 
$\lambda_1, \dots, \lambda_k$. Consider the measure $\mu = (u_1, \dots, u_k)_\sharp \nu$. 
Let  $\tilde G_i \subset \R^k$  be the clusters obtained by k-means clustering of $\mu$.
Then $G_i = (u_1, \dots, u_k)^{-1}(\tilde G_i)$ for $i=1, \dots, k$ define  the \emph{spectral clustering} of $\nu$. 

%%%%%%%%%%%%%%%%%%%%%%%%%%%%%%%%%%%%%%%%%%%%%%%%%%%
\subsection{Description of spectral clustering in the continuum setting: the normalized cases} \label{DescContClustN}

The object that characterizes the limit of the symmetric normalized graph Laplacians $\mathcal{N}^{sym}_{n, \veps_n}$ as $n \rightarrow \infty$ is the differential operator 
$$\mathcal{N}^{sym}: u \mapsto - \frac{1}{\rho^{3/2}}\divergence \left(\rho^2 \nabla \left( \frac{u}{\sqrt{\rho}}\right)  \right).$$
We consider the space
\begin{equation}
  H^1_{\sqrt{\rho}}(D) := \left\{ u \in L^2(D) \: : \:  \frac{u}{\sqrt{\rho}} \in H^1(D)  \right\}. 
\label{H1sqrtrho}
\end{equation}
The spectrum of $\mathcal{N}^{sym}$ is the set of pairs $\tau \in \R$ and $u \in H^1_{\sqrt{\rho}}(D)$, where $u$ is not identically equal to zero, such that 
\begin{alignat}{2} 
\label{PDELimitFunctionalN} 
\begin{aligned}
 \mathcal{N}^{sym}(u) &= \tau u , &  &  \te{in }   D \\ 
 \frac{\partial (u / \sqrt{\rho}) }{\partial \bf{n}} & =0 & & \te{on }  \partial D.
\end{aligned}
\end{alignat}
The sense in which \eqref{PDELimitFunctionalN} holds is made precise in Subsection \ref{SectionSpectrumL}. The spectrum of the operator $\mathcal{N}^{sym}$ has similar properties to those of the spectrum of $\mathcal{L}$. We let 
$$0 = \tau_1 \leq \tau_2 \leq \dots,$$
denote the eigenvalues of $\mathcal{N}^{sym}$, repeated according to multiplicity.
\medskip

The continuum spectral clustering analogous to the discrete one of Algorithm 2 is as follows. 
Let $u_1,\dots, u_k : D \to \R$ be the orthonormal  set of eigenfunctions (with respect to the inner product $\langle\cdot ,\cdot \rangle_{L^2(D, \nu)}$)  corresponding to eigenvalues 
$\tau_1, \dots, \tau_k$. 
Normalize them by
\[ (\tilde u_1(x), \dots, \tilde u_k(x)) = \frac{( u_1(x), \dots, u_k(x)) }{\| ( u_1(x), \dots, u_k(x))\|} \;\te{ for all } x \in D. \]
Consider the measure $\tilde \mu = (\tilde u_1, \dots, \tilde u_k)_\sharp \nu$. 
Let  $\tilde G_i \subset \R^k$  be the clusters obtained by k-means clustering of $\tilde \mu$.
Then $G_i = (\tilde{u}_1, \dots, \tilde{u}_k)^{-1}(\tilde G_i)$ for $i=1, \dots, k$ define  the \emph{spectral clustering} o
$\nu$.
\medskip

Finally, the operator that describes the limit of the graph Laplacians $\mathcal{N}^{rw}_{n,\veps_n} ={ \mathcal{D}^{\veps_n}}^{-1} \mathcal{L}_{n, \veps_n}$ is described by
the operator $\mathcal{N}^{rw}$:
$$ \mathcal{N}^{rw}(u) = - \frac{1}{\rho^2} \divergence(\rho^2 \nabla u ). $$
As discussed in Subsection \ref{SectionSpectrumL}, the eigenvalues of $\mathcal{N}^{rw}$ are equal to the eigenvalues of $\mathcal{N}^{sym}$. The continuum clustering, which is analogous to the discrete one of Algorithm 3, is as in Subsection \ref{DescContClustU}, where eigenfunctions of $\mathcal{N}^{rw}$ are used.

%%%%%%%%%%%%%%%%%%%%%%%%%
\subsection{Passage from discrete to continuum.}

We are interested in showing that as $n \to \infty$ eigenvalues of discrete graph Laplacians and the associated eigenvectors converge towards eigenvalues and eigenfunctions of corresponding differential operators. 
The issue that arises is how to compare functions on discrete and continuum setting.  Typically this is achieved by introducing an interpolation operator that takes discretely defined functions to continuum ones and a restriction operator which restricts the continuum function to the discrete setting.
For this setting to work some smoothness of functions considered is required. Furthermore the choice of the interpolation operator and its properties adds an intermediate step that needs to be understood.

We choose a different route and introduce a way to compare the functions between settings directly. 
This approach is quite general and does not require any regularity assumptions. 
We use  the $TL^p$-topologies introduced in \cite{GTS} and in particular in this paper we focus in the $TL^2$-topology that we now recall. Denote by $\nu_n$ the empirical measure associated to the $n$ data points, that is
\begin{equation} \label{empirical}
\nu_n:=\frac{1}{n}\sum_{i=1}^{n}\delta_{\x_i}.
\end{equation}
For a given function $u \in L^2(D,\nu)$, the question is how to compare $u$ with a function $v \in L^2(D,\nu_n)$ (a function defined on the set $V$). 
%Note that without regularity assumptions on $u$, a restriction of $u$ to the data points is not a well defined operation since $u$ in fact represents an equivalence class of functions that can disagree on sets of Lebesgue measure zero and in particular on the set $V$.
 More generally, one can consider the problem of how to compare functions in $L^2(D,\mu)$ with those in $L^2(D,\theta)$ for arbitrary probability measures $\mu$, $\theta$ on $D$. We define the set of objects that 
 includes both the functions in discrete setting and those in continuum setting as follows: 
\[ TL^2(D) := \{ (\mu, f) \; : \:  \mu \in \mathcal P(D), \, f \in L^2(D, \mu) \}, \]
where $\mathcal{P}(D)$ denotes the set of Borel probability measures on $D$. For $(\mu,f)$ and $(\theta,g)$ in $TL^2$ we define the distance
  \begin{equation*} 
d_{TL^2}((\mu,f), (\theta,g)) =
   \inf_{\pi \in \Gamma(\mu, \theta)} \left(\iint_{D \times D} |x-y|^2 +  |f(x)-g(y)|^2  d\pi(x,y) \right)^{\frac{1}{2}},
\end{equation*}
where $\Gamma(\mu, \theta)$ is the set of all {\em couplings} (or \emph{transportation plans})  between $\mu$ and $\theta$, that is, the set of all  Borel probability measures on $D \times D$ for which the marginal on the first variable is $\mu$ and the marginal on the second variable is $\theta$. It was proved in \cite{GTS} that $d_{TL^2}$ is indeed a metric on $TL^2$. As remarked in \cite{GTS}, one of the nice features of the convergence in $TL^2$ is that it simultaneously generalizes the weak convergence of probability measures and the convergence in $L^2$ of functions. It also provides us with a way to compare functions which are supported in sets as different as point clouds and continuous domains. In Subsection \ref{TT} we present more details about this metric.

For a given $\mu \in \mathcal{P}(D)$ we denote by $L^2(\mu)$ the space of $L^2$-functions with respect to the measure $\mu$. Also, for  $f,g \in L^2(\mu)$ we write
$$\langle f, g \rangle_{\mu}:=\int_{D} fg d\mu \quad \te{ and }  \quad  \| f \|_\mu^2 =   \langle f, f\rangle_{\mu}. $$
Finally, if the measure $\mu$ has a density $\rho$, that is, if $d \mu = \rho dx$, we may write $\langle f, g \rangle_{\rho}$ and $\| f\|_\rho$ instead of $\langle f, g \rangle_{\mu}$ and $\|f\|_\mu$.

\subsection{Convergence of eigenvalues, eigenvectors, and of spectral clustering: the unnormalized case.}

Here we present one of the main results of this paper.
We state the conditions on $\veps_n$ for the spectrum of 
 the unnormalized graph Laplacian $\mathcal{L}_{n , \veps_n}$, given in \eqref{GLen}, to converge to the spectrum of $\mathcal{L}$, given by \eqref{defL} and for the spectral clustering of Algorithm 1 to converge to the clustering of Subsection \ref{DescContClustU}.
Let $\lambda_1 \leq \lambda_2 \leq \cdots $ be the eigenvalues of $\mathcal L$ and $u_1, u_2, \dots$ the corresponding orthonormal eigenfunctions, as in  Subsection \ref{DescContClustU}. We recall that orthogonality is considered with respect to the inner product in $L^2(\nu)$. 
 
 To state the results it is convenient to introduce 
$ 0=\overline{\lambda}_1 < \overline{\lambda}_2 < \cdots, $
the sequence of distinct eigenvalues of $\mathcal{L}$. For a given $k \in \N$, we denote by $s(k)$ the multiplicity of the eigenvalue $\overline{\lambda}_k$ and we let $\hat{k} \in \N$ be such that 
$\overline{\lambda}_k =  \lambda_{\hat{k}+1} = \dots = \lambda_{\hat{k}+s(k)}.$
Also, we denote by $\Proj_k: L^2(\nu) \rightarrow L^2(\nu)$ the projection (with respect to the inner product $\langle \cdot , \cdot \rangle_{\nu}$) onto the eigenspace of $\mathcal{L}$ associated to the eigenvalue $\overline{\lambda}_k$. For all large enough $n$, we denote  by $\Proj_k^{(n)}: L^2(\nu_n) \rightarrow L^2(\nu_n)$ the projection (with respect to the inner product $\langle \cdot , \cdot \rangle_{\nu_n}$) onto the space generated by all the eigenvectors of $\mathcal{L}_{n, \veps_n}$ associated to the eigenvalues $\lambda_{\hat{k}+1}^{(n)}, \dots, \lambda_{\hat{k}+ s(k)}^{(n)}$. Here, as in the rest of the paper, we identify $\R^n$ with the space $L^2(D, \nu_n)$.

\begin{theorem}[Convergence of the spectra of the unnormalized graph Laplacians] Let $d \geq 2$ and let $D \subseteq \R^d$, be an open, bounded, connected set with Lipschitz boundary.  Let $\nu$ be a probability measure on $D$ with continuous density $\rho$, satisfying
\begin{equation}
(\forall x \in D) \quad m \leq \rho(x) \leq M,
  \label{DensityBound}
  \end{equation}
for some $0< m\leq M$. Let $\x_1, \dots, \x_n, \dots$ be a sequence of i.i.d. random points chosen according to $\nu$. Let $\left\{ \veps_n \right\}_{n \in \N}$ be a sequence of positive numbers converging to $0$ and satisfying
\begin{align} \label{HypothesisEpsilon}
\begin{split}
%\lim_{n \rightarrow \infty} \frac{\sqrt{\log(\log n)}}{ n^{1/2}} \frac{1}{\veps_n}& =0 \:\:  if \: d=1, \\
\lim_{n \rightarrow \infty} \frac{(\log n)^{3/4}}{ n^{1/2} } \frac{1}{\veps_n}& =0 \:\:  if \: d=2, \\
\lim_{n \rightarrow \infty} \frac{(\log n)^{1/d}}{ n^{1/d} } \frac{1}{\veps_n}& =0 \:\: if \: d\geq3.
\end{split}
\end{align}
Assume the kernel $\eta$ satisfies conditions (K1)-(K3). Then, with probability one, all of the following statements hold true:
\begin{list1}
\item  Convergence of Eigenvalues: For every $k \in \N$
\begin{equation}
\lim_{n \rightarrow \infty}    \frac{ 2\lambda_{k}^{(n)} }{ n \veps_n^{2} }= \sigma_\eta \lambda_k,
\label{EigenvalueConvergence}
\end{equation}
where $\sigma_\eta$ is defined in \eqref{sigma_eta}.
\item For every $k \in \N$, every  sequence $\left\{ u_k^n \right\}_{n \in \N}$ with $u_k^n$ an eigenvector of $\mathcal{L}_{n , \veps_n}$ associated to the eigenvalue $\lambda_k^{(n)}$ and with $\|u_k^n\|_{\nu_n}=1$ is pre-compact in $TL^2$. Additionally,
whenever $u_k^n \overset{{TL^2}}{\longrightarrow} u_k$ along a subsequence as $n \to \infty$, 
then $\| u_k\|_{\nu} =1$ 
and $u_k$ is an eigenfunction of $\mathcal{L}$ associated to $\lambda_k$. 
\item Convergence of Eigenprojections: For all $k \in \N$ and 
for arbitrary sequence 
 $v_n \in L^2(\nu_n)$, if $v_n \overset{{TL^2}}{\longrightarrow} v$ 
 as $n \to \infty$ along some subsequence. Then along that subsequence
 \[   \Proj^{(n)}_k(v_n) \overset{{TL^2}}{\longrightarrow} \Proj_k(v),\: \text{ as }  n \rightarrow \infty.\]
\item Consistency of Spectral Clustering. Let $G^n_1, \dots G^n_k$ be the clusters obtained in Algorithm 1. Let $\nu_i^{n}= \nu_n \llcorner_{G^n_i}$ (the restriction of $\nu_n$ to $G_i^n$) for $i=1, \dots, k$. Then $(\nu_1^{n}, \dots, \nu_k^{n})$ is precompact with respect to weak convergence of measures and furthermore if $(\nu_1^{n}, \dots, \nu_k^{n})$ converges along a subsequence to $(\nu_{1}, \dots, \nu_{k})$ then 
$(\nu_{1}, \dots, \nu_{k}) = (\nu_{\llcorner G_1}, \dots, \nu_{ \llcorner G_k})$ where $G_1, \dots, G_k$ is a spectral clustering of $\nu$, described in Subsection \ref{DescContClustU}.

\end{list1}
\label{ConvergenceSpectrumTheorem}
\end{theorem}

%\red
%\begin{remark}
%To clarify the second statement of Theorem \ref{ConvergenceSpectrumTheorem}, we remark that it is not true that the projection to the eigenspace associated to $k$-th smallest eigenvalue of $\mathcal{L}_{n,\veps_n}$, converges to the projection of the eiganspace associated to the $k$-th smallest eigenvalue of $\mathcal{L}$. To illustrate why such condition can not hold, let us consider an example in the context of two by two matrices. Let $A_n$ be the matrix
%\[A_n = \left(\begin{array}{cc}
%    1       & 0 \\
%    0      & 1 + \frac{1}{n}
%\end{array}\right)
% \]
%and let
%\[A = \left(\begin{array}{cc}
%    1       & 0 \\
%    0      & 1 
%\end{array}\right).
% \]
%Note that the projection of the vector $(1,1)$ on the eigenspace associated to the the smallest eigenvalue of $A_n$ ( namely $\lambda_n =1$) is the vector $(1,0)$, whereas the projection of $(1,1)$ on the eiganspace associated to the smallest eigenvalue of $A$ ( namely $\lambda=1$) is the vector $(1,1)$ itself ( as in this case such eigenspace is $\R^2$ itself). Note however that by considering a neighborhood $U$ around $\lambda =1$,  we deduce that for large enough $n$, the space spanned by the eiganspaces of those eigenvalues of $A_n$ in $U$ is $\R^2$. 
%\end{remark}

\begin{remark}
We remark that although the choice of the $TL^2$-topology used in the previous theorem may seem unusual at first sight, it actually reduces to a more common notion of convergence (like the one used in \cite{vonLuxburg} which we described below) in the presence of regularity assumptions on the density $\rho$ and the domain $D$. In fact, assume for simplicity that $D$ has smooth boundary and that $\rho$ is a smooth function. Consider $\left\{ u_k^n \right\}_{n \in \N}$ where $u_k^n$ is an eigenvector of $\mathcal{L}_{n , \veps_n}$ associated to the eigenvalue $\lambda_k^{(n)}$ and satisfying $\|u_k^n\|_{\nu_n}=1$. The second statement in Theorem \ref{ConvergenceSpectrumTheorem} says that up to subsequence, $u_{k}^n \overset{{TL^2}}{\longrightarrow} u_k$, where $u_k$ is an eigenfunction of $\mathcal{L}$ associated to $\lambda_k$. From the regularity theory of elliptic PDEs it follows that $u_k$ is smooth up to the boundary.  In particular, it makes sense to define a function $\tilde{u}_k^n$ on the point cloud, by simply taking the restriction of $u_k$ to the points $\left\{\x_1, \dots, \x_n \right\}$. It is straightforward to check that $\tilde{u}_k^n \overset{{TL^2}}{\longrightarrow} u_k$ due to the smoothness of $u_k$. In turn, $u_k^n \overset{{TL^2}}{\longrightarrow} u_k $, implies that $d_{TL^2}((\nu_n, \tilde{u}_k^n - u_k^n  ), ( \nu,0)     ) \rightarrow 0$. From this and Proposition \ref{EquivalenceTLp}, we conclude that
$$ \| u_k^n -  \tilde{u}_k^n\|_{\nu_n}  \rightarrow 0.    $$
This is precisely the mode of convergence used in \cite{vonLuxburg}.
\end{remark}
The proof of Theorem \ref{ConvergenceSpectrumTheorem} relies on the study of the limiting behavior of the following rescaled form of the Dirichlet energy \eqref{Perimeter} on the graph:
\begin{equation} \label{tTV}
G_{n , \veps}(u):= \frac{1}{\veps^2n^2}\sum_{i,j}W_{i,j}(u(\x_i)- u(\x_j) )^2.  
\end{equation}
The type of limit which is relevant for the problem, is the one given by variational convergence known as the  $\Gamma$-convergence. The notion of $\Gamma$-convergence is recalled in Subsection \ref{sec:gamma}. 
This notion of convergence is particularly suitable in order to study the convergence of minimizers of objective functionals on graphs as $n \to \infty$, as it is discussed in \cite{CheegerRatioConsistency}.

The relevant continuum energy is 
 the \emph{weighted Dirichlet} energy  $G: L^2(D) \rightarrow [0, \infty]$:
 \begin{equation} \label{WeightedDirichlet}
 G(u):= 
 \begin{cases}
  \int_{D}|\nabla u (x)|^2 \rho^2(x) dx   &  \te{if }   u \in H^{1}(D)  \\
\infty & \te{if } u \in L^2(D) \setminus H^{1}(D).
\end{cases}
\end{equation}

\begin{theorem}[$\Gamma$-convergence of Dirichlet energies]
Consider the same setting as in Theorem \ref{ConvergenceSpectrumTheorem} and the same assumptions on $\eta$  and on $\left\{ \veps_n \right\}_{n \in \N}$. Then, $ G_{n,\veps_n}$, defined by \eqref{tTV}, $\Gamma$-converge to $\sigma_\eta G$ as $n \rightarrow \infty$ in the $TL^2$ sense, where $\sigma_\eta$ is given by \eqref{sigma_eta}
 and $G$ is the weighted Dirichlet energy with weight $\rho^2$ defined in \eqref{WeightedDirichlet}.
Moreover, the sequence of functionals $\left\{ G_{n , \veps_n} \right\}_{n \in \N}$ satisfies the compactness property with respect to the $TL^2$-metric. That is, every sequence $\left\{ u_n \right\}_{n \in \N}$ with $u_n \in L^2(\nu_n)$ for which
$$  \sup_{n \in \N}\| u_n\|_{\nu_n} < \infty, \quad   \sup_{n \in \N } G_{n , \veps_n}(u_n) < \infty,  $$
is precompact in $TL^2$.
 \label{DiscreteGamma}
\end{theorem}

The fact that the weight in the limiting functional $G$ is $\rho^2$ (and not $\rho$) essentially follows from the fact that the graph Dirichlet energy defined in \eqref{tTV} is a double sum. This is the same weight that shows up in the study of the continuum limit of the \textit{graph total variation} in \cite{GTS}. Theorem \ref{DiscreteGamma} is analogous to Theorems 1.1 and  1.2 in \cite{GTS} combined.

\subsection{Convergence of eigenvalues, eigenvectors, and of spectral clustering: the normalized case.}

We also study the limit of the spectra of $ \mathcal{N}^{sym}_{n , \veps_n} $, the symmetric normalized graph Laplacian which we recall is given by,
$$  \mathcal{N}_{n , \veps_n}^{sym} :=\mathcal{D}^{-1/2} \mathcal{L}_{n , \veps_n} \mathcal{D}^{-1/2}.  $$
For a function $u: V \rightarrow \R$, \eqref{LaplacianAndEnergyN0} can be written as
\begin{equation}
\langle \mathcal{N}^{sym}_{n , \veps_n}u , u  \rangle = \frac{1}{2} \sum_{i,j}W_{i,j}\left(\frac{u(\x_i)}{\sqrt{\mathcal{D}_{ii}}} - \frac{u(\x_j)}{\sqrt{\mathcal{D}_{jj}}}\right)^2.
\label{LaplacianAndEnergyNormalized}
\end{equation}
We denote by 
$$0=\tau_{1}^{(n)}\leq \dots \leq \tau_{n}^{(n)}$$
the eigenvalues of $\mathcal{N}^{sym}_{n , \veps_n}$ repeated according to multiplicity. Their limit is described by  differential operator 
$$\mathcal{N}^{sym}: u \mapsto - \frac{1}{\rho^{3/2}}\divergence \left(\rho^2 \nabla \left( \frac{u}{\sqrt{\rho}}\right)  \right).$$
%We consider the space
%\begin{equation}
%  H^1_{\sqrt{\rho}}(D) := \left\{ u \in L^2(D) \: : \:  \frac{u}{\sqrt{\rho}} \in H^1(D)  \right\}. 
%\label{H1sqrtrho}
%\end{equation}
%The spectrum of $\mathcal{N}^{sym}$ is the set of pairs $\tau \in \R$ and $u \in H^1_{\sqrt{\rho}}(D)$, where $u$ is not identically equal to zero, such that 
%\begin{alignat}{2} 
%\label{PDELimitFunctionalN} 
%\begin{aligned}
% \mathcal{N}^{sym}(u) &= \tau u , &  &  \te{in }   D \\ 
% \frac{\partial (u / \sqrt{\rho}) }{\partial \bf{n}} & =0 & & \te{on }  \partial D.
%\end{aligned}
%\end{alignat}
%The sense in which \eqref{PDELimitFunctionalN} holds is made precise in Subsection \ref{SectionSpectrumL}. The spectrum of the operator $\mathcal{N}^{sym}$ has similar properties to those of $\mathcal{L}$. 
Let 
$$0 = \tau_1 \leq \tau_2 \leq \dots,$$
denote the eigenvalues of $\mathcal{N}^{sym}$, repeated according to multiplicity. We write 
$ 0=\overline{\tau}_1 < \overline{\tau}_2 < \dots, $
to denote the distinct eigenvalues of $\mathcal{L}$. For a given $k \in \N$, we denote by $s(k)$ the multiplicity of the eigenvalue $\bar{\tau}_k$ and we let $\hat{k} \in \N$ be such that 
$\bar{\tau}_k =  \tau_{\hat{k}+1} = \dots = \tau_{\hat{k}+s(k)}.$
We define $\Proj_k$ and $\Proj_k^{(n)}$ analogously to the way we defined them in the paragraph preceding Theorem \ref{ConvergenceSpectrumTheorem}. The following is analogous to Theorem \ref{ConvergenceSpectrumTheorem}. 
%\begin{equation}
%\label{PDELimitFunctionalNormalized} 
% \left\{ \begin{array}{l}
% -\frac{ \divergence(\rho^2 \nabla ( \frac{u}{\sqrt{\rho}} )  )}{\rho}=  \lambda u , \:  \te{in} \:  D 
%\\ \frac{\partial u}{\partial \bf{n}} =0, \: \te{on} \: \partial D.
%\end{array} \right.
%\end{equation}
%For the same scaling of $\veps$ with respect to $n$ as in Theorem \ref{ConvergenceSpectrumTheorem} we study the limit of the eigenvalues $\mu^{(n)}_{1} \leq \dots \leq \mu^{(n)}_n$ of $\mathcal{N}_{n , \veps}$ and the limit of the corresponding associated eigenvectors.  In the limit, the relevant PDE is
%\begin{equation}
%\label{PDELimitFunctionalNormalized} 
% \left\{ \begin{array}{l}
% - \divergence(\rho^2 \nabla ( \frac{u}{\sqrt{\rho}} )  )=  \lambda u \rho, \:  \te{in} \:  D 
%\\ \frac{\partial u}{\partial \bf{n}} =0, \: \te{on} \: \partial D.
%\end{array} \right.
%\end{equation}
 \begin{theorem}[Convergence of the spectra of the normalized graph Laplacians]
Consider the same setting as in Theorem \ref{ConvergenceSpectrumTheorem} and the same assumptions on $\eta$  and on $\left\{ \veps_n \right\}_{n \in \N}$. Then, with probability one, all of the following statements hold
\begin{list1}
\item Convergence of Eigenvalues: For every $k \in \N$
\begin{equation*}
\lim_{n \rightarrow \infty}    \frac{2 \tau_{k}^{(n)} }{ \veps_n^{2} }= \frac{\sigma_\eta}{\beta_\eta  } \tau_k,
\label{EigenvalueConvergence}
\end{equation*}
where $\beta_\eta$ is given by
\begin{equation}
\beta_\eta:= \int_{\R^d}\eta(h) dh,
\label{BetaEta}
\end{equation}
and where $\sigma_\eta$ is given by \eqref{sigma_eta}.
%Convergence of Eigenprojections: If $u_n \overset{{TL^2}}{\longrightarrow} u$ then, for every $k \in \N$
%$$   \tilde{\Proj}_{k}^{(n)}(u_n) \overset{{TL^2}}{\longrightarrow} \tilde{\Proj}_{k}(u),\: as \: n \rightarrow \infty, $$
%where $\tilde{\Proj}_{k}^{(n)}(u_n)$ is the projection (with respect to the inner product $\langle \cdot, \cdot \rangle_{\nu_n}$) of $u_n$ onto the eigenspace associated to the $k$-th smallest eigenvalue of $\mathcal{N}_{n , \veps_n}$ (among the different eigenvalues); $\tilde{\Proj}_{k}(u)$ is the projection ( with respect to the inner product $\langle\cdot, \cdot \rangle_{\rho}$)  of $u$ onto the eigenspace associated to the $k$-th smallest eigenvalue of $\mathcal{N}$ (among the different eiganvalues).
\item For every $k \in \N$, every  sequence $\left\{ u_k^n \right\}_{n \in \N}$ with $u_k^n$ being an eigenvector of $\mathcal{N}_{n , \veps_n}^{sym}$ associated to the eigenvalue $\tau_k^{(n)}$ and with $\|u_k^n\|_{\nu_n}=1$ is pre-compact in $TL^2$. 
Additionally, whenever $u_k^n \overset{{TL^2}}{\longrightarrow} u_k$ along a subsequence as $n \to \infty$, 
then $\| u_k\|_{\nu} =1$ \nc 
and $u_k$ is an eigenfunction of $\mathcal{N}$ associated to $\tau_k$. 
\item Convergence of Eigenprojections: For all $k \in \N$ and for arbitrary $v_n \in L^2(\nu_n)$, if $v_n \overset{{TL^2}}{\longrightarrow} v$   along a subsequence as $n \to \infty$ then,
\[   \Proj^{(n)}_k(u_n) \overset{{TL^2}}{\longrightarrow} \Proj_k(u),\: \text{ as }  n \rightarrow \infty, \,
\te{along the subsequence. \nc} \]
\item Consistency of Spectral Clustering. Assume $\rho \in C^1(D)$. Let $G^n_1, \dots G^n_k$ be the clusters obtained in Algorithm 2. Let $\nu_{i}^n= \nu_n \llcorner_{G^n_i}$ for $i=1, \dots, k$. Then $(\nu_1^{n}, \dots, \nu_k^{n})$ is precompact with respect to weak convergence of measures and furthermore if $(\nu_1^{n}, \dots, \nu_k^{n})$ converges along a subsequence to $(\nu_{1}, \dots, \nu_{k})$ then 
$(\nu_{1}, \dots, \nu_{k}) = (\nu_{\llcorner G_1}, \dots, \nu_{ \llcorner G_k})$ where $G_1, \dots, G_k$ is a spectral clustering of $\nu$, described in Subsection \ref{DescContClustN}.
\end{list1}
\label{ConvergenceSpectrumTheoremNormalized}
\end{theorem}
 The proof of the previous theorem is completely analogous to the one of Theorem \ref{ConvergenceSpectrumTheorem} once one has proved the variational convergence of the relevant energies. Indeed, consider $\overline{G}_{n , \veps}: L^2(\nu_n) \rightarrow \R$ defined by 
\begin{equation}
\overline{G}_{n , \veps}(u):= \frac{1}{n \veps^2 } \sum_{i,j}W_{i,j} \left( \frac{u(\x_i)}{\sqrt{\mathcal{D}_{ii}}} - \frac{u(\x_j)}{\sqrt{\mathcal{D}_{jj}}}\right)^2
\label{NormalizedDiscreteDirichlet}
\end{equation}
and $\overline{G}: L^2(D) \rightarrow [0, \infty]$ by
\begin{equation} \label{WeightedDirichletNormalized}
 \overline{G}(u):= 
 \begin{cases}
 \displaystyle{\int_{D} \left|\nabla \left(\frac{u}{\sqrt{\rho}} \right) \right|^2 \rho^2(x)dx }  &  \te{if }   u \in H^{1}_{\sqrt{\rho}}(D), \\ 
 \infty & \te{if }  u \in L^2(D) \setminus H^{1}_{\sqrt{\rho}}(D),
\end{cases}
\end{equation}
where $H^1_{\sqrt{\rho}}(D)$ is defined in \eqref{H1sqrtrho}. The following holds. 
\begin{theorem}[$\Gamma$-convergence of normalized Dirichlet energies]
With the same setting as in Theorem \ref{ConvergenceSpectrumTheorem} and the same assumptions on $\eta$  and on $\left\{ \veps_n \right\}_{n \in \N}$, $ \overline{G}_{n,\veps_n}$, defined by \eqref{NormalizedDiscreteDirichlet}, $\Gamma$-converge to $\frac{\sigma_\eta}{\beta_\eta  } \overline{G}$ as $n \rightarrow \infty$ in the $TL^2$-sense, where $\overline{G}$ is defined in \eqref{WeightedDirichletNormalized}, $\sigma_\eta$ and $\beta_\eta$ are given by \eqref{sigma_eta} and \eqref{BetaEta} respectively. Moreover, the sequence of functionals $\left\{\overline{G}_{n , \veps_n} \right\}_{n \in \N}$ satisfies the compactness property with respect to the $TL^2$-metric. That is, every sequence $\left\{ u_n \right\}_{n \in \N}$ with $u_n \in L^2(\nu_n)$ for which
$$  \sup_{n \in \N}\| u_n\|_{\nu_n} < \infty, \quad   \sup_{n \in \N } \overline{G}_{n , \veps_n}(u_n) < \infty,  $$
is precompact in $TL^2$.
 \label{DiscreteGammaNormalized}
\end{theorem}

Finally, we consider the limit of the spectrum of $\mathcal{N}^{rw}_{n, \veps_n}$, where $\mathcal{N}^{rw}_{n,\veps_n} = \mathcal{D}^{-1} \mathcal{L}_{n, \veps_n}$. Consider the operator $\mathcal{N}^{rw}$ given by
$$ \mathcal{N}^{rw}(u) = - \frac{1}{\rho^2} \divergence(\rho^2 \nabla u ). $$
As discussed in Subsection \ref{SectionSpectrumL}, the eigenvalues of $\mathcal{N}^{rw}$ are equal to the eigenvalues of $\mathcal{N}^{sym}$. Thus from \eqref{ConditionNsymNrw} and from Theorem \ref{ConvergenceSpectrumTheoremNormalized}, it follows that after appropriate rescaling, the eigenvalues of $\mathcal{N}_{n,\veps_n}^{rw}$ converge to the eigenvalues of $\mathcal{N}^{rw}$. Moreover, using again \eqref{ConditionNsymNrw} and Theorem \ref{ConvergenceSpectrumTheoremNormalized}, we have the following convergence of eigenvectors.
\begin{corollary}
Consider the same setting as in Theorem \ref{ConvergenceSpectrumTheorem} and the same assumptions on $\eta$  and on $\left\{ \veps_n \right\}_{n \in \N}$. Then, with probability one, the following statement holds:
 For every $k \in \N$, every  sequence $\left\{ u_k^n \right\}_{n \in \N}$ with $u_k^n$ being an eigenvector of $\mathcal{N}_{n , \veps_n}^{rw}$ associated to the eigenvalue $\tau_k^{(n)}$ and with $\|u_k^n\|_{\nu_n}=1$ is pre-compact in $TL^2$. Additionally, all its cluster points are eigenfunctions of $ \mathcal{N}^{rw}$ with eigenvalue $\tau_k$. Finally the clusters obtained by Algorithm 3 converge to clusters
obtained by spectral clustering corresponding to $\mathcal N^{rw}$ described at the end of
Subsection \ref{DescContClustU}. 
\label{ConvergenceSpectrumTheoremNormalizedRW}
\end{corollary}

\subsection{Stability of $k$--means clustering} \label{sec:kmean}

One of the final elements of the proof of the consistency results of spectral clustering (statement 4. in Theorems \ref{ConvergenceSpectrumTheorem} and \ref{ConvergenceSpectrumTheoremNormalized}) requires new results on stability of $k$-means clustering with respect to perturbations of the measure being clustered. These results extend the result of Pollard \cite{pollard1981strong} who proved the consistency of $k$-means clustering. It is important to extend such results because in our setting, at the discrete level, the point set used as input for the $k$-means algorithm is not a sample from a given distribution and thus one can not apply the results in \cite{pollard1981strong} directly. 

Given $k \in \N$ and given a measure $\mu$ on $\R^N$ with finite second moments, let $F_{\mu,k} : \R^{N\times k} \to [0, \infty)$ be defined by
\begin{equation} \label{kmeans}
F_{\mu,k}(z_1, \dots, z_k) := \int d(x, \{z_1, \dots, z_k\})^2 d \mu(x) 
\end{equation}
where $z_i \in \R^N$ for $i=1, \dots, k$. For brevity we write
 $\z$ both for $(z_1, \dots, z_k)$ and $\{z_1, \dots, z_k\}$ where the object considered should be clear from the context. 
The problem of $k$-means clustering is to minimize $F_{\mu,k}$ over $\R^{N\times k}$. 
In Subsection \ref{sec:kmeans} we show the existence of minimizers of the functional \eqref{kmeans}. The main result is the following.
\begin{theorem}[Stability of $k$-means clustering]  \label{th:km}
Let $k \geq 1$. Let $\mu$ be a Borel probability measure on $\R^N$ with finite second moments and whose support has at least $k$ points. Assume $\left\{ \mu_m \right\}_{m \in \N}$ is a sequence of probability measures on $\R^N$ with finite second moments which converges in the Wasserstein distance (see \eqref{ot plan}) to  $\mu$. Then, 
$$ \lim_{m \rightarrow \infty}  \min_{\z} F_{\mu_m, k}(\z)  = \min_{\z}F_{\mu,k}(\z). $$
Moreover, if $\z_m$ is a minimizer of $F_{\mu_m,k}$ for all $m$, then the set $\{\z_m, m \in \N\}$ is precompact in $\R^{N\times k}$ and all of its  accumulation points are minimizers of $F_{\mu,k}$. 
\end{theorem}
We present the proof of the previous Theorem in Subsection \ref{sec:kmeans}.

The clusters corresponding to $\z$ minimizing the $F_{\mu,k}$ are the Voronoi cells:  
$G_i =  \{x \in \R^N \::\: d(x, \z) = d(x, z_i)  \}$. 
We prove in Lemma \ref{lem:nomassbdry} that the measure of the boundaries of clusters is zero, that is we show that if $i \neq j$ then $\mu(G_i \cap G_j) = 0$. 
In other words it is irrelevant to which cluster are the points on the boundary assigned to and because of this, we are allowed to define the clusters to be either open or closed sets.  We furthermore note that the associated measures, $\mu_{\llcorner G_i}$ are mutually orthogonal and  satisfy $\sum_i \mu_{\llcorner G_i} = \mu$. 

A consequence of Theorem \ref{th:km} is that as the cluster centers converge so do the measures representing the clusters.

\begin{corollary} \label{cor:km}
Under the assumptions of Theorem \ref{th:km}, if $\z^m$ converge along a subsequence to $\z$ then 
the measures $\mu_{m,\llcorner G_i^m}$ converge weakly in the sense of measures to 
$\mu_{\llcorner G_i}$ as $m \to \infty$ for all $i=1, \dots, k$. 
\end{corollary}
The corollary follows from Theorem \ref{th:km}, since the convergence of centers of Voronoi cells, along with the fact that the boundaries of cells change continuously with respect to cell centers implies that the measures converge in Levy-Prokhorov metric, which characterizes the weak convergence of measures. 
\nc

\subsection{Discussion}\label{Discussion}

Theorems 
\ref{ConvergenceSpectrumTheorem},
\ref{ConvergenceSpectrumTheoremNormalized}, and \ref{ConvergenceSpectrumTheoremNormalizedRW} establish the consistency of spectral clustering.
 An important difference between our work and the available consistency results is that
 we provide an explicit range of rates 
 at which $\veps_n$ (the length scale used to construct the graph) is allowed to converge to $0$ as $n \to \infty$. 
In \cite{vonLuxburg} the parameter $\veps$ is not allowed to depend on $n$. As a result, the functional obtained in the limit is a non-local (i.e. integral, rather than differential) operator. Operators with very different spectral properties are obtained in the limit depending on whether one uses a normalized or unnormalized graph Laplacian. In particular, it is argued that normalized spectral clustering is more advantageous than the unnormalized clustering, because in the normalized case the spectrum of the limiting operator is better behaved and the spectral consistency in the unnormalized case is only guaranteed in restrictive settings. We remark that our results show that when the parameter $\veps_n$ decays to zero such difference between the normalized and the unnormalized settings disappears and the limiting operators in both cases have a discrete spectrum.

When constructing the graph it is advantageous, from the point of view of computational complexity, to 
have fewer edges (that is to take $\veps$ small).
However below some threshold the graph thus constructed does not contain enough information to accurately  recover the geometry of the underlying ground-truth distribution. 
How large $\veps$ should be taken depends on $n$, the number of data points available. 
As number of data points increases $\veps$ converges to zero.
We remark that for $d \geq 3$,  the results of Theorems 
\ref{ConvergenceSpectrumTheorem},
\ref{ConvergenceSpectrumTheoremNormalized}, and \ref{ConvergenceSpectrumTheoremNormalizedRW}
are (almost) optimal in the sense of scaling. Namely,  we show that if the kernel $\eta$ used to construct the graph is compactly supported, then convergence holds if
$\veps_n \gg \frac{\log(n)^{1/d}}{n^{1/d}}$, while if $\veps_n \ll \frac{\log(n)^{1/d}}{n^{1/d}}$
the convergence does not hold. \nc This follows from the results on the connectivity of random geometric graphs in \cite{GuptaKumar,Goel,PenroseBook} which show that with high probability for large  $n$ the graph thus obtained is disconnected.  

Finally, we remark that our results are essentially independent of the kernel used to construct the weights. For example, when the points are sampled from the uniform distribution on a domain $D$, our results show that the spectra of the graph Laplacians converge to the spectrum of the Laplacian on the domain $D$, regardless of the kernel used. 
\subsection{Outline of the approach.} \label{OutlineApproach}

Theorem \ref{ConvergenceSpectrumTheorem} is based on the variational convergence of the energies $G_{n , \veps_n}$ towards $ \sigma_\eta G $, together with the corresponding compactness result (Theorem \ref{DiscreteGamma}). In order to show Theorem \ref{DiscreteGamma}, we first introduce the functional $G_{\veps_n}: L^2(D,\rho) \rightarrow [0,\infty)$  given by,  
\begin{equation}
G_{\veps_n}(u):= \frac{1}{\veps_n^2}\int_{D}\int_{D} \eta_{\veps_n}(x-y)|u(x)-u(y)|^2\rho(x) \rho(y) dxdy,
\label{NonlocalTV}
\end{equation}
which serves as an intermediate object between the functionals $G_{n,\veps_n}$ and $G$. It is important to observe that the argument of $G_{n,\veps_n}$ is a function $u_n$  supported on the data points, whereas the argument of $G_{\veps_n}$ is a $L^2(D,\rho)$ function; in particular a function defined on $D$. The functional $G_{\veps_n}$ is a \textit{non-local} functional, where the term \textit{non-local} refers to the fact that differences of a given function on a $\veps_n$-neighborhood are averaged, which contrasts the local approach of averaging derivatives of the given function. Non-local functionals have been of interest in the last decades due to their wide range of applications which includes phase transitions, image processing and PDEs.    From a statistical point of view, for a fixed function $u:D \rightarrow \R$, $G_{\veps_n}(u)$ is nothing but the expectation of $G_{n, \veps_n}(u)$. On the other hand, the functional $G_{\veps_n}$ is relevant for our purposes because not only it approximates $G$ defined in \eqref{WeightedDirichlet} in a pointwise sense,  but it also approximates it in a variational sense (as the parameter $\veps_n$ goes to zero). More precisely the following holds.

\begin{proposition}
Consider an open, bounded domain  $D$  in $\mathbb{R}^d$ with Lipschitz boundary. 
Let $\rho: D \to \R$ be continuous and bounded below and above by positive constants. Le $\left\{ \veps_k \right\}_{k \in \N}$ be a sequence of positive numbers converging to zero. Then, $\left\{ G_{\veps_k} \right\}_{k\in \N}$ (defined in \eqref{NonlocalTV}) $\Gamma$-converges with respect to the $L^2(D,\rho)$-metric to $\sigma_{\eta} G$, where $\sigma_\eta$ is defined in \eqref{sigma_eta} and $G$ is defined in \eqref{WeightedDirichlet}. Moreover, the functionals $\left\{ G_{\veps_k} \right\}_{k\in \N}$ satisfy the compactness property, with respect to the $L^2(D,\rho)$-metric. That is, every sequence $\left\{ u_k \right\}_{n \in \N}$ with $u_k \in L^2(D,\rho)$ for which
$$  \sup_{k \in \N}\| u_k\|_{L^2(D,\rho)} < \infty, \quad   \sup_{k \in \N } G_{\veps_k}(u_n) < \infty,  $$
is precompact in $L^2(D,\rho)$. 
Finally, for every $u \in L^2(D,\rho)$ 
\begin{equation}
\lim_{n \rightarrow \infty}G_{\veps_k}(u) = \sigma_\eta G(u).
\label{PointwiseNonlocal}
\end{equation} 
\label{NonlocalDirichlet}
\end{proposition}
\begin{proof}
When $\rho$ is constant, the proof may be found in the Appendix of \cite{AB2} in case $D$ is a convex set, and in \cite{Ponce} for a general domain $D$ satisfying the assumptions in the statement. In case $\rho$ is not constant the results are obtained in a  straightforward way by adapting the arguments presented in \cite{Ponce} just as it is done in Section 4 in \cite{GTS} when studying the variational limit of the non-local functional
$$ TV_{\veps}(u):= \frac{1}{\veps_n}\int_{D}\int_{D} \eta_{\veps}(x-y)|u(x)-u(y)|\rho(x) \rho(y) dxdy,  $$
which is the $L^1$ analogue of $G_{\veps}$.
\end{proof}

As observed earlier, the argument of  $G_{n , \veps_n}$ is a function $u_n$ supported on the data points, while the argument of $G_{\veps_n}$ is an $L^2(D)$ function. For a function $u_n$ defined on the set $V= \left\{ \x_1, \dots, \x_n \right\}$, the idea is to associate an $L^2(D)$ function $\tilde{u}_n$ which approximates $u_n$ in the $TL^2$-sense and is such that $G_{\veps_n}(\tilde{u}_n)$ is comparable to $G_{n , \veps_n}(u_n)$. The purpose of doing this is to use Proposition  \ref{NonlocalDirichlet}. We construct the approximating function $\tilde{u}_n$ by using transportation maps (i.e. measure preserving maps) between the measure $\nu$ and $\nu_n$. More precisely, we set $\tilde u_n = u_n \circ T_n$ where $T_n$ is a transportation map between $\nu$ and $\nu_n$ which moves mass as little as possible. The estimates on how far the mass needs to be moved were known in the literature when $\rho$ is constant and when the domain $D$ is the unit cube $(0,1)^d$ (see \cite{LeightonShor,Talagrand,TalagrandGenericChain,TalagrandNewBook} for $d=2$ and \cite{ShorYukich} for $d\geq 3$). In \cite{W8L8} these estimates are extended to general domains $D$ and densities $\rho$ satisfying \eqref{DensityBound}. Indeed, the following is proved.

\begin{proposition}
Let $D\subseteq \R^d$ be a bounded, connected, open set with Lipschitz boundary. Let $\nu$ be a probability measure on $D$ with density $\rho: D \rightarrow (0,\infty)$ satisfying \eqref{DensityBound}. Let $\x_1, \dots, \x_n, \dots$ be i.i.d. samples from $\nu$. Let $\nu_n$ be the empirical measure associated to the $n$ data points.
Then, for any fixed $\alpha >2$, except on a set with probability $O(n^{-\alpha/2})$, there exists a transportation map $T_n: D \rightarrow D$ between the measure $\nu$ and the measure  $\nu_n$ (denoted $T_{n \sharp} \nu = \nu_n$) such that
$$ \| T_n - Id \|_{\infty} \leq C
\begin{cases}
 \frac{\ln(n)^{3/4}}{n^{1/2}},  & \te{if } d =2, \medskip \\
 \frac{\ln(n)^{1/d}}{n^{1/d}},  & \te{if } d \geq 3,
\end{cases}
$$
where $C$ depends only on $\alpha $, $D$, and the constants $m,M$.
\label{TheoremMatchings}
\end{proposition}
From the previous result, Chebyshev's inequality and Borel-Cantelli lemma one obtains the following rate of convergence of the $\infty$-\textit{transportation} distance between the empirical measures $\nu_n$ and the measure $\nu$ (see \cite{W8L8} for details associated to Proposition \ref{TheoremMatchings}).
\begin{proposition} \label{thm:InifinityTransportEstimate}
 Let $D$ be an open, connected and bounded subset of $\R^d$ which has Lipschitz boundary. Let $\nu$ be a probability measure on $D$ with density $\rho$ satisfying \eqref{DensityBound}. Let  $\x_1, \dots, \x_n, \dots$ be a sequence of independent samples from $\nu$ and let $\nu_n$ be the associated empirical measures \eqref{empirical}. 
Then, there is a constant $C>0$ such that with probability one, there exists a sequence of transportation maps $\left\{ T_n \right\}_{n \in \N}$ from $\nu$ to $\nu_n$   ($T_{n \sharp} \nu = \nu_n$) and  such that: 
\begin{alignat}{2}
\te{if } d&=2 \te{ then}\qquad  & \limsup_{n \rightarrow \infty}   \frac{n^{1/2} \|Id - T_n\|_\infty }{(\log n)^{3/4}} & \leq C\label{InifinityTransportEstimate d=2}\\
\te{and\ if } d & \geq 3 \te{ then} & \limsup_{n \rightarrow \infty}  \frac{n^{1/d} \|Id - T_n\|_\infty }{(\log n)^{1/d}} & \leq C.
\label{InifinityTransportEstimate d>2}
\end{alignat}
\end{proposition}
As shown in Section \ref{ProofsMainTheorems}, Proposition \ref{NonlocalDirichlet} and Proposition \ref{thm:InifinityTransportEstimate} are at the backbone of Theorem \ref{DiscreteGamma}. Schematically, 
$$  G_{\veps} \overset{\Gamma}{\longrightarrow}  \sigma_\eta G \quad \text{ in }L^2 \quad  + \quad  \text{Proposition } \ref{thm:InifinityTransportEstimate} \Longrightarrow  G_{n, \veps_n} \overset{\Gamma}{\longrightarrow}  \sigma_\eta G \quad  \text{ in } TL^2.$$
We note that the statement $ G_{\veps} \overset{\Gamma}{\longrightarrow}  \sigma_\eta G$ is a purely analytic, purely deterministic fact. Proposition \ref{thm:InifinityTransportEstimate}, on the other hand contains all the probabilistic estimates needed to establish all the results on this paper. Such estimates in particular provide the constraints on the parameter $\veps_n$ in Theorem \ref{DiscreteGamma}. It is worth observing that Proposition \ref{thm:InifinityTransportEstimate} is a statement that only involves the underlying measure $\nu$ and the empirical measure $\nu_n$, and that in particular it does not involve estimates on the difference between the functional $G_{\veps_n}(u)$ and the functional $G_{n, \veps_n}(u)$ for $u$ belonging to a small (in the sense of $VC$-dimension) class of functions. In other words our estimates are related to the domains where the functions are defined (discrete/continuous) and not to the actual values of functions defined on those domains.

\medskip

With Theorem \ref{DiscreteGamma} at hand, the proof of Theorem \ref{ConvergenceSpectrumTheorem} now relies on some spectral properties of the operator $\mathcal{L}$ and analogous properties of $\mathcal{L}_{n , \veps_n}$. As shown in Section \ref{SectionSpectrumL}, the space $L^2(D,\rho)$ has a countable orthonormal basis (with respect to the inner product $\langle\cdot, \cdot  \rangle_{\rho}$) formed with eigenfunctions of $\mathcal{L}$. Additionally, the different eigenvalues of $\mathcal{L}$ can be organized as an increasing sequence of positive numbers converging to infinity. Each of the eigenvalues has finite multiplicity. Moreover, the eigenvalues of $\mathcal{L}$ have a variational characterization, as they can be written as the minimum value of optimization problems over successive subspaces of $L^2(D,\rho)$. This is the content of the Courant-Fisher mini-max principle which states that for every $k$
\begin{equation}
\lambda_{k} = \sup_{ S \in \Sigma_{k-1}} \min_{\|u\|_{\rho}=1 \:, \: u \in S^\perp}  G(u),
\label{MinimaxUnnormalizedCont}
\end{equation}
where we recall $0 = \lambda_1 \leq \lambda_2 \leq \dots,$ denote the eigenvalues of $\mathcal{L}$ repeated according to multiplicity, $\Sigma_{k-1}$ denotes the set of $(k-1)$-dimensional subspaces of $L^2(D,\rho)$, and where $S^\perp$ represents the orthogonal complement of $S$ with respect to the inner product $\langle \cdot, \cdot \rangle_{\rho}$. Moreover, the supremum in \eqref{MinimaxUnnormalizedCont} is attained by the span of the first $(k-1)$ eigenfunctions of $\mathcal{L}$. In Subsection \ref{SectionSpectrumL} we review the previously mentioned spectral properties of $\mathcal{L}$. Likewise, we can write the eigenvalues of $\mathcal{L}_{n, \veps_n}$ as
\begin{equation}
\lambda_{k}^{(n)} = \frac{n\veps_n^2}{2} \sup_{ S \in \Sigma^{(n)}_{k-1}} \min_{\|u\|_{\nu_n}=1 \:, \: u \in S^\perp} G_{n, \veps_n}(u),
\label{MinimaxUnnormalizedDiscrete}
\end{equation}
where $\Sigma_{k-1}^{(n)}$ denotes the set of $(k-1)$-dimensional subspaces of $\R^n$, and where $S^\perp$ represents the orthogonal complement of $S$ with respect to the inner product $\langle\cdot, \cdot \rangle_{\nu_n}$ in $\R^n$. Moreover, as in the continuum setting, the supremum in \eqref{MinimaxUnnormalizedDiscrete} is attained by the span of the first $(k-1)$ eigenvectors of $\mathcal{L}_{n, \veps_n}$. Theorem \ref{DiscreteGamma} allows us to exploit expressions \eqref{MinimaxUnnormalizedDiscrete} and \eqref{MinimaxUnnormalizedCont} and in fact in Section \ref{ProofsMainTheorems} we show how \ref{MinimaxUnnormalizedDiscrete} and \ref{MinimaxUnnormalizedCont} together with Theorem \ref{DiscreteGamma} imply Theorem \ref{ConvergenceSpectrumTheorem}, thus establishing the spectral convergence in the unnormalized case.

In the normalized case, the same approach used in the unnormalized case can be taken. In fact, the proof of Theorem \ref{ConvergenceSpectrumTheoremNormalized} follows from the proof of Theorem \ref{ConvergenceSpectrumTheorem} by \emph{mutatis mutandis} after Theorem \ref{DiscreteGammaNormalized} has been proved.

The paper is organized as follows. Section \ref{Notation} contains the notation and the background we need in the rest of the paper. In particular in Subsection \ref{TT} we review some facts about the $TL^2$ space, in Subsection \ref{sec:gamma} we review the definition of $\Gamma$-convergence, in Subsection \ref{sec:kmeans} we present some results on stability of $k$-means clustering, and in Subsection \ref{SectionSpectrumL} some facts about the spectrum of the operators $\mathcal{L}$, $\mathcal{N}^{sym}$ and $\mathcal{N}^{rw}$. In Section \ref{ProofsMainTheorems} we prove Theorem \ref{DiscreteGamma} and Theorem \ref{ConvergenceSpectrumTheorem}. Finally, in Section \ref{SectionGammaNormalized} we prove Theorem \ref{DiscreteGammaNormalized}, Theorem \ref{ConvergenceSpectrumTheoremNormalized} and Corollary \ref{ConvergenceSpectrumTheoremNormalizedRW}. 
 
\medskip

\section{Preliminaries}
\label{Notation}
\subsection{Transportation theory and the $TL^2$ space.} \label{TT}
Let $D$ be an open domain in $\R^d$.  We denote by $\mathfrak{B}(D)$ the Borel $\sigma$-algebra of $D$, by $\mathcal{P}(D)$ the set of all Borel probability measures on $D$ and by $\mathcal{P}_2(D)$ the Borel probability measures on $D$ with finite second moments. The Wasserstein distance between $\mu, \tilde{\mu} \in \mathcal{P}_2(D)$ (denoted by $d_2(\mu, \tilde{\mu})$)  is defined by:
\begin{equation} \label{ot plan}
d_2(\mu, \tilde{\mu}):= \min\left\{  \left( \int_{D \times D}|x-y|^2 d \pi (x,y) \right)^{1/2}\: : \: \pi \in \Gamma(\mu, \tilde{\mu})  \right\},
\end{equation}
where $\Gamma(\mu, \tilde{\mu})$ is the set of all {\em couplings} between $\mu$ and $\tilde{\mu}$, that is, the set of all  Borel probability measures on $D \times D$ for which the marginal on the first variable is $\mu$ and the marginal on the second variable is $\tilde{\mu}$.  The elements $\pi \in \Gamma(\mu, \tilde{\mu})$ are also referred as \emph{transportation plans} between $\mu$ and $\tilde{\mu}$. The existence of minimizers, which justifies the definition above, is straightforward to show, see \cite{villani2003topics}. It is known that  the convergence in Wasserstein metric is equivalent to weak convergence of probability measures and uniform integrability of second moments.

In the remainder, unless otherwise stated, we assume that $D$ is a bounded set. In that setting, we have $\mathcal{P}(D)= \mathcal{P}_2(D)$ and  uniform integrability of second moments is immediate. In particular, convergence in  the Wasserstein metric is equivalent to weak convergence of measures.  For details see for instance \cite{villani2003topics}, \cite{ambrosio2008gradient} and the references within. In particular, $\mu_n \rightharpoonup \mu $ (to be read $\mu_n $ converges weakly to $\mu$) if and only if there is a sequence of transportation plans between $\mu_n$ and $\mu$, $\left\{ \pi_n \right\}_{n \in \N}$, for which:
%We also consider,
%\begin{equation} \label{inf ot}
%d_\infty(\mu, \tilde{\mu}):= \inf \left\{  \esssup_\pi \{ |x-y| \::\: (x,y) \in D \times D \}  \: : \: \pi \in \Gamma(\mu, \tilde{\mu})  \right\},
%\end{equation}
%which defines a metric on $\mathcal P(D)$. We refer to this metric as the $\infty$-transportation distance.
% paragraph{}
%One could consider the previous notions in more generality, for example by working on a separable metric space with the Radon property, that is, spaces where every Borel probability measure has the inner regularity property.
\begin{equation}
 \lim_{n \rightarrow \infty} \iint_{D \times D} | x- y |^2 d \pi_n(x,y) =0. 
\label{convergentPlans}
\end{equation}
Actually, note that if $D$ is bounded, \eqref{convergentPlans} is equivalent to $ \lim_{n \rightarrow \infty} \iint_{D \times D} | x- y | d \pi_n(x,y) =0$. 
We say that a sequence of transportation plans, $\left\{ \pi_n \right\}_{n \in \N}$  (with $\pi_{n} \in \Gamma(\mu, \mu_n)$),  is \emph{stagnating} if it satisfies \eqref{convergentPlans}.
%We remark that, since $D$ is bounded, it is straightforward to show that a sequence of transportation plans is stagnating if and only if 
%$\pi_n$ converges weakly in the space of probability measures on $D \times D$ to $\pi = (id \times id)_\sharp \mu$. 
Given a Borel map $T: D \rightarrow D$ and $\mu \in \mathcal{P}(D)$, the \emph{push-forward} of $\mu$ by $T$, denoted by $T_{\sharp} \mu \in \mathcal{P}(D)$ is given by:
\begin{equation*}
T_{\sharp} \mu(A):= \mu\left( T^{-1}(A) \right), \: A \in \mathfrak{B}(D).
\end{equation*}
For any bounded Borel function $\varphi: D \rightarrow \R$ the following change of variables in the integral holds:
\begin{equation} \label{chofvar}
 \int_D \varphi(x) \: d (T_\sharp \mu)(x) = \int_D \varphi(T(x)) \, d \mu(x).
\end{equation}

We say that a Borel map $T : D \rightarrow D$ is a \emph{transportation map} between the measures $\mu\in \mathcal{P}(D)$ and $\tilde{\mu} \in \mathcal{P}(D)$ if $\tilde{\mu} = T_\sharp \mu$. In this case, we associate a transportation plan $\pi_T \in \Gamma(\mu , \tilde{\mu})$ to $T$ by:
\begin{equation}
\pi_T:= (\id \times T)_{\sharp}\mu,
\label{TransportationMapPlan1}
\end{equation}
where $(\id \times T): D \rightarrow D \times D$  is given by $(\id \times T)(x) = \left(x, T(x) \right)$. 
%Then for any $c \in L^1(D\times D , \mathfrak{B}\left(D \times D \right),\pi_T)$
%\begin{equation}
%\int_{D \times D} c(x,y) d \pi_T(x,y) = \int_{D}c\left(x, T(x)\right) d \mu(x).
%\label{TransportationMapPlan}
%\end{equation}

It is well known that when the measure $\mu \in \mathcal{P}_2(D)$ is absolutely continuous with respect to the Lebesgue measure, the problem on the right hand side of \eqref{ot plan} is equivalent to:
\begin{equation}
 \min \left\{  \left( \int_{D }|x-T(x)|^2 d \mu(x) \right)^{1/2}\: : \:  T_{\sharp}\mu= \tilde{\mu}  \right\}.
 \label{ot map}
\end{equation}
In fact, the problem \eqref{ot plan} has a unique solution which is induced (via \eqref{TransportationMapPlan1}) by a transportation map $T$ solving \eqref{ot map} (see \cite{villani2003topics}).
In particular, boundedness of $D$ implies that when $\mu$ has a density, then $\mu_n \rightharpoonup{\mu}$ as $n \rightarrow \infty$ is equivalent to the existence of a sequence $\left\{ T_n\right\}_{n \in \N}$ of transportation maps, (${T_n}_{\sharp} \mu= \mu_n$)  such that:
\begin{equation}
\int_{D} | x- T_n(x) |^2 d \mu(x) \rightarrow 0, \: \te{ as } \: n \rightarrow \infty.
\label{convergentMaps}
\end{equation}
We say that a sequence of transportation maps $\left\{ T_n \right\}_{n \in \N}$ is \emph{stagnating} if it satisfies  \eqref{convergentMaps}.
\medskip

We now introduce the space of objects that allows to simultaneously consider the discrete and continuum setting. Let
\[ TL^2(D) := \{ (\mu, f) \; : \:  \mu \in \mathcal P_2(D), \, f \in L^2(\mu) \}, \]
where $L^2(\mu)$ denotes the space of $L^2$ functions with respect to measure $\mu$. 
%If $\mu$ is absolutely continuous with respect to the Lebesgue measure, that is if $d \mu = \rho(x) dx$, we write $L^2(\rho)$ to represent $L^2(\mu)$ and if $\mu$ is the Lebesgue measure restricted to $D$ we simply write $L^2(D)$.  
For $(\mu,f)$, $(\nu,g)$ in $TL^2$ define
  \begin{equation} \label{tlpmetric}
d_{TL^2}((\mu,f), (\nu,g)) =
   \inf_{\pi \in \Gamma(\mu, \nu)} \left(\iint_{D \times D} |x-y|^2  + |f(x)-g(y)|^2  d\pi(x,y) \right)^{\frac{1}{2}}.
\end{equation}
The set $TL^2$ and $d_{TL^2}$ were introduced in \cite{GTS}, where it was also proved that $d_{TL^2}$ is a metric. Note that if we delete the second term on the right hand side of \eqref{tlpmetric} we recover the Wasserstein distance between the measures $\mu$ and $\nu$. The idea of introducing the second term on the right hand side of \eqref{tlpmetric} is to make it possible to compare functions in spaces as different as point clouds and continuous domains. We have the following characterization of convergence in $TL^2$. See \cite{GTS}[Propositions 3.3 and 3.12] for its proof.
\begin{proposition} 
Let $(\mu,f) \in TL^2$ and let
 $\left\{ \left(\mu_n , f_n \right) \right\}_{n \in \N}$ be a sequence in $TL^2$. The following statements are equivalent:
\begin{list1}
\item  $ \left(\mu_n , f_n \right)  \overset{{TL^2}}{\longrightarrow} (\mu, f)$ as $n \rightarrow \infty$.
\item The graphs of functions considered as measures converge in  the Wasserstein sense \eqref{ot plan}, that is
\[ (I \times f_n)_\sharp \mu_n \overset{d_2}{\longrightarrow} (I \times f)_\sharp \mu  \quad \te{ as } n \to \infty. \]
\item $\mu_n \rightharpoonup \mu$ and for every stagnating sequence of transportation plans $\left\{\pi_n \right\}_{n \in \N}$
(with $\pi_n \in \Gamma(\mu, \mu_n)$) 
\begin{equation}
\iint_{D \times D} \left| f(x) - f_n(y) \right|^2 d\pi_n(x,y) \rightarrow 0, \:  as \: n \rightarrow \infty.
\label{convergentTLp}
\end{equation} 
\item $\mu_n \rightharpoonup \mu$ and there exists a stagnating sequence of transportation plans $\left\{\pi_n \right\}_{n \in \N}$ (with $\pi_n \in \Gamma(\mu, \mu_n)$) for which \eqref{convergentTLp} holds.
% paragraph{}
\end{list1}
Moreover, if the measure $\mu$ is absolutely continuous with respect to the Lebesgue measure, the following are  equivalent to the previous statements:
\begin{list1}
\addtocounter{broj1}{3}
\item  $\mu_n \rightharpoonup \mu$ and there exists a stagnating sequence of transportation maps $\left\{ T_n \right\}_{n \in \N}$ (with ${T_n}_\sharp \mu = \mu_n$) such that:
\begin{equation}
\int_{D} \left| f(x) - f_n\left(T_n(x)\right) \right|^2 d\mu(x) \rightarrow 0, \:  as \: n \rightarrow \infty.
\label{convergentTLpMap}
\end{equation}
\item $\mu_n \rightharpoonup \mu$ and for any stagnating sequence of transportation maps $\left\{ T_n \right\}_{n \in \N}$  (with ${T_n}_\sharp \mu = \mu_n$)  \eqref{convergentTLpMap} holds.
\end{list1}
\label{EquivalenceTLp}
\end{proposition}

\begin{remark}
One can think of the convergence in $TL^2$ as a generalization of weak convergence of measures and convergence in $L^2$ of functions. By this we mean that $\left\{ \mu_n \right\}_{n \in \N}$ in $\mathcal{P}(D)$ converges weakly (and in the Wasserstein sense) to $\mu\in \mathcal{P}(D)$ if and only if $ \left(\mu_n , 1 \right) \overset{{TL^2}}{\longrightarrow} (\mu, 1)$ as $n \rightarrow \infty$,  and that for $\mu \in \mathcal{P}(D)$ a sequence $\left\{ f_n \right\}_{n\in \N}$ in $L^2(\mu)$ converges in $L^2(\mu)$ to $f$ if and only if $(\mu, f_n) \overset{{TL^2}}{\longrightarrow} (\mu,f)$ as $n \rightarrow \infty$. The last fact is established
in Proposition \ref{EquivalenceTLp}.
\end{remark}

\begin{definition}
Suppose $\left\{ \mu_n  \right\}_{n \in \N} $ in $\mathcal{P}(D)$ converges weakly to $\mu \in \mathcal{P}(D)$. We say that the sequence $\left\{ u_n\right\}_{n \in \N}$ (with $u_n \in L^2(\mu_n)$) converges in the $TL^2$-sense to $u \in L^2(\mu)$, if $\left\{  \left(\mu_n , u_n  \right) \right\}_{n \in \N}$ converges to $(\mu, u)$ in the $TL^2$-metric. In this case we use a slight abuse of notation and write $u_n \overset{{TL^2}}{\longrightarrow} u$ as $n \rightarrow \infty$. Also, we say the sequence $\left\{ u_n\right\}_{n \in \N}$ (with $u_n \in L^2(\mu_n)$) is precompact in $TL^2$ if the sequence $\left\{   \left(\mu_n , u_n  \right)\right\}_{n \in \N}$ is precompact in $TL^2$.
\label{DefintionTL2Functions}
\end{definition}

\begin{remark}
Thanks to Proposition \ref{EquivalenceTLp} when $\mu$ is absolutely continuous with respect to the Lebesgue measure, $u_n \overset{{TL^2}}{\longrightarrow} u$ as $n \rightarrow \infty$ if and only if for every (or one) $\left\{ T_n\right\}_{n \in \N}$ stagnating  sequence  of transportation maps (with $T_{n \sharp} \mu= \mu_n$) it is true that $u_n \circ T_n \overset{{L^2(\mu)}}{\longrightarrow} u $ as $n \rightarrow \infty$. Also $\left\{ u_n  \right\}_{n \in \N} $ is precompact in $TL^2$ if and only if  for every (or one) $\left\{ T_n\right\}_{n \in \N}$ stagnating sequence  of transportation maps (with $T_{n \sharp} \mu= \mu_n$) it is true that  $\left\{ u_n \circ T_n \right\}_{n \in \N} $ is pre-compact in $L^2(\mu)$.
\label{RemarkDefinitionTL2Functions}
\end{remark}

\begin{lemma} \label{lem:tl_char}
Let $\mu_n$ be a sequence of  Borel probability measures on $\R^N$ with finite second moments, converging to a probability measure $\mu$ in the Wasserstein sense. 
Let $A_n$ be $\mu_n$ measurable, and $A$ be $\mu$ measurable. 
Let $\tilde{\mu}_n = \mu_{n \llcorner A_n}$ and $\tilde{\mu}= \mu_{\llcorner A}  $. 
Then,
\begin{equation}
(\mu_n, \chi_{A_n}) \overset{{TL^2}}{\longrightarrow} (\mu, \chi_A) \quad \te{ if and only if } \quad
\tilde{\mu}_n \rightharpoonup \tilde{\mu}
\end{equation}
as $n \to \infty$.
\end{lemma}
\begin{proof}
From Proposition \ref{EquivalenceTLp} follows that $(\mu_n, \chi_{A_n}) \overset{{TL^2}}{\longrightarrow} (\mu, \chi_A)$ if and only if  $\tilde{\mu}_n \times \delta_{\{1\}} + (\mu_n - \tilde{\mu}_n) \times \delta_{ \{0\}} \overset{d_2}{\longrightarrow}
\tilde{\mu} \times \delta_{\{1\}} + (\mu - \tilde{\mu}) \times \delta_{\{0\}}$, as $n \rightarrow \infty$.

Since convergence in Wasserstein distance implies weak convergence, we deduce that $\tilde{\mu}_n \times \delta_{\{1\}} + (\mu_n - \tilde{\mu}_n) \times \delta_{ \{0\}} \rightharpoonup \tilde{\mu} \times \delta_{\{1\}} + (\mu - \tilde{\mu}) \times \delta_{\{0\}}$, and in particular we conclude that
$$ \tilde{\mu}_n \rightharpoonup \tilde{\mu}, \quad \text{as } n \rightarrow \infty. $$
 
Conversely, the weak convergence $\tilde{\mu}_n \rightharpoonup \tilde{\mu},$ together with the fact that $\mu_n \overset{d_2}{\longrightarrow} \mu$ (which in particular implies that $\mu_n \rightharpoonup \mu$), imply that
$$\tilde{\mu}_n \times \delta_{\{1\}} + (\mu_n - \tilde{\mu}_n) \times \delta_{ \{0\}} \rightharpoonup \tilde{\mu} \times \delta_{\{1\}} + (\mu - \tilde{\mu}) \times \delta_{\{0\}}.$$

In order to conclude that the above convergence also holds in the Wasserstein sense, we simply note that this follows from the  the uniform integrability of the second moments of $\left\{ \tilde{\mu}_n \right\}_{n\in \N}$, which in turn follows from

$$\lim_{t \rightarrow \infty} \sup_{n \in \N} \int_{ \left\{ |x| > t \right\}} |x|^2 d \tilde{\mu}_n(x) \leq  \lim_{t \rightarrow \infty} \sup_{n \in \N} \int_{\left\{ |x| > t \right\}} |x|^2 d \mu_n(x)  =0.$$
The equality in the previous expression follows from the fact that $\mu_n \overset{d_2}{\longrightarrow} \mu$.
\end{proof}

%\red
%Note that the statement above did not change... I guess  one could  formulate $TL^2$ convergence of functions in terms of weak convergence of  level sets...
%\nc

The following proposition states that inner products are continuous with respect to the $TL^2$-convergence.

\begin{proposition}
Suppose that $(\mu_n,u_n) \overset{{TL^2}}{\longrightarrow} (\mu,u) $ and $( \mu_n, v_n ) \overset{{TL^2}}{\longrightarrow} (\mu,v)$ as $n \rightarrow \infty$. Then, 

\begin{equation}
 \lim_{n \rightarrow \infty} \langle u_n , v_n \rangle_{\mu_n} = \langle u,v \rangle_{\mu}.
\end{equation}
\label{ContinuityInnerProdTL2}
\end{proposition}
\begin{proof}
By the polarization identity, it is enough to prove that if $( \mu_n,u_n ) \overset{{TL^2}}{\longrightarrow} (\mu,u) $ then,
\begin{equation}
 \lim_{n \rightarrow \infty}  \| u_n \|_{\mu_n} = \| u \|_{\mu}.
\end{equation}
For this purpose, consider a stagnating sequence of transportation plans $ \left\{ \pi_n \right\}_{n \in \N}$ with $\pi_n \in \Gamma(\mu, \mu_n)$.  We can write $\| u_n \|_{\mu_n} = \left(\int_{D\times D}|u_n(y) |^2 d\pi_n(x,y) \right)^{1/2}$ and $\|u\|_{\mu} = \left(\int_{D\times D}|u(x) |^2 d\pi_n(x,y) \right)^{1/2}$. Hence,
\begin{equation}
\left| \: \: \|u_n\|_{\mu_n}  - \|u\|_{\mu} \right| \leq \left(\int_{D \times D}|u_n(y)-u(x)|^2 d \pi_n(x,y) \right)^{1/2} \rightarrow 0 \:, \text{ as } \: n \rightarrow \infty.
\end{equation}
\end{proof}

In proving the convergence of $k$-means clustering (statement 4. in Theorems \ref{ConvergenceSpectrumTheorem} and \ref{ConvergenceSpectrumTheoremNormalized})) we also need the following result on $TL^2$ convergence of a composition of functions.

\begin{lemma}[Continuity of composition in $TL^2$] \label{TL2comp}
Let $\left\{ \mu_n \right\}_{n\in \N}$ and $\mu$ be a collection of Borel probability measures on $\R^d$ with finite second moments. Let  $F_n \in L^2(\mu_n, \R^d: \R^k)$  for all $n \in \N$ and  $F \in L^2(\mu,\R^d:\R^k)$. Consider the measures $\tilde{\mu}_n:= {F_{n }}_{\sharp} \mu_n$ for all $n \in \N$ and $\tilde{\mu}:={F}_{\sharp}\mu$. Finally, let $\tilde{f}_n \in L^2(\tilde \mu_n,\R^k: \R)$ for all $n \in \N$ and $\tilde{f} \in L^2(\tilde \mu,\R^k: \R)$. If 
$$ (\mu_n, F_n) \overset{{TL^2}}{\longrightarrow} (\mu, F)  \quad \text{as } n \rightarrow \infty, $$
and
$$  (\tilde{\mu}_n , \tilde{f}_n ) \overset{{TL^2}}{\longrightarrow} (\tilde{\mu}, \tilde{f}) \quad \text{as } n \rightarrow \infty. $$
Then,
$$ (\mu_n,  \tilde{f}_n\circ F_n   ) \overset{{TL^2}}{\longrightarrow} (\mu, \tilde{f}\circ F_n) \quad \text{as } n \rightarrow \infty. $$
\end{lemma}
\begin{proof}
First of all note that the fact that $F_n \in L^2(\mu_n, \R^d: \R^k) $ and $F \in L^2(\mu,\R^d:\R^k)$ guarantees that $\tilde{\mu}_n$ and $\tilde{\mu}$ are probability measures on $\R^k$ with finite second moments. On the other hand, $(\mu_n, F_n) \overset{{TL^2}}{\longrightarrow} (\mu, F)$ as $n \rightarrow \infty$ implies the existence of a stagnating sequence of transportation maps $\left\{  \pi_n\right\}_{n \in \N}$ with $\pi_n \in \Gamma(\mu, \mu_n)$ such that
\begin{equation}
\lim_{n \rightarrow \infty} \int_{\R^d \times \R^d} | F(x)-F_n(y) |^2 d \pi_n(x,y) =0.
\label{auxTL2comp}
\end{equation}
We consider the measures $\tilde{\pi}_n:= (F \times F_n )_{\sharp} \pi_n$ for all $n \in \N$. It is straightforward to check that $\tilde{\pi}_n \in \Gamma(\tilde{\mu}, \tilde{\mu}_n)$ for all $n \in \N$ and by the definition of $\tilde{\pi}_n$ that
$$ 
\lim_{n \rightarrow \infty}  \int_{\R^k \times \R^k} | \tilde{x} - \tilde{y}|^2 d \tilde{\pi}_{n}(\tilde{x}, \tilde{y}) = \lim_{n \rightarrow \infty} \int_{\R^d \times \R^d} | F(x)-F_n(y) |^2 d \pi_n(x,y) =0$$
In other words, $\left\{ \tilde{\pi}_n \right\}_{n \in \N}$ is a stagnating sequence of transportation maps with $\tilde{\pi}_n \in \Gamma(\tilde{\mu}, \tilde{\mu}_n )$. From the fact that $  (\tilde{\mu}_n , \tilde{f}_n ) \overset{{TL^2}}{\longrightarrow} (\tilde{\mu}, \tilde{f}) \quad \text{as } n \rightarrow \infty$ and from Proposition \ref{EquivalenceTLp} it follows that
$$ \lim_{n \rightarrow \infty}  \int_{\R^k \times \R^k}| \tilde{f}(\tilde{x}) - \tilde{f}_n(\tilde{y})|^2 d \tilde{\pi}_n(\tilde{x}, \tilde{y}) =0.  $$
But again by the definition of $\tilde{\pi}_n$ we deduce
$$\lim_{n \rightarrow \infty}  \int_{\R^d \times \R^d}| \tilde{f}(F(x)) - \tilde{f}_n(F_n(y))|^2 d \pi_n(x, y)    = \lim_{n \rightarrow \infty}  \int_{\R^k \times \R^k}| \tilde{f}(\tilde{x}) - \tilde{f}_n(\tilde{y})|^2 d \tilde{\pi}_n(\tilde{x}, \tilde{y}) =0   $$
Using again Proposition \ref{EquivalenceTLp} we obtain the desired result.
\end{proof}

%%%%%%%%%%%%%%%%%%%%%%%%%%%%%
\subsection{$\Gamma$-convergence.} \label{sec:gamma}
We recall the notion of $\Gamma$-convergence in general setting. 
\begin{definition}
Let $(X,d_X)$ be a metric space and let $(\Omega , \mathcal{F}, \mathbb{P})$ be a probability space. Let $\left\{ F_n\right\}_{n \in \N}$ be a sequence of (random) functionals $F_n: X \times \Omega \rightarrow [0, \infty] $ and let $F$ be a (deterministic) functional $F : X \rightarrow [0, \infty]$. We say that the sequence of functionals $\left\{ F_n \right\}_{n \in \N}$ $\Gamma$-converges (in the $d_X$ metric) to $F$, if for almost every $\omega \in \Omega$, \textbf{all} of the following conditions hold:
\begin{enumerate}
\item \textbf{Liminf inequality:} For all $x \in X$ and all sequences $\left\{ x_n\right\}_{n \in \N}$ converging to $x$ in the metric $d_X$ it is true that
$$  \liminf_{n \rightarrow \infty} F_n(x) \geq F(x)   $$
\item \textbf{Limsup inequality:} For all $x \in X$ there exists a sequence $\left\{ x_n\right\}_{n \in \N}$ converging to $x$ in the metric $d_X$ such that
$$  \limsup_{n \rightarrow \infty} F_n(x) \leq F(x)   $$
\end{enumerate}
\end{definition}
The notion of $\Gamma$-convergence is particularly useful when combined with an appropriate notion of compactness. See \cite{braides2002gamma,DalMaso}.
\begin{definition}
We say that  the sequence of nonnegative random functionals $\left\{ F_n\right\}_{n \in \mathbb{N}}$ satisfies the compactness property if for almost every $\omega \in \Omega$, it is true that every bounded (with respect to $d_X$) sequence $\left\{x_n\right\}_{n \in \N}$ in $X$ for which
\begin{equation*}
\sup_{n \in \N} F_{n}(x) < \infty,
\end{equation*} 
is precompact in $X$.
\label{defCompac}
\end{definition}
Now that we have defined the $TL^2$-space, and we have defined the notion of $\Gamma$-convergence, we can rephrase the content of Theorem \ref{DiscreteGamma} in the following way. Under the conditions on the domain $D$, the density $\rho$ and the parameter $\veps_n$ in Theorem \ref{DiscreteGamma}, with probability one, \textbf{all} of the following statements hold:
\begin{enumerate}
\item \textbf{Liminf inequality:} For all $u \in L^2(\nu)$, and all sequences $\left\{  u_n\right\}_{n \in \N}$ with $u_n \in L^2( \nu_n)$ and with $u_n \overset{{TL^2}}{\longrightarrow} u$ it is true that
$$  \liminf_{n \rightarrow \infty}  G_{n, \veps_n}(u_n) \geq \sigma_\eta G(u). $$
\item \textbf{Limsup inequality:} For all $u \in L^2(\nu)$, there exists a sequence $\left\{  u_n\right\}_{n \in \N}$ with $u_n \in L^2( \nu_n)$ and with $u_n \overset{{TL^2}}{\longrightarrow} u$ for which 
$$  \limsup_{n \rightarrow \infty}  G_{n, \veps_n}(u_n) \leq \sigma_\eta G(u). $$
\item \textbf{Compactness:} Every sequence $\left\{u_n\right\}_{n \in \N}$ with $u_n \in L^2(\nu_n)$, satisfying
\begin{equation*}
\sup_{n \in \N} F_{n}(x) < \infty, 
\end{equation*} 
is precompact in $TL^2$, that is, every subsequence of $\left\{u_n\right\}_{n \in \N}$ has a further subsequence, which converges in the $TL^2$-sense to an element of $L^2(D)$.
\end{enumerate}
In a similar fashion we can rephrase the content of Theorem \ref{DiscreteGammaNormalized}.

%Since the functionals $G_{n, \veps_n}$ and $\overline{G}_{n , \veps}$ depend on data (and hence are random), we need to recall  what it means for a sequence of random functionals to $\Gamma$-converge to a deterministic functional. 
%We remark that the previous definition guarantees that the conclusions of Proposition \ref{comp_gen} hold with probability one. We use the convention of omitting the dependence of $F_n$ on $\omega$ and write $F_n:X \rightarrow [0 , \infty] $, keeping in mind that we are always working with a fixed  $\omega \in \Omega$, and hence we can think that we are working with deterministic functionals. This convention of course applies to the functionals $G_{n , \veps_n}$  and $\overline{G}_{n , \veps_n}$ studied in this paper.

\subsection{Stability of $k$-means clustering} \label{sec:kmeans}
Here we prove some basic facts about the functional $F_{\mu,k}$ defined in \eqref{kmeans} and about Theorem \ref{th:km}. Our first observation is that there exist minimizers of $F_{\mu,k}$. We note that $F_{\mu,k}$ is a continuous function which is non-negative and the existence of minimizers can be obtained from a straightforward application of the direct method of the calculus of variations as we now illustrate. If the support of $\mu$ has $k$ or fewer points, then including these points in $\z$ provides a minimizer for which $F(\z)=0$.  On the other hand, if the support of $\mu$ has more than $k$ points then to show that a minimizer exists it is enough to obtain pre-compactness of a minimizing sequence due to the continuity of $F_{\mu,k}$. Let $\left\{ \z^m \right\}_{m \in \N}$ be a minimizing sequence of $F_{\mu,k}$.

By considering a subsequence we can assume that for any $i=1, \dots, k$, $z_i^m$ either converges to
some $z_i \in \R^N$ or diverges to $\pm \infty$. Also without the loss of generality we can assume that for some $1 \leq l \leq k+1$, the sequence $\left\{ z_i^m\right\}_{m \in \N}$ converges for $i < l$ and diverges for $i \geq l$. Our goal is to show that $l = k+1$. Assume for the sake of contradiction that $l \leq k$. First note that if $l=1$ (when no subsequence converges) then $F_{\mu,k}(\z^m) \to \infty$ as $m  \to \infty$, which is impossible. So we can assume that $z_1^m$ converges to $z_1$ as $m \to \infty$. 
%We can assume that for all $m$, $|z_1^m| < R = |z_1|+1$. ...
It is straightforward to show using the finiteness of the second moment of $\mu$, that 
\begin{equation}
F_{\mu,k}(\z^m) \to F_{\mu, l-1}(\left\{z_1,\dots,z_{l-1} \right\}), \quad \text{as } m \to \infty.
\label{auxStabilitykmeans1}
\end{equation}
However unless $l=k+1$, adding $k-(l-1)$ points from $\supp(\mu) \backslash \{z_1, \dots, z_{l-1}\}$ to $\{z_1, \dots, z_{l-1}\}$  would result on a value of $F_{\mu,k}$ that is strictly below $F_{\mu, l-1}(\left\{z_1,\dots,z_{l-1} \right\})$ and from \eqref{auxStabilitykmeans1}, this would contradict the assumption that $\left\{ \z^m \right\}_{m \in \N}$ is a minimizing sequence. We conclude that $\left\{ \z^m\right\}_{m \in \N}$ converges up to subsequence.

We now turn to comparing the properties of $F_{\mu,k}$ for different measures $\mu$, the ultimate goal is to prove Theorem \ref{th:km}. Let $\mu$ and $\nu$ be Borel probability measures on $\R^N$ with finite second moments and let  $\pi \in \Gamma(\mu,\nu)$ be the optimal transportation plan realizing the Wasserstein distance between $\mu$ and $\nu$, that is, assume 
$$d_2^2(\mu, \nu) = \iint_{\R^N \times \R^N} |x-y|^2 d\pi(x,y).$$  
Then
\begin{align*}
|F_{\mu,k}(\z) - F_{\nu,k}(\z)| & = \left| \int_{\R^N} d(x,\z)^2 d \mu(x) -  \int_{\R^N} d(y,\z)^2 d\nu(y) \right| \\
& = \left| \iint_{\R^N\times \R^N} \left( d(x,\z)^2 - d(y,\z)^2 \right) d\pi(x,y) \right|  \\
& \leq  \iint_{\R^N\times \R^N} \left|  d(x,\z)^2 - d(y,\z)^2 \right|  d\pi(x,y)  \\
& \leq \iint_{\R^N \times \R^N} (|x-y| + d(y,\z))^2 - d(y,\z)^2 d\pi(x,y) \\
& \leq d_2^2(\mu,\nu) + 2d_2(\mu, \nu) \sqrt{F_{\nu,k}(\z)},
\end{align*}
where the last inequality is obtained after expanding the integrand and using Cauchy-Schwartz inequality. By symmetry, we conclude that
\begin{equation} \label{contFmu}
| F_{\mu,k}(\z) - F_{\nu,k}(\z)|  \leq d_2(\mu, \nu) \left(2\min \left\{  \sqrt{F_{\mu,k}(\z)},  \sqrt{F_{\nu,k}(\z)} \right\} +d_2(\mu, \nu) \right).
\end{equation}
We also need the following Lemma.

\begin{lemma} \label{lem:nomassbdry}
Let $\mu $ be a Borel probability measure on $\R^k$ with finite second moment and at least $k$ points in its support. 
Let $\z= (z_1,\dots, z_k)$ be a minimizer of the functional $F_{\mu,k}$. Denote by $V_1,\dots, V_k$ the (closed) Voronoi cells induced by the points $z_1,\dots, z_k$, i.e. , $V_i:= \left\{ x \in \R^k \: : \: |x-z_i| = d(x, \z)  \right\}$. Then,
$$  \mu(V_i\cap V_j ) = 0  , \quad \forall i\not = j.$$
\end{lemma}
\begin{proof}
We start by recalling that if $k=1$ then the minimizer $z_1$ of $F_{\mu,1}$ is the centroid of $\mu$, that is $z_1 = \int x d \mu(x)$. We now consider $k \geq 2$. 
Since the support of $\mu$ has at least $k$ points, the points $z_1, \dots, z_k$ are distinct.
Assume that $\mu(V_i \cap V_j) >0$ for some $i \not =j$. Note that the set $V_i \cap V_j$ is contained in the plane $P_{ij}$ with normal vector $z_i - z_j$, and containing the point $\frac{1}{2}z_i + \frac{1}{2}z_j$. 
Let $\overline \mu_i = \mu \llcorner_{V_i}$ and $\overline \theta_i = \mu - \overline \mu_i$. Let $\hat \z = (z_1, \dots, z_{i-1}, z_{i+1}, \dots, z_n)$.
Note that $F_{\mu,k}(\z) = F_{\overline \mu_i, 1}(z_i) +  F_{\overline \theta_i, k-1}(\hat \z)$. Consequently $z_i$ minimizes $F_{\overline \mu_i, 1}$ and by remark above, $z_i$ is the centroid of $\overline \mu_i$, that is
$$  z_i =\frac{1}{\overline \mu_i(\R^k)} \int_{\R^k} x d  \overline \mu_i(x).$$
Now, let $\underline \mu_i = \mu_{\llcorner V_i \backslash V_j}$ and  $\underline \theta_i = \mu - \underline \mu_i$. Analogously to above 
$F_{\mu,k}(\z) = F_{\underline \mu_i, 1}(z_i) +  F_{\underline \theta_i, k-1}(\hat \z)$.
Hence $z_i$ minimizes $F_{\underline \mu_i, 1}$ and thus is the centroid of $\underline \mu_i$, i.e., 
$$  z_i =\frac{1}{\underline \mu_i(\R^k)} \int_{\R^k} x d  \underline \mu_i(x).$$
But note that $\mu(V_i \cap V_j) > 0$ implies that 
$$\frac{1}{\overline \mu_i(\R^k)} \int_{\R^k}  \langle x ,  z_j - z_i \rangle  d  \overline \mu_i(x) > \frac{1}{\underline \mu_i(\R^k)} \int_{\R^k} \langle x ,  z_j - z_i\rangle d  \underline \mu_i(x). $$
This, contradicts the fact that the centroids of $\underline \mu_i$ and $\overline \mu_i$ are both equal to $z_i$.
\end{proof}

The proof of Theorem \ref{th:km} is now a direct consequence of \eqref{contFmu} and Lemma \ref{lem:nomassbdry}.
\begin{proof}[Proof of Theorem \ref{th:km}]
Let $a^k = F_{\mu,k}(\z^k)$ and $a_m^k = F_{\mu_m,k}(\z_m^k)$, where $\z^k$ is a minimizer of $F_{\mu,k}$ and where $\z_m^k$ is a minimizer of $F_{\mu_m,k}$. Note that since the support of $\mu$ has at least $k$ points, $a_l > a_k$ for all $l <k$.
From \eqref{contFmu} it follows that $a_m^k \to a^k$ as $m \to \infty$ for any $k \in \N$. 
To show that  $\{\z_m^k, m \in \N\}$ is precompact it is enough to show that all coordinates are uniformly bounded. If this is not the case then there exists $1 \leq l \leq k$ such that coordinates
$1$ to $l$ converge, while those between $l+1$ and $k$ diverge to $\pm \infty$. 
Arguing as in the proof of the existence of minimizers at the beginning of this Section, and using \eqref{contFmu}, one obtains that 
$F_{\mu_m,k}(\z_m^k)$ converges to $F_{\mu , l}(\left\{ z_1,\dots,z_l \right\})$ for some $z_1, \dots, z_l \in \R^N$. If $l<k$, then this would imply that $a_l \leq  F_{\mu , l}(\left\{z_1,\dots,z_l\right\}) = \lim_{m \to \infty} a_m^k = a_k$, which would contradict the fact that $a_l > a_k$. Thus concluding that along a subsequence $\z_m^k \to \overline{\z}^k$ for some $\overline{\z}^k$. To show that $\overline{\z}^k$ minimizes $F_{\mu,k}$ simply observe that from \eqref{contFmu} and the continuity of $F_{\mu,k}$ it follows that
$$ F_{\mu,k}(\overline{\z}^k) = \lim_{m \rightarrow \infty} F_{\mu,k}(\z_m^k) = \lim_{n \rightarrow \infty} F_{\mu_m,k}(\z_m^k)= \lim_{m \rightarrow \infty} a_m^k= a^k, $$
which implies that indeed $\overline{\z}^k$ minimizes $F_{\mu,k}$.

Finally, the last part of the Theorem on convergence of clusters, follows from the fact that $\mu_m$ converge weakly to $\mu$, that their second moments are uniformly bounded, and that the boundaries of Voronoi cells change continuously when the centers are perturbed. 
\end{proof}

\subsection{The Spectra of $\mathcal{L} $, $ \mathcal{N}^{sym}$ and $\mathcal{N}^{rw}$}\label{SectionSpectrumL} 
The purpose of this section is to present some facts about the spectra of the %(formally defined)
 operators $\mathcal{L}$, $\mathcal{N}^{sym}$, and $\mathcal{N}^{rw}$. 
 %We start by defining rigorously what is meant by the spectrum of $\mathcal{L}$ via a weak formulation of equation \eqref{PDELimitFunctional0}. 
 These facts are standard %in the literature of linear elliptic PDEs
  (see \cite{evans2010partial}, or Chapter 8 in \cite{ButtazzoBook}). We present them for the convenience of the reader.

Let $D$ be an open, bounded, connected domain with Lipschitz boundary and let $\rho: D \rightarrow \R$ be a continuous density function satisfying \eqref{DensityBound}. For a given $w \in L^2(D)$ we consider the PDE
\begin{alignat}{2} \label{PDELimitFunctional} 
\begin{aligned}
 \mathcal{L}(u) &= w  \: & & \te{in }   D, \\ 
 \frac{\partial u}{\partial \bf{n}} & =0 \: &  & \te{on }  \partial D,
\end{aligned}
\end{alignat}
where we recall $\mathcal{L}$ is formally defined as $\mathcal{L}(u)= - \frac{1}{\rho}\divergence(  \rho^2 \nabla u )$. We say that $u \in H^1(D)$ is a \emph{weak solution} of \eqref{PDELimitFunctional} if
\begin{equation}
\int_{D} \nabla u \cdot \nabla v \rho^2(x) dx = \int_{D}vw \rho(x)dx, \quad \forall v \in H^1(D).
\label{weakEqn}
\end{equation}
\begin{remark}
Note that if $u$ is a solution of \eqref{PDELimitFunctional} in the classical sense, then integration by parts shows that $u$ is a weak solution of \eqref{PDELimitFunctional}. 
\end{remark}
A necessary condition for \eqref{PDELimitFunctional} to have a solution in the weak sense, is that $w $ belongs to the space
$$\mathcal{U}:= \left\{ w\in L^2(D) \: : \: \int_{D}w \rho(x) dx =0 \right\}.$$ 
This can be deduced by considering the test function $v \equiv 1$ in \eqref{weakEqn}. We consider the space
$$\mathcal{V}:= \left\{ v \in H^1(D) \: : \: \int_{D}v \rho(x) dx =0 \right\},$$ 
 and consider the bilinear form $a : \mathcal{V} \times \mathcal{V} \rightarrow \R$ given by
\begin{equation} \label{def:a}
  a(u,v) :=  \int_{D} \nabla u \cdot \nabla v  \rho^2\, dx. 
 \end{equation}
One can use the assumptions on $\rho$ in \eqref{DensityBound}, and Poincare's inequality (see Theorem 12.23 in \cite{Leoni}), to show that $a$ is coercive with respect to the $H^1$ inner product on $\mathcal{V}$, defined by
$$ \langle  u, v \rangle_{H^1(D)}:= \int_{D}u v dx + \int_{D} \nabla u \cdot \nabla v dx. $$
In addition,  $a$ is continuous and symmetric. 

Therefore by Lax-Milgram theorem \cite{evans2010partial}[Sec. 6.2] for any $w \in \mathcal{U}$ there exists a unique solution $u \in \mathcal{V}$ to \eqref{PDELimitFunctional}. From \eqref{weakEqn} and the assumption \eqref{DensityBound} on $\rho$, it follows that 
%From the continuity, coercivity and symmetry of $a$, we deduce that $a$ is an inner product on $\mathcal{V}$ which is equivalent to that inherited from $H^1(D)$. From now on we can consider $\mathcal{U}$ endowed with the inner product $\langle  \cdot, \cdot \rangle_{\rho}$ and $\mathcal{V}$ endowed with the inner product $a(\cdot, \cdot)$. The Riesz representation theorem then implies that for any $w \in \mathcal{U}$ there exists a unique solution $u \in \mathcal{V}$ to \eqref{PDELimitFunctional} and from \eqref{weakEqn} and the assumption \eqref{DensityBound} on $\rho$, it follows that 
\begin{equation}
\int_{D} | \nabla u |^2 \rho^2(x) dx \leq C \int_{D} | w |^2  \rho(x)dx,
\label{ContinuityOfInverse}
\end{equation}
for a constant $C$. We can then define the inverse $\mathcal{L}^{-1}: \mathcal{U} \rightarrow \mathcal{V}$ of $\mathcal{L}$, by letting  $  \mathcal{L}^{-1}: w \mapsto  u,  $
where $u$ is the unique solution of \eqref{PDELimitFunctional}.  From \eqref{ContinuityOfInverse}, it follows that $\mathcal{L}^{-1}$ is a continuous linear function. 
Rellich--Kondrachov theorem (see Theorem 11.10 in \cite{Leoni}) implies that
$\mathcal{L}^{-1}$ is compact.
\medskip

We say that  $\lambda \in \R$ is an \emph{eigenvalue} of the operator $\mathcal{L}$, if there exists a nontrivial $u \in H^1(D)$ which is a weak solution of \eqref{PDELimitFunctional0}. That is if
\begin{equation} 
a(u,v) = \int_{D} \nabla u \cdot \nabla v \rho^2(x) dx = \lambda \int_{D}u v \rho(x)dx = \lambda \langle u,v \rangle_\rho\,, \quad  \forall v \in H^1(D).  
\label{DefEigenfuncWeak}
\end{equation}
Such function $u$ is called an \emph{eigenfunction}.
%
%We can now interpret equation \eqref{PDELimitFunctional0} defining the spectrum of $\mathcal{L}$ in the weak sense.
%\begin{definition}
%We say that $\lambda \in \R$ is an \emph{eigenvalue} of the operator $\mathcal{L}$, if there exists a nontrivial $u \in H^1(D)$ for which
%\begin{equation} 
%\int_{D} \nabla u \cdot \nabla v \rho^2(x) dx = \lambda \int_{D}u v \rho(x)dx\:, \: \forall v \in H^1(D).  
%\label{DefEigenfuncWeak}
%\end{equation}
%The function $u$ is then called an eigenfunction of $\mathcal{L}$ with eigenvalue $\lambda$.
%\label{DefEigenfuncWeak1}
%\end{definition}
\begin{remark}
We remark that $\lambda_1=0$ is an eigenvalue of $\mathcal{L}$ and that the function $u_1$ identically equal to one is an eigenfunction associated to $\lambda_1$. Given that $D$ is connected, it follows that the eigenspace associated to $\lambda_1=0$ is the space of constant functions on $D$. We also remark that $\mathcal{U}$ is by definition the orthogonal complement (with respect to the inner product $\langle\cdot,\cdot \rangle_\rho$) of $\Span \left\{u_1 \right\}$. 
\end{remark}
Using the definition of $\mathcal{L}$, the definition of weak solutions to \eqref{PDELimitFunctional} it follows that
\begin{equation}
u \text{ is an eigenfunction of } \mathcal{L} \text{ with eigenvalue } \lambda \not = 0 \; \text{ iff } \; \mathcal{L}^{-1}(u) = \frac{1}{\lambda} u.
\label{EigenFuncOfInverse} 
\end{equation}
In other words the non-constant eigenfunctions of $\mathcal{L}$ are the eigenfunctions of $\mathcal{L}^{-1}$, and the nonzero eigenvalues of $\mathcal{L}$ are the reciprocals of the eigenvalues of $\mathcal{L}$. Thus, by understanding the structure of the spectrum of $\mathcal{L}^{-1}$, one can obtain properties of the spectrum of $\mathcal{L}$.
\begin{proposition} \label{prop:spcomp}
The operator $\mathcal{L}^{-1}: \mathcal{V} \rightarrow \mathcal{V}$ is a  selfadjoint, positive semidefinite (with respect to the inner product $a(\cdot, \cdot)$) and compact.  The eigenvalues of $\mathcal{L}^{-1}$ can be arranged as a decreasing sequence of positive numbers,
$$ \lambda_2^{-1} \geq \lambda_{3}^{-1} \geq \dots  $$
repeated according to (finite) multiplicity and converging to zero. Moreover, there exists an orthonormal basis $\left\{v_k \right\}_{k \geq 2}$ of $\mathcal{V}$, where each of the functions $v_k$ is an eigenfunction of $\mathcal{L}^{-1}$ with corresponding eigenvalue $\lambda_k^{-1}$.
\end{proposition}
\begin{proof}
 In order to show that $\mathcal{L}^{-1}: \mathcal{V} \rightarrow \mathcal{V}$ is self-adjoint with respect to $a(\cdot, \cdot)$, take $v_1, v_2 \in \mathcal{V}$ and let $u_i= \mathcal{L}^{-1}v_i$ for $i=1,2$. 
We claim that 
$$  a( \mathcal{L}^{-1}v_1 , v_2   ) = \langle v_1 , v_2 \rangle_\rho.   $$ 
In fact, from the definition of $\mathcal{L}^{-1}$ it follows that
$$  a( \mathcal{L}^{-1}v_1 , v_2   ) = a(u_1, v_2) = \int_{D} \nabla u_1 \cdot \nabla v_2 \rho^2(x) dx = \int_{D} v_1 v_2 \rho(x) dx = \langle v_1, v_2   \rangle_\rho      $$
From the previous identity, it immediately follows that $\mathcal{L}^{-1}$ is self-adjoint and positive semidefinite with respect to the inner product $a(\cdot, \cdot)$.
The compactness of $\mathcal{L}^{-1}$ follows from Rellich--Kondrachov theorem (see Theorem 11.10 in \cite{Leoni}). The statements about the spectrum of $\mathcal{L}^{-1}$ are a direct consequence of Riesz-Schauder theorem and Hilbert-Schmidt theorem  (see \cite{SimonReed}).
\end{proof}
For $k \geq 2$, let $v_k$ be eigenfunctions as in the previous proposition and define $u_k$  by
\begin{equation}
 u_k := \sqrt{\lambda_k} v_k.
 \label{orthonormalBase} 
 \end{equation}
We claim that $\left\{ u_k \right\}_{k \geq 2}$ is an orthonormal base of $\mathcal{U}$ with respect to $\langle \cdot,\cdot \rangle_{\rho}$. In fact, it follows from the definition of $\mathcal{L}^{-1}$ and\eqref{DefEigenfuncWeak} that
$$\delta_{kl} = a(v_k, v_l) = \lambda_k \langle v_k , v_l \rangle_{\rho}  = \frac{\sqrt{\lambda_k}}{\sqrt{\lambda_l}} \langle u_k , u_l , \rangle_\rho,$$
where $\delta _{kl} =1$ if $k=l$ and $\delta_{kl}=0$ if $k \not = l$. Hence 
$ \langle u_k , u_l , \rangle_\rho = \delta_{kl}$.
In other words $\left\{ u_k \right\}_{k \geq 2}$ is an orthonormal set.
Completeness follows from the completeness in Proposition \ref{prop:spcomp} and density of $H^1(D)$ in $L^2(D)$.
%
% Now let us show, that actually, $\left\{ u_k\right\}_{k \geq 2}$ is a orthonormal base for $\mathcal{U}$. Suppose that $w \in \mathcal{U}$ is such that $\langle
%w , u_k  \rangle_\rho =0 $ for all $k$, we want to show that $w \equiv 0$. Note that by definition of the $u_k$ and by \eqref{weakEqn}, we have for all $k \geq 2$
%$$ 0 =   \langle w , v_k  \rangle_\rho = a(\mathcal{L}^{-1}w , v_k ).    $$  
%Since $\left\{ v_k \right\}_{k \geq 2}$ is a orthonormal base for $\mathcal{V}$, we conclude that $\mathcal{L}^{-1}w \equiv 0$. Using \eqref{weakEqn} again, we conclude that
%$$ \int_{D} w v \rho dx =0, \forall v \in H^1(D).  $$ 
%Since $H^1(D)$ is dense in $L^2(D)$, we conclude that $w \equiv 0$ as we wanted to show. 

By setting $u_1 \equiv 1$ and by noticing that $L^2(D)= \Span\left\{ u_1 \right\} \oplus \mathcal{U} $, we conclude that $\left\{ u_k \right\}_{k \in \N}$ is a orthonormal base for $L^2(D)$ with inner product $\langle \cdot ,\cdot \rangle_{\rho}$. 
The next proposition is a direct consequence of the previous discussion and \eqref{EigenFuncOfInverse}. \nc
\begin{proposition}
$\mathcal{L}$ has a countable family of eigenvalues $\left\{ \lambda_{k} \right\}_{k \in \N}$ which can be written as an increasing sequence of nonnegative numbers which tends to infinity as $k$ goes to infinity, that is,
$$  0=\lambda_1 < \lambda_2   \leq \dots \leq \lambda_k \leq \dots  $$
Each eigenvalue, is repeated according to (finite) multiplicity. Moreover, there exists $\left\{ u_k \right\}_{k\in \N}$ an orthonormal basis (with respect to $\langle\cdot, \cdot \rangle_\rho$)  of $L^2(D)$, such that for every $k\in \N$, $u_k$ is an eigenfunction of $\mathcal{L}$ associated to $\lambda_k$.
\label{PropoSpectrumL}
\end{proposition}

Finally we present the Courant-Fisher maxmini principle. 
\begin{proposition}
Consider an orthonormal base $\left\{  u_k\right\}_{k \in \N}$ for $L^2(D)$ with respect to the inner product $\langle \cdot, \cdot \rangle_\rho$, where for each $k\in \N$, $u_k$ is an eigenfunction of $\mathcal{L}$ with eigenvalue $\lambda_k$. Then, for every $k \in N$
\begin{equation}
\lambda_k = \min_{\|u\|_{\rho}=1 \:, \: u \in {S^*}^{\perp}}  G(u), 
\label{VarEigen1L}
\end{equation}
where $S^*=\Span\left\{ u_1, \dots, u_{k-1} \right\}$ and where $S^*$ denotes the orthogonal complement of $S^*$ with respect to the inner product $\langle \cdot, \cdot \rangle_{\rho}$. Additionally,
\begin{equation}
\lambda_{k} = \sup_{ S \in \Sigma_{k-1}} \min_{\|u\|_{\rho}=1 \:, \: u \in S^\perp}  G(u),
\label{VarEigen2L}
\end{equation}
where $\Sigma_{k-1}$ denotes the set of $(k-1)$-dimensional subspaces of $L^2(D)$, and where $S^\perp$ represents the orthogonal complement of $S$ with respect to the inner product $\langle \cdot, \cdot \rangle_{\rho}$.
\label{CourantFisher} 
\end{proposition}
The proof (of a similar statement) can be found in Chapter 8.3 in \cite{ButtazzoBook}. 

\begin{remark}
If the density $\rho$ is smooth, then the eigenfunctions of $\mathcal{L}$ are smooth inside $D$. 
\end{remark}
\medskip 

We now turn to  the spectrum of $\mathcal{N}^{sym}$.
We say that $\tau \in \R$ is an \emph{eigenvalue} of the operator $\mathcal{N}^{sym}$, if there exists a nontrivial $u \in H^1_{\sqrt{\rho}}(D)$ which solves \eqref{PDELimitFunctionalN}. That is if 
\begin{equation} 
\int_{D} \nabla \left( \frac{u}{\sqrt{\rho}} \right) \cdot \nabla \left( \frac{v}{\sqrt{\rho}} \right)\rho^2(x) dx = \tau \int_{D}u v \rho(x)dx, \quad \forall v \in H_{\sqrt{\rho}}^1(D).  
%\label{DefEigenfuncWeak}
\end{equation}
The function $u$ is then called an eigenfunction of $\mathcal{N}^{sym}$ with eigenvalue $\tau$.
%\label{DefEigenfuncWeak1}
\begin{remark}
We remark that $\tau_1=0$ is an eigenvalue of $\mathcal{N}^{sym}$ and that the function $\overline{u}_1$ equal to
$$ \overline{u}_1(x)= \frac{\sqrt{\rho(x)}}{\| \sqrt{\rho}\|_{\rho}} $$
is an eigenfunction of $\mathcal{N}^{sym}$, with eigenvalue $\tau_1=0$.  Given that $D$ is connected, it actually follows that $\tau_1=0$ has multiplicity one and thus the eigenspace associated to $\tau_1=0$ is the space of multiples of $\sqrt{\rho}$. 
\end{remark}
Following the same ideas used when considering the spectrum of $\mathcal{L}$, we can establish the following analogous results. 
\begin{proposition}
$\mathcal{N}^{sym}$ has a countable family of eigenvalues $\left\{ \tau_{k} \right\}_{k \in \N}$  which can be written as an increasing sequence of nonnegative numbers which tends to infinity as $k$ goes to infinity, that is,
$$  0=\tau_1\leq \tau_2   \leq \dots \leq \tau_k \leq \dots  $$
Each eigenvalue, is repeated according to (finite) multiplicity. Moreover, there exists $\left\{ \overline{u}_k \right\}_{k\in \N}$ an orthonormal basis (with respect to $\langle\cdot, \cdot \rangle_{\rho}$)  of $L^2(D)$, such that for every $k\in \N$, $\overline{u}_k$ is an eigenfunction of $ \mathcal{N}^{sym}$ associated to $\tau_k$. 
\end{proposition}
\begin{proposition}
Consider a orthonormal base $\left\{  \overline{u}_k\right\}_{k \in \N}$ for $L^2(D)$ with respect to the inner product $\langle \cdot, \cdot \rangle_\rho$, where for each $k\in \N$, $\overline{u}_k$ is an eigenfunction of $\mathcal{N}^{sym}$ with eigenvalue $\tau_k$. Then, for every $k \in \N$
\begin{equation}
\tau_k = \min_{\|u\|_{\rho}=1 \:, \: u \in {S^*}^{\perp}}  \overline{G}(u), 
\label{VarEigen1}
\end{equation}
where $S^*=\Span\left\{ \overline{u}_1, \dots, \overline{u}_{k-1} \right\}$. Additionally,
\begin{equation*}
\tau_{k} = \sup_{ S \in \Sigma_{m-1}} \min_{\|u\|_{\rho}=1 \:, \: u \in S^\perp}  \overline{G}(u),
\label{VarEigen2}
\end{equation*}
where $\Sigma_{m-1}$ denotes the set of $(m-1)$-dimensional subspaces of $L^2(D)$, and where $S^\perp$ represents the orthogonal complement of $S$ with respect to the inner product $\langle \cdot, \cdot \rangle_{\rho}$. 
\end{proposition}

Finally, we consider the spectrum of $\mathcal{N}^{rw}$.
We say that $\tau \in \R$ is an \emph{eigenvalue} of the operator $\mathcal{N}^{rw}$, if there exists a nontrivial $u \in H^1(D)$ for which
\begin{equation} 
\int_{D} \nabla u \cdot \nabla v\rho^2(x) dx = \tau \int_{D}u v \rho^2(x)dx\:, \: \forall v \in H^1(D).  
\label{DefEigenfuncWeak}
\end{equation}
The function $u$ is then called an eigenfunction of $\mathcal{N}^{rw}$ with eigenvalue $\tau$.
\label{DefEigenfuncWeak1}
From the definition, it follows that $\tau$ is an eigenvalue of $\mathcal{N}^{rw}$ with eigenfunction $u$ if and only if $\tau$ is an eigenvalue of $\mathcal{N}^{sym}$ with eigenvector $w := \sqrt{\rho}u$. This is analogous to \eqref{ConditionNsymNrw} in the discrete case.

\section{Convergence of the spectra of unnormalized graph Laplacians}
\label{ProofsMainTheorems}
We start by establishing Theorem \ref{DiscreteGamma}. 

\begin{proof}[Proof of Theorem \ref{DiscreteGamma}]
As done in Section 5 in \cite{GTS} and due to the assumptions $\textbf{(K1)}-\textbf{(K3)}$ on $\bm{\eta}$, we can reduce the problem to that of showing the result for the kernel $\bm{\eta}$ defined by
$$ \bm{\eta}(t):=
\begin{cases}
 1,  & \te{if } t\in[0,1], \medskip \\
 0,  & \te{if } t>1.
\end{cases}
$$
We use the sequence of transportation maps $\left\{ T_n\right\}_{n \in \N}$ from Proposition \ref{thm:InifinityTransportEstimate}.  Let $\omega \in \Omega$ be such that
\eqref{InifinityTransportEstimate d=2} and \eqref{InifinityTransportEstimate d>2} hold in cases  $d=2$ and $d \geq 3$ respectively.  By Proposition \ref{thm:InifinityTransportEstimate} the complement in $\Omega$ of such $\omega$'s is contained in a set of probability zero. The key idea in the proof is that the estimates of Proposition \ref{thm:InifinityTransportEstimate} imply that the transportation happens on a length scale which is small compared to $\veps_n$. By taking a kernel with slightly smaller radius than $\veps_n$ we can then obtain a lower bound, and by taking a slightly larger radius a matching upper bound on the functional $G_{n,\veps_n}$.

\textbf{Liminf inequality:} Assume that $u_n \overset{TL^1}{\longrightarrow} u$ as $n \rightarrow \infty$. Since  $T_{n\sharp} \nu= \nu_n$, using the change of variables \eqref{chofvar} it follows that
\begin{equation}
G_{n , \veps_n}(u_n) = \frac{1}{\veps_n^2} \int_{D \times D} \eta_{\veps_n}\left(T_n(x) - T_n(y) \right)( u_n\circ T_n(x)  - u_n \circ T_n(y) )^2\rho(x)\rho(y) dx dy.
\label{RepTvGraph}
\end{equation}  
Note that for $(x,y) \in D \times D $ 
\begin{equation}
\left| T_n(x) - T_n(y) \right| > \veps_n \Rightarrow |x-y| > \veps_n - 2\|Id - T_n\|_\infty.
\label{implication}
\end{equation}

Thanks to the assumptions on $\left\{  \veps_n \right\}_{n \in \N}$ 
(\eqref{InifinityTransportEstimate d=2} and \eqref{InifinityTransportEstimate d>2} in cases  $d=2$ and $d \geq 3$ respectively), for large enough $n \in \N$:
\begin{equation}
\tilde{\veps}_n: =  \veps_n - 2\|Id - T_n\|_\infty >0.
\label{tildevepsn}
\end{equation}
By \eqref{implication}, and our choice of kernel $\bm{\eta}$, for large enough $n$ and for every $(x,y) \in D\times D$, we obtain 
\begin{equation*}
\bm{\eta}\left(\frac{|x-y|}{ \tilde{\veps}_n} \right) \leq \bm{\eta}\left( \frac{|T_n(x)- T_n(y)|}{\veps_n} \right).
\end{equation*}
We now consider $\tilde u_n = u_n \circ T_n$. Thanks to the previous inequality and \eqref{RepTvGraph}, for large enough $n$ 
\begin{align*}
G_{n , \veps_n}(u_n) \geq & \frac{1}{\veps_n^{d+1}} \int_{D \times D}\bm{ \eta}\left(\frac{|x-y|}{\tilde{\veps}_n}\right) ( \tilde u_n(x) - \tilde u_n(y) )^2\rho(x)\rho(y) dxdy
\\ = &  \left(\frac{\tilde{\veps}_n}{\veps_n} \right)^{d+2} G_{\tilde{\veps}_n}(\tilde u_n).
\end{align*}
Note that $\frac{\tilde{\veps}_n}{\veps_n} \rightarrow 1$ as $n \rightarrow \infty$ and that $u_n \overset{{TL^1}}{\longrightarrow} u$ by definition implies  $\tilde u_n \overset{{L^1(D,\rho)}}{\longrightarrow} u$ as $n \rightarrow \infty$. We deduce from the liminf inequality of Proposition \ref{NonlocalDirichlet} that $  \liminf_{n \rightarrow \infty} G_{\tilde{\veps}_n}(\tilde u_n) \geq \sigma_\eta G(u)  $ and hence:
\begin{equation*}
\liminf_{n \rightarrow \infty} G_{n , \veps_n}(u_n) \geq \sigma_\eta G(u).
\end{equation*}

\textbf{Limsup inequality:}
By using a diagonal argument it is enough to establish the limsup inequality for a dense subset of $L^2(D)$ and in particular we consider the set of Lipschitz continuous functions $u : D \rightarrow \R $. That is, we want to show that if $u: D \rightarrow \R$ is a Lipschitz continuous function, then there exists a sequence of functions $\left\{ u_n \right\}_{n \in \N}$, where $u_n \in L^2(\nu_n)$ and
$$u_n \overset{TL^2}{\longrightarrow} u   \quad \text{ as } n \rightarrow \infty , \quad \limsup_{n \rightarrow \infty}G_{n, \veps_n}(u_n) \leq \sigma_\eta G(u).  $$

We define $u_n$ to be the restriction of $u$ to the first $n$ data points $\x_1, \dots , \x_n$. We note that this operation is well defined due to the fact that $u$ is in particular continuous. It is straightforward to show that given that $u$ is Lipschitz we have $u_n \overset{{TL^2}}{\longrightarrow} u$.

Now, consider $\tilde{\veps}_n := \veps_n + 2\|Id - T_n\|_\infty$ and let $\tilde u_n = u \circ T_n$. The choice of kernel $\bm{\eta}$ implies that for every $(x,y)  \in D\times D$ 
\begin{equation*}
\bm{\eta}\left(\frac{|T_n(x)-T_n(y)|}{ \veps_n} \right) \leq \bm{\eta}\left( \frac{|x- y|}{ \tilde{\veps}_n} \right).
\end{equation*}
It follows that for all $n\in \N$ 
\begin{align}
\begin{split}
\frac{1}{\tilde{\veps}_n^{d+2}} \int_{D \times D} \bm{ \eta} & \left(  \frac{|T_n(x)- T_n(y)|}{\veps_n} \right)  ( \tilde u_n(x) - \tilde u_n(y) )^2 \rho(x)\rho(y) dxdy
\\ & \leq \frac{1}{\tilde{\veps}_n^2}\int_{D \times D} \eta_{\tilde{\veps}_n}\left(x-y \right) (\tilde u_n (x) - \tilde u_n(y)  )^2 \rho(x)\rho(y) dxdy.
\end{split}
\label{Ineq111}
\end{align}
Now let $A_n$ and $B_n$ be given by
$$  A_n := \frac{1}{\tilde{\veps}_n^2}    \int_{D \times D} \eta_{\tilde{\veps}_n}(x-y) (u(x) - u(y))^2  \rho(x)\rho(y) dxdy   $$
$$ B_n:= \frac{1}{\tilde{\veps}_n^2}    \int_{D \times D} \eta_{\tilde{\veps}_n}(x-y) (\tilde{u}_n(x) - \tilde{u}_n(y))^2  \rho(x)\rho(y) dxdy. $$

Then, 
\begin{align}
\begin{split}
 \left(  \sqrt{A_n}- \sqrt{B_n} \right)^2 \leq 
\frac{1}{\tilde{\veps}_n^2} &   \int_{D \times D} \eta_{\tilde{\veps}_n}(x-y)\left(u(x) - \tilde{u}_n(x) + \tilde{u}_n(y) - u(y) \right)^2 \rho(x)\rho(y) dxdy 
\\ & \leq \frac{4}{\tilde{\veps}_n^2}  \int_{D \times D} \eta_{\tilde{\veps}_n}(x-y)(u(x) - \tilde{u}_n(x))^2 \rho(x) \rho(y)dxdy
\\ & \leq  \frac{4C\Lip(u)^2\|\rho\|_{L^\infty(D)}^2 \|Id-T_n\|_{\infty}^2}{\tilde{\veps}_n^2} ,
\end{split}
\label{AssympLips}
\end{align}
where the first inequality follows using Minkowski's inequality, and where $C= \int_{\mathbb{R}^d} \eta(h) dh$. The last term of the previous expression goes to $0$ as $n \rightarrow \infty$, yielding 
$$ \lim_{n \rightarrow \infty } | \sqrt{A_n} - \sqrt{B_n}| =0. $$
 On the other hand, by \eqref{PointwiseNonlocal} it follows that $A_n$ is bounded on $n$ and in particular it follows that
\begin{equation}
  \lim_{n \rightarrow \infty}| A_n - B_n| =0.
   \label{auxDiscreteGamma}
 \end{equation}
We conclude that
\begin{align*}
\limsup_{n \rightarrow \infty} G_{n , \veps_n}(u_n) 
= & \limsup_{n \rightarrow \infty}\frac{1}{\tilde{\veps}_{n}^{d+2}} \int_{D \times D}\bm{ \eta} \left(  \frac{|T_n(x)- T_n(y)|}{\veps_n} \right) ( \tilde{u}_n(x) - \tilde{u}_n(y) )^2 \rho(x)\rho(y)dxdy
\\ \leq & \limsup_{n \rightarrow \infty} \frac{1}{\tilde{\veps}_n^2} \int_{D \times D} \eta_{\tilde{\veps}_n}(x-y) ( \tilde{u}_n(x) - \tilde{u}_n(y) )^2\rho(x)\rho(y)dxdy
\\ = & \limsup_{n \rightarrow \infty} G_{\tilde{\veps}_n}(u) = \sigma_\eta G(u),
\end{align*}
where the first equality is obtained from the fact that $\frac{\veps_n}{\tilde{\veps}_n} \rightarrow 1$ as $n \rightarrow \infty$, the first inequality is obtained from \eqref{Ineq111}, the second equality is obtained from \eqref{auxDiscreteGamma} and the last equality is obtained from \eqref{PointwiseNonlocal}.

\textbf{Compactness:} Finally, to see that the compactness statement holds suppose that $\left\{ u_n \right\}_{n \in \N}$ is a sequence with $u_n \in L^2(\nu_n)$ and such that
$$ \sup_{n \in \N} \| u_n\|_{L^2(\nu_n)}  < \infty, \quad \sup_{n \in \N} G_{n, \veps_n}(u_n) < \infty. $$
Note that in particular $ \sup_{n \in \N} \| u_n\circ T_n\|_{L^2(\nu)} < \infty $. We want to show that 
$$\sup_{n \in \N} G_{\veps_n}(u_n \circ T_n) < \infty  $$

To see this, note that for large enough $n$, we can set  $\tilde{\veps}_n := \veps_n - 2\| Id - T_n\|_\infty$ as in \eqref{tildevepsn}. Thus, for large enough $n$:
\begin{align*}
\frac{1}{\veps_n^{d+2}} \int_{D\times D} & \bm{\eta}\left(\frac{|z-y|}{\tilde{\veps}_n} \right) ( u_n \circ T_n (z) - u_n \circ T_n(y)  )^2 \rho(z) \rho(y)dzdy \\
& \leq  \frac{1}{\veps_n^{d+2}} \int_{D\times D}  \bm{\eta}\left(\frac{|T_n(z)-T_n(y)|}{\tilde{\veps}_n} \right) ( u_n \circ T_n (z) - u_n \circ T_n(y)  )^2 \rho(z) \rho(y) dzdy
\\&= G_{n , \veps_n} (u_n).
\end{align*}
Thus
\begin{equation*}
\sup_{n \in \N} \frac{1}{\veps_n^{d+2}} \int_{D\times D}  \bm{\eta}\left(\frac{|z-y|}{\tilde{\veps}_n} \right) (u_n \circ T_n (z) - u_n \circ T_n(y)  )^2\rho(z)\rho(y) dzdy < \infty.
\end{equation*} 
Finally noting that $\frac{\tilde{\veps}_n}{\veps_n} \rightarrow 1$ as $n \rightarrow \infty$ we deduce that:
\begin{equation*}
\sup_{n \in \N}G_{\veps_n}(u_n \circ T_n)< \infty.
\end{equation*} 
By Proposition \ref{NonlocalDirichlet}  we conclude that $\left\{ u_n \circ T_n \right\}_{n \in \N}$ is relatively compact in $L^2(\nu)$ and hence $\left\{ u_n \right\}_{n \in \N}$ is relatively compact in $TL^2$.
\end{proof}

Now we prove Theorem \ref{ConvergenceSpectrumTheorem}. 
\subsection{Convergence of Eigenvalues}
\label{SecConvEigenvalues}

First of all note that because $\mathcal{L}_{n, \veps_n}$ is self-adjoint with respect to the Euclidean inner product in $\R^n$, in particular it is also self-adjoint with respect to the inner product $\langle\cdot, \cdot\rangle_{\nu_n}$ and  furthermore, it is positive semi-definite. In particular, we can use the Courant-Fisher maxmini principle to write the eigenvalues $0=\lambda_1^{(n)}\leq \dots \leq \lambda_{n}^{(n)}$ of $\mathcal{L}_{n, \veps_n}$ as
\begin{equation*}
\lambda_{k}^{(n)} = \sup_{ S \in \Sigma^{(n)}_{k-1}} \min_{\|u\|_{\nu_n}=1 \:, \: u \in S^\perp}  \langle  \mathcal{L}_{n, \veps_n} u , u \rangle_{\nu_n} ,
\end{equation*}
where $\Sigma^{(n)}_{k-1}$ denotes the set of subspaces of $\R^n$ of dimension $k-1$. On the other hand, for any $u_n \in L^2(\nu_n)$, from \eqref{LaplacianAndEnergy} it follows that
\begin{equation}
G_{n,\veps_n}(u_n) = \frac{2}{n \veps_n^2} \langle \mathcal{L}_{n,\veps_n} u_n , u_n \rangle_{\nu_n}
\label{LaplacianAndEnergyProofs}
\end{equation}
Therefore, 
\begin{equation*}
\frac{2\lambda_{k}^{(n)}}{ n \veps_n^2} = \sup_{ S \in \Sigma^{(n)}_{k-1}} \min_{\|u\|_{\nu_n}=1 \:, \: u \in S^\perp} G_{n , \veps_n}(u).
\label{CourantFisherDiscreteUnnormalized}
\end{equation*}
%On the other hand, the eigenvalues $\lambda_1 \leq \lambda_2 \leq \dots$ of $\mathcal{L}$, can be written as 
%\begin{equation}
%\lambda_k= \sup_{ S \in \Sigma_{k-1}} \min_{\|u\|_{\rho}=1 \:, \: u \in S^\perp} G(u),
%\label{CourantFisherContinuumUnnormalized}
%\end{equation}
%where $\Sigma_{k-1}$ denotes the set of subspaces of $L^2(D)$ with dimension $k-1$, and $S^\perp$ denotes the orthogonal complement of $S$ with respect to the inner product $\langle\cdot, \cdot \rangle_{\rho}$ in $L^2(D)$.
%\begin{theorem}
%For every $k \in \N$ we have:
%\begin{equation*}
%\lim_{n \rightarrow \infty}    \frac{2 \lambda_{k}^{(n)} }{ n \veps_n^{2} }= \sigma_\eta \lambda_k
%\label{EigenvalueConvergence}
%\end{equation*}
%\end{theorem}
Let us first prove the first statement from Theorem \ref{ConvergenceSpectrumTheorem}.  
The proof is by induction on $k$. For $k=1$, we know that $\lambda_1^{(n)}=0$ for every $n$. Also, $\lambda_1=0$, so trivially \eqref{EigenvalueConvergence} is true when $k=1$.
Now, suppose that \eqref{EigenvalueConvergence} is true for  $i=1 , \dots ,k -1$. We want to prove that the result holds for $k$.
 
\textbf{Step 1: } In this first step we prove that $\sigma_\eta \lambda_k \leq \liminf_{n \rightarrow \infty} \frac{2\lambda_k^{(n)}}{n \veps_n^{2}}$. Let $S \in \Sigma_{k-1}$, we let $\left\{ u_1 , \dots , u_{k-1} \right\}$ be an orthonormal base for $S$.  Then, for every $i=1,\dots, k-1$, there exists a sequence $\left\{ u_i^{n} \right\}_{n \in \N} $  (with $u_i^n \in L^2(\nu_{n})$) such that $u_i^{n} \overset{{TL^2}}{\longrightarrow} u_i$ as $n \rightarrow \infty$. The existence of such sequence follows from the limsup inequality of Theorem \ref{DiscreteGamma}. 
Proposition \ref{ContinuityInnerProdTL2} implies that for all $i=1,\dots, k-1$
\begin{equation*}
\lim_{n \rightarrow \infty}\|u_i^{n}\|_{\nu_n}= \|u_i\|_{\rho}=1, \:
\label{auxConvergenceInnerProds1}
\end{equation*}
and that for $i \not = j$
\begin{equation}
\lim_{n \rightarrow \infty} \langle u_i^{n}, u_j^{n}  \rangle_{\nu_n} = \langle u_i , u_j\rangle_\rho=0.
\label{auxConvergenceInnerProds2}
\end{equation}
Thus,  for large enough $n$, the space generated by $\left\{ u_1^n , \dots, u_{k-1}^n  \right\}$
is $k-1$ dimensional. We can use the Gram-Schmidt orthogonalization process, to obtain an orthonormal base $\left\{\tilde{u}_1^n , \dots, \tilde{u}_{k-1}^n \right\}$. 
That is, we define
$ \tilde{u}_1^{n}:= {u_1^{n} }/{\|u_1^{n} \|_{\nu_n}}$,
and recursively 
$\tilde{v}_i^{n} := u_i^{n}- \sum_{j=1}^{i-1}\langle u_i^{n}, \tilde{u}_j^{n} \rangle_{\nu_n} \tilde{u}_{j}^{n}$, and $\tilde{u}_{i}^{n} := { \tilde{v}_i^{n} }/{\|\tilde{v}_i^{n} \|_{\nu_n}}$ for $i=2, \dots, k-1$.

%That is, we define:
%\begin{equation*}
% \tilde{u}_1^{n}:= \frac{u_1^{n} }{\|u_1^{n} \|_{\nu_n}},
%\end{equation*}
%and recursively
%\begin{align*}
%\tilde{v}_i^{n} &:= u_i^{n}- \sum_{j=1}^{i-1}\langle u_i^{n}, \tilde{u}_j^{n} \rangle_{\nu_n} \tilde{u}_{j}^{n},
%\\
%\tilde{u}_{i}^{n}& := \frac{ \tilde{v}_i^{n} }{\|\tilde{v}_i^{n} \|_{\nu_n}}.
%\end{align*}
It follows from \eqref{auxConvergenceInnerProds2} and Proposition \ref{ContinuityInnerProdTL2} that $ \tilde{u}_i^{n } \overset{{TL^2}}{\longrightarrow} u_i $ as $n \rightarrow \infty$ for every $i=1 , \dots , k-1$. Let $S_{n}:= \Span \left\{  \tilde{u}^{n}_{1}, \dots, \tilde{u}^{n}_{k-1} \right\} $. We claim that 
\begin{equation}
 \liminf_{n \rightarrow \infty } \frac{ 2\lambda_k^{(n)}}{n \veps_n^2} \geq \min_{\|u\|_\rho=1, \: u \in S^{\perp} } \sigma_\eta G(u)
\label{ClaimEigenvalues}
\end{equation}
First, note that if 
$$\liminf_{n \rightarrow \infty} \min_{\|u\|_{\nu_n}=1,\: u \in {S_n}^{\perp}} G_{n , \veps_n}(u) = \infty,$$ 
then in particular
$$ \liminf_{n \rightarrow \infty }\frac{ 2\lambda_k^{(n)}}{n \veps_n^2} \geq  \liminf_{n \rightarrow \infty} \min_{\|u\|_{\nu_n}=1,\: u \in {S_n}^{\perp}} G_{n , \veps_n}(u) = \infty, $$ 
and in that case \eqref{ClaimEigenvalues} follows trivially. Let us now assume that  $\liminf_{n \rightarrow \infty} \min_{\|u\|_{\nu_n}=1,\: u \in {S_n}^{\perp}} G_{n , \veps_n}(u) < \infty$. Working on a subsequence that we do not relabel, we can assume without the loss of generality that the liminf is actually a limit, that is,
 \begin{equation*}
\lim_{n \rightarrow \infty} \min_{\|u\|_{\nu_{n}}=1,\: u \in {S_{n}}^{\perp}} G_{n , \veps_n}(u) = \liminf_{n \rightarrow \infty} \min_{\|u\|_{\nu_n}=1,\: u \in {S_n}^{\perp}} G_{n , \veps_n}(u)  < \infty.
\end{equation*}
Consider now a sequence $\left\{ v_n \right\}_{n \in \N}$ with $\|v_n\|_{\nu_n}=1$ and $v_n \in {S_{n}}^{\perp}$ such that
\begin{equation*}
 \lim_{n \rightarrow \infty}  G_{n , \veps_n}(v_n)= \lim_{n \rightarrow \infty} \min_{\|u\|_{\nu_n}=1,\: u \in {S_{n}}^{\perp}} G_{n , \veps_n}(u)  < \infty.
\end{equation*}
Using the compactness from Theorem \ref{DiscreteGamma} and working on a subsequence that we do not relabel, we may assume that
\begin{equation}
v_n\overset{{TL^2}}{\longrightarrow} v, \: \text{as} \: n \rightarrow \infty,
\end{equation} 
for some $v \in L^2(D)$. From Proposition \ref{ContinuityInnerProdTL2}, $\|v\|_\rho = \lim_{n \rightarrow \infty} \|v_n\|_{\nu_{n}}=1$ and $\langle v , u_i \rangle_\rho= \lim_{n \rightarrow \infty} \langle v_n, \tilde{u}_i^{n}  \rangle_{\nu_n} =0$ for every $i=1, \dots , k-1$. In particular, $\|v\|_{\rho}=1$ and $v \in S^{\perp}$. Moreover, given that $v_n \overset{{TL^2}}{\longrightarrow} v$, it follows from the liminf inequality of Theorem \ref{DiscreteGamma} that
\begin{align*}
\min_{\|u\|_\rho=1, \: u \in S^{\perp} } \sigma_\eta G(u) &\leq \sigma_\eta G(v) 
\\
&\leq \liminf_{n \rightarrow \infty} G_{n, \veps_{n}}(v_n ) 
\\
&= \lim_{n \rightarrow \infty} \min_{\|u\|=1, \: u \in {S_{n}}^{\perp}} G_{n , \veps_n}(u)
\\ &\leq  \liminf_{n \rightarrow \infty} \sup_{\tilde{S} \in \Sigma_{k-1}^{(n)}}  \min_{\|u\|=1, \: u \in \tilde{S}^{\perp}} G_{n, \veps_n}(u)
\\& = \liminf_{n \rightarrow \infty} \frac{2\lambda_k^{(n)}}{n\veps_n^2} .
\end{align*}
Thus showing \eqref{ClaimEigenvalues} in all cases. Finally, since $S \in \Sigma_{k-1}$ was arbitrary, taking the supremum over all $S \in \Sigma_{k-1}$ and using the Courant-Fisher maxmini principle we deduce that
\begin{equation*}
\sigma_\eta \lambda_k \leq \liminf_{n \rightarrow \infty} \frac{2\lambda_k^{(n)}}{n \veps_n^2}.
\end{equation*}
\textbf{Step 2:} Now we prove that $\limsup_{n \rightarrow \infty} \frac{2 \lambda_k^{(n)}}{n \veps_n^2} \leq \lambda_k$. Consider $\left\{ u_1^{n}, \dots , u_{k-1}^{n} \right\}$ an orthonormal set (with respect to $\langle\cdot, \cdot \rangle_{\nu_n}$) with $u_i^{n}$ an eigenvector of $\mathcal{L}_{n , \veps_n}$ associated to $\lambda_i^{(n)}$ (this is possible because $\mathcal{L}_{n , \veps_n}$ is self-adjoint with respect to $\langle\cdot, \cdot \rangle_{\nu_n}$). Consider then $S^*_n:= \Span\left\{ u_1^{n} , \dots , u_{k-1}^{n}  \right\}$. We have:
\begin{equation*}
\frac{2\lambda_{k}^{(n)}}{n \veps_n^2} = \sup_{S \in \Sigma_{k-1}^{(n)}} \min_{\|u\|_{\nu_n}=1, \: u \in S^\perp} G_{n , \veps_n}(u) = \min_{\|u\|_{\nu_n}=1, \: u \in {S_n^*}^\perp} G_{n , \veps_n}(u).
\end{equation*}
Working along a subsequence that we do not relabel, we can assume without the loss of generality that $\limsup_{n \rightarrow \infty} \frac{2\lambda_k^{(n)}}{n \veps_n^2} = \lim_{n \rightarrow \infty} \frac{2\lambda_k^{(n)}}{n \veps_n^2}$.  Note that by the induction hypothesis, for every $i=1 , \dots , k-1$ we have:
\begin{equation}
\lim_{n \rightarrow \infty} G_{n , \veps_n}(u_i^{n}) = \lim_{n \rightarrow \infty} \frac{2\lambda_{i}^{(n)}}{n \veps_n^2} = \sigma_\eta \lambda_i < \infty.
\end{equation}
Thanks to this, we can use the compactness from Theorem \ref{DiscreteGamma} to conclude that for every $i =1 , \dots, k-1$ (working with a subsequence  that we do not relabel) :
\begin{equation*}
u_i^{n} \overset{{TL^2}}{\longrightarrow} u_i, \: \text{as} \: n \rightarrow \infty,
\end{equation*}
for some $u_i \in L^2(D)$. From Proposition \ref{ContinuityInnerProdTL2}, $\langle u_i , u_j \rangle_\rho= \lim_{n \rightarrow \infty} \langle u_i^{n}, u_j^{n}  \rangle_{\nu_n}=0$ for $i \not = j$ and $ \|u_i\|_\rho=  \lim_{n \rightarrow \infty} \|u_i^{n}\|_{\nu_n} =1 $ for every $i$. Take $S:=\Span \left\{u_1, \dots , u_{k-1} \right\}$, note that in particular $S \in \Sigma_{k-1}$. Also, take $v \in S^{\perp}$ with $\|v\|_\rho=1$ and such that:
\begin{equation}
\sigma_\eta G(v) = \min_{\|u\|_\rho=1, \: u \in S^\perp} \sigma_\eta G(u) \leq \sigma_\eta \lambda_k.
\label{Aux0}
\end{equation}
%Fix $\epsilon>0$, we can find $v \in S^{\perp}$ with $\|v\|=1$ such that:
%\begin{align*}
%\langle Lv , v \rangle \leq \min_{\|u\|=1, \: u \in S^\perp} \langle Lu, u \rangle + \epsilon
%\\
%\\\leq \lambda_m + \epsilon,
%\end{align*}
The last inequality in the previous expression holds thanks to the Courant-Fisher maxmini principle. By the limsup inequality from Theorem \ref{DiscreteGamma}, we can find $\left\{ v_n \right\}_{n\in \N}$ with $v_n \overset{{TL^2}}{\longrightarrow} v$ as $ n \rightarrow \infty$ and such that $ \limsup_{n \rightarrow \infty}  G_{n , \veps_n}(v_n) \leq \sigma_\eta G(v)$. Let $\tilde{v}_n$ be given by 
 \begin{equation*}
 \tilde{v}_n := v_n  - \sum_{i=1}^{k-1}\langle v_n , u_{i}^{n} \rangle_{\nu_n}u_{i}^{n}.
 \end{equation*}
Note that $\tilde{v}_n \in {S^*_{n}}^{\perp}$. Also note that from Proposition \ref{ContinuityInnerProdTL2}, we deduce that $\langle v_n , u_i^{n} \rangle_{\nu_n} \rightarrow 0$ as $n \rightarrow \infty$ for all $i=1, \dots, k-1$ and thus $\tilde{v}_n \overset{{TL^2}}{\longrightarrow} v$ as $n \rightarrow \infty$. Moreover,
\begin{align}
\begin{split}
G_{n , \veps_n}(\tilde{v}_n)  & =  \frac{2}{n \veps_n^2} \langle \mathcal{L}_{n, \veps_n} \tilde{v}_n, \tilde{v}_n  \rangle \\
&= \frac{2}{n \veps_n^2} \langle \mathcal{L}_{n, \veps_n} v_n, v_n  \rangle - \frac{2}{n \veps_n^2} \sum_{i=1}^{k-1} \langle v_n , u_i^n \rangle_{
\nu_n} \langle \mathcal{L}_{n, \veps_n} v_n , u_i^n \rangle_{\nu_n}-  \frac{2}{n \veps_n^2} \sum_{i=1}^{k-1} \langle v_n , u_i^n \rangle_{
\nu_n} \langle \mathcal{L}_{n, \veps_n} u_i^n , \tilde{v}_n \rangle_{\nu_n}
\\
&= G_{n, \veps_n}(v_n) -  \sum_{i=1}^{k-1} \frac{ 2\lambda_{i}^{(n)}}{n \veps_n^2} \langle v_n , u_i^{n} \rangle_{\nu_n}^2-  \frac{2}{n \veps_n^2} \sum_{i=1}^{k-1} \lambda_i^{(n)}\langle v_n , u_i^n \rangle_{
\nu_n} \langle  u_i^n , \tilde{v}_n \rangle_{\nu_n}
\\
&= G_{n, \veps_n}(v_n) -  \sum_{i=1}^{k-1} \frac{ 2\lambda_{i}^{(n)}}{n \veps_n^2} \langle v_n , u_i^{n} \rangle_{\nu_n}^2
\\
& \leq G_{n, \veps_n}(v_n).
%G_{n , \veps_n}(v_n) - \sum_{i=1}^{k-1}\frac{ 2\lambda_{i}^{(n)}}{n \veps_n^2} \langle v_n , u_i^{n} \rangle_{\nu_n}^2 \leq G_{n , \veps_n}(v_n) .
\end{split}
\end{align}
Therefore,
\begin{equation}
\limsup_{n \rightarrow \infty} G_{n , \veps_n}( \tilde{v}_n )  \leq \limsup_{n \rightarrow \infty}  G_{n , \veps_n}(v_n)  \leq \sigma_\eta G(v).
\label{Aux1}
\end{equation}
Since $\tilde{v}_{n}\overset{{TL^2}}{\longrightarrow} v$ and $\|v\|_\rho=1$, once again from Proposition \ref{ContinuityInnerProdTL2}  we obtain $\lim_{n \rightarrow \infty}\|\tilde{v}_n\|_{\nu_n}=1 $. In particular we can set $\tilde{u}_n:= \frac{\tilde{v}_n}{\| \tilde{v}_n\|_{\nu_n}}$ and use \eqref{Aux1} together with \eqref{Aux0} to conclude that:
\begin{align*}
 \lim_{n \rightarrow \infty}\frac{2 \lambda_{k}^{(n)}}{n\veps_n^2} &= \lim_{n \rightarrow \infty} \min_{\|u\|_{\nu_n}=1, \: u \in {S^*_{n}}^{\perp}  }  
  G_{n, \veps_n}(u)
\\&\leq  \limsup_{n \rightarrow \infty} G_{n , \veps_n} (\tilde{u}_n)
\\ & = \limsup_{n \rightarrow \infty} G_{n , \veps_n} (\tilde{v}_n) 
\\& \leq \sigma_\eta G(v)
\\&\leq \sigma_\eta \lambda_k,
\end{align*}
which implies the desired result.
%
%From now on we denote by $\lambda_1 < \lambda_2 < \lambda_3 < ...$ the \textbf{different} eigenvalues of $L$, and similarly $\lambda_1^{(n)} < \lambda_2^{(n)}< \dots$ the \textbf{different} eigenvalues of $L^{(n)}$. Denote by $E_k$ the eigenspace associated to the eigenvalue $\lambda_k$ of $L$, and similarly $E_k^{(n)}$ the eigenspace of $L^{(n)}$ associated to $\lambda_{k}^{(n)}$. We also consider the projections $\Proj_{k}: L^2(\mu) \rightarrow L^2(\mu)$ on $E_k$  and $\Proj_{k}^{(n)}: L^2(\mu_n) \rightarrow L^2(\mu_n)$ the projection on $E_k^{(n)}$.
%\paragraph{}
% From the previous theorem we have the following corollary:
%
%\begin{corollary}
%For every $m \in \N$ we have:
%\begin{enumerate}
%\item $\lim_{n \rightarrow \infty} \lambda_m^{(n)} = \lambda_m$
%\item $\lim_{n \rightarrow \infty} \dim \left(E_m^{(n)} \right) = \dim(E_m) $
%\end{enumerate}
%\label{ConvergenceEigenvalues}
%\end{corollary}
\subsection{Convergence of Eigenprojections}
We prove the second and third part of Theorem \ref{ConvergenceSpectrumTheorem}. We recall that the numbers $\bar{\lambda}_1 < \bar{
\lambda}_2 < \dots$
denote the distinct eigenvalues of $\mathcal{L}_{n , \veps_n}$. For a given $k \in \N$, we recall that $s(k)$ is the multiplicity of the eigenvalue $\bar{\lambda}_k$ and that $\hat{k} \in \N$ is such that $$\bar{\lambda}_k =  \lambda_{\hat{k}+1} = \dots = \lambda_{\hat{k}+s(k)}.$$
We let $E_k$ be the subspace of $L^2(D)$ of eigenfunctions of $\mathcal{L}$ associated to $\overline{\lambda}_k$, and for large $n$  we let $E_k^{(n)}$ be the subspace of $\R^n$ generated by all the eigenvectors of $\mathcal{L}_{n , \veps_n}$ corresponding to  all eigenvalues listed in $\lambda_{\hat{k}+1}^{(n)}, \dots , \lambda_{\hat{k}+s(k)}^{(n)}$. We remark that by the convergence of the eigenvalues proved in Subsection \ref{SecConvEigenvalues} we have 
\begin{equation}
 \lim_{n \rightarrow \infty} \dim(E_k^{(n)} ) =\dim(E_k)= s(k). 
 \label{convergenceDimension}
\end{equation}
%
%To avoid introducing new notation, we use the convention that the sequence $\lambda_1^{(n)} < \lambda_{2}^{(n)}< \dots$ represents the different eigenvalues of $\mathcal{L}_{n , \veps_n}$  and that $E_1^{(n)} , E_2^{(n)}, \dots$ represent the respective associated eigenspaces. Likewise, the sequence $\lambda_1 < \lambda_2 < \dots$ represents the different eigenvalues of $\mathcal{L}$ and $E_1 , E_2, \dots$ represent the corresponding associated eigenspaces. 
%
%
%
%
%\begin{remark}
%With the previous convention, from the first part of Theorem \ref{ConvergenceSpectrumTheorem} we conclude that for every $k \in \N$
%\begin{equation*}
%\lim_{n \rightarrow \infty}    \frac{2 \lambda_{k}^{(n)} }{ n \veps_n^{2} }= \sigma_\eta \lambda_k.
%\end{equation*}
%and
%\begin{equation*}
%\lim_{n \rightarrow \infty} \dim(E_k^{(n)}) = \dim(E_k).
%\end{equation*}
%\label{RemarkConvergenceDifferentEigenvalues}
%\end{remark}
%
We prove simultaneously the second and third statement of Theorem \ref{ConvergenceSpectrumTheorem}. The proof is by induction on $k$. 

\emph{Base Case:} Let $k=1$. Suppose that $u_n \overset{{TL^2}}{\longrightarrow} u$. We need to show that $\Proj_1^{(n)}(u_n)  \overset{{TL^2}}{\longrightarrow} \Proj_1(u)$. Now, note that since the domain $D$ is connected, the multiplicity of $\overline{\lambda}_1$ is equal to one. In particular, $\Proj_1(u)$ is the function which is identically equal to 
$$\langle u,1\rangle_\rho= \int_{D} u d \nu(x). $$
On the other hand, thanks to \eqref{convergenceDimension}, it follows that for all large enopugh $n$, we have $\dim(E_1^{(n)})=1$ (note that in particular this means that assymptotcally the graphs are connected regardless of what kernel $\eta$ is being used). Therefore, for large enough $n$, $\Proj_1^{(n)}(u_n)$ is the function which is identically equal to $\langle u_n, 1 \rangle_{\nu_n} $. Proposition \ref{ContinuityInnerProdTL2} implies that $\lim_{n \rightarrow \infty} \langle u_n, 1 \rangle_{\nu_n}  = \langle u, 1 \rangle_{\rho}$ and thus $\Proj_1^{(n)}(u_n) \overset{{TL^2}}{\longrightarrow} \Proj_1(u)$ as desired. The second statement of Theorem \ref{ConvergenceSpectrumTheorem} is trivial in this case since for large enough $n$, the only two eigenvectors of $\mathcal{L}_{n , \veps_n}$ with eigenvalue $\lambda_1^{(n)}=0$ and with $\| \cdot\|_{\nu_n}$-norm equal to one is the function which is identically equal to one or the function that is identically equal to $-1$.

\emph{Inductive Step:}
Now, suppose that the second and third statements of Theorem \ref{ConvergenceSpectrumTheorem} are true for $1, \dots , k-1$. We want to prove the result for $k$.  Let $j \in \left\{ \hat{k}+1, \dots, \hat{k}+ s(k)  \right\}$.
We start by proving the second statement of the theorem. Consider $\left\{  u_{j}^n\right\}_{n \in \N}$ as in the statement. From \eqref{LaplacianAndEnergyProofs} it follows that $G_{n, \veps_n}(u_j^n)=\frac{2 \lambda_j^{(n)}}{n \veps_n^2} $. Now, from Subsection \ref{SecConvEigenvalues}, we know that 
$$ \lim_{n \rightarrow \infty} \frac{2 \lambda_j^{(n)}}{n \veps_n^2} = \sigma_\eta \lambda_j  $$ 
and so in particular we have:
$$  \sup_{n \in \N} G_{n , \veps_n}(u_j^n) < \infty.  $$
Since the norms of the $u_j^n$ are equal to one, the compactness statement from Theorem \ref{DiscreteGamma}, implies that $\left\{ u_{j}^n \right\}_{n \in \N}$ is pre-compact.  We have to prove now that every cluster point of $\left\{ u_{j}^n \right\}_{n \in \N}$ is an eigenfunction of $\mathcal{L}$ with eigenvalue $\lambda_j$. So without the loss of generality let us assume that $u_{j}^n \overset{{TL^2}}{\longrightarrow} u_j$ for some $u_j$. Our goal is to show that $u_j$ is an eigenfunction of $\mathcal{L}$ with eigenvalue $\lambda_j$.

By the induction hypothesis, we have $ \Proj_i ^{(n)}(u_j^n) \overset{{TL^2}}{\longrightarrow} \Proj_i(u_j) $ for every $i =1, \dots, k-1$. On the other hand, since $\Proj_i ^{(n)}(u_j^n)=0$ for every $n \in \N $ and for every $i=1, \dots, k-1$, we conclude that $\Proj_i(u_j)=0$ for all $i=1, \dots, k-1$. A straightforward computation as in the proof of Proposition  \ref{PropoSpectrumL} shows that:
\begin{equation}
G(u_j)= \sum_{i=k}^{\infty} \overline{\lambda}_i \| \Proj_i(u_j)\|_\rho^2\geq \overline{\lambda}_k  \sum_{i=k}^{\infty}  \| \Proj_i(u_j)\|_\rho^2=\overline{\lambda}_k \| u_j\|_\rho^2.
\label{AuxPrueba1}
\end{equation}
In addition, since $\|u_j^n\|_{\nu_n}=1$ for all $n$, we deduce from Proposition \ref{ContinuityInnerProdTL2} that $\|u_j\|_\rho=1$. Thus, 
$$ G(u_j) \geq \overline{\lambda}_k .$$ 
On the other hand, the liminf inequality  from Theorem \ref{DiscreteGamma} implies that:
$$  \sigma_\eta \overline{\lambda}_k= \sigma_\eta \lambda_j = \lim_{n \rightarrow \infty} \frac{2 \lambda_j^{(n)}}{n \veps_n^2}  = \lim_{n \rightarrow \infty} G_{n , \veps_n}(u_n^j) \geq \sigma_\eta G(u^j) \geq \sigma_\eta \overline{\lambda}_k.  $$
Therefore, $G(u_j)= \overline{\lambda}_k$ and from \eqref{AuxPrueba1} we conclude that $\| \Proj_i(u_j)\|_\rho =0$ for all $i \not =k$. Thus, $u_j$ is an eigenfunction of $\mathcal{L}$ with corresponding eigenvalue $\lambda_j$ ($=\overline{\lambda}_k$).

Now we prove the third statement from Theorem \ref{ConvergenceSpectrumTheorem}. Suppose that $u^n  \overset{{TL^2}}{\longrightarrow}u$. We want to show that $\Proj_k^{(n)}(u^n)  \overset{{TL^2}}{\longrightarrow} \Proj_k(u)$. To achieve this we prove that for a given sequence of natural numbers there exists a further subsequence for which the convergence holds. We do not relabel subsequences to avoid cumbersome notation. 

From \eqref{convergenceDimension} it follows that for large enough $n$, $\dim(E_k^{(n)})= s(k)$. Hence, for large enough $n$, we can consider $\left\{u_1^n, \dots, u_{s(k)}^n \right\}$ an orthonormal basis (with respect to the inner product $\langle\cdot, \cdot \rangle_{\nu_n}$) for $E_k^{(n)}$, where $u_j^n$ is an eigenvector of $\mathcal{L}_{n , \veps_n}$ with corresponding eigenvalue $\lambda_{\hat{k} +j}^{(n)}$. Now, by the first part of the proof, for every $j=1 , \dots, s(k)$, the sequence $\left\{ u_j^n \right\}_{n \in \N}$ is pre-compact in $TL^2$. Therefore, passing to a subsequence that we do not relabel we can assume that for every $j =1 , \dots , s(k)$ we have:
\begin{equation}
u_j^n \overset{{TL^2}}{\longrightarrow} u_j, \: \text{as} \: n\rightarrow \infty
\label{auxLastProof2}
\end{equation}
for some $u_j \in L^2(D)$. From  \eqref{ContinuityInnerProdTL2}, the $u_j$ satisfy $\|u_j\|_\rho= 1$ for every $j$ and $\langle u_i , u_j \rangle_\rho=0$ for $i \not = j$. In other words, $\left\{ u_1, \dots , u_{s(k)} \right\}$ is an orthonormal set in $L^2(D)$ (with respect to $\langle \cdot, \cdot\rangle_\rho$). Furthermore,  $u_j \in E_k$ for all $j$ by the first part of the proof. In other words, $\left\{u_1, \dots , u_{s(k)} \right\}$ is an orthonormal basis for $E_k$ and in particular
$$\Proj_k(u)= \sum_{j=1}^{s(k)}\langle u, u_j \rangle_{\rho} u_j.$$
On the other hand, for large enough $n$, we have
$$  \Proj_k^{(n)}(u^n) = \sum_{j=1}^{s(k)}\langle u^n, u_j^n \rangle_{\nu_n} u_j^n.  $$
Finally, the fact that $u^n \overset{{TL^2}}{\longrightarrow} u$ and \eqref{auxLastProof2} combined with  Proposition \ref{ContinuityInnerProdTL2} imply that
$$  \Proj_k^{(n)}(u^n) = \sum_{j=1}^{s(k)}\langle u^n, u_j^n \rangle_{\nu_n} u_j^n \overset{{TL^2}}{\longrightarrow}  \sum_{j=1}^{s(k)}\langle u, u_j \rangle_{\rho} u_j = \Proj_k(u). $$

\subsection{Consistency of spectral clustering} \label{sec:csc}

Here we prove statement 4. of Theorem \ref{ConvergenceSpectrumTheorem}.

The procedure in Algorithm 1, can be reformulated as follows.
Let $\mu_n = (u_1^n, \dots, u_k^n)_\sharp \nu_n$, where 
 $u_1^n, \dots, u_k^n$ are orthonormal eigenvectors of $\mathcal{L}_{n, \veps_n}$ corresponding to
 eigenvalues $\lambda_1^{(n)}, \dots, \lambda_k^{(n)}$, respectively. 
Consider the functional $F_{\mu_n,k}$. Let $\z_n$ be its minimizer, and let 
 $\tilde G^n_1, \dots \tilde G^n_k$ be corresponding clusters.  
  The clusters $G_1, \dots, G_k$ of Algorithm 1 are defined by $G_i = (u_1^n, \dots, u_k^n)^{-1}(\tilde G_i)$.
  
By Theorem \ref{th:km} the sequence $\z_n$ is precompact. By Corollary \ref{cor:km} the sequence of measures $\mu_{n}^i= \mu_n \llcorner_{\tilde G^n_i}$ is precompact for all $i=1, \dots, k$.  Consider a subsequence along which $\mu_{n}^i$ converges for every $i=1, \dots, k$, and denote the limit by $\mu^i$. Since $z_n^i = \dashint y d\mu_n^i(y)$ it follows that $z_n^i$ converge as $n \to \infty$, along the same subsequence.
 By statement 2. of Theorem \ref{ConvergenceSpectrumTheorem} along a further subsequence $(\nu_n, u^n_i)$ converge to $(\nu, u_i)$ in $TL^2$ sense  for all $i=1, \dots, k$ as $n \to \infty$. Furthermore from the definition of $TL^2$ convergence follows that
 measures $\mu_n$ converge in the Wasserstein sense to $ \mu:=(u_1, \dots u_k)_\sharp \nu$. 
 Combined with convergence of $\mu_n^i$ to $\mu^i$ implies, via Lemma \ref{lem:tl_char}, that ($\mu_n, \chi_{\tilde G_i^n})$ converge in $TL^2$ topology to $(\mu, \chi_{\tilde G_i})$.
 Consequently, by Lemma \ref{TL2comp}, 
 $(\nu_n, \chi_{\tilde G_i^n} \circ (u_1^n, \dots, u_k^n))$ converge to 
 $(\nu, \chi_{\tilde G_i} \circ (u_1, \dots, u_k))$ in $TL^2$ topology.
 Noting that $\chi_{G^n_i} =\chi_{\tilde G_i^n} \circ (u_1^n, \dots, u_k^n)$ and
 $\chi_{G_i} = \chi_{\tilde G_i} \circ (u_1, \dots, u_k)$ implies
  that $\nu_n \llcorner_{G^n_i}$ converges weakly to $\nu \llcorner_{G_i}$ as desired.
\nc

%%%%%%%%%%%%%%%%%%%%%%%%%%%%%%%%%%%%%%%%%%%%%%%%%
\section{Convergence of the spectra of normalized graph Laplacians}
\label{SectionGammaNormalized}
We start by proving Theorem \ref{DiscreteGammaNormalized}. Recall that for given $u_n \in L^2(\nu_n)$ 
$$  \overline{G}_{n , \veps_n}(u_n) =   \frac{1}{ n \veps_n^2}  \sum_{i,j} W_{i,j} \left(   \frac{u_n(\x_i)}{\sqrt{\mathcal{D}_{ii}}}   -  \frac{u_n(\x_j)}{\sqrt{\mathcal{D}_{jj}}} \right)^2 ,   $$
where $W_{ij}= \eta_{\veps_n}(\x_i-\x_j) $ and $\mathcal{D}_{ii} =\sum_{k=1}^n\eta_{\veps_n}(\x_i-\x_k)$. With a slight abuse of notation we set 
$$\mathcal{D}(\x_i):= \mathcal{D}_{ii}.$$ 
For $u_n \in L^2(\nu_n)$, define $\bar{u}_n \in L^2(\nu_n)$ by 
\begin{equation}
\bar{u}_n(\x_i) := \frac{u_n(\x_i)}{\sqrt{\mathcal{D}(\x_i)/n }}, \quad i \in \left\{1,\dots,n \right\}.
\label{NormalizedFunctionDiscrete}  
\end{equation}
From the definition of $G_{n, \veps_n}$ and $\overline{G}_{n , \veps_n}$, it follows that $\overline{G}_{n , \veps_n}(u_n) =   G_{n , \veps_n} ( \bar{u}_n) $. Similarly, for every $u \in L^2(D)$ it is true that $ \overline{G}(u) =  G(\frac{u}{ \sqrt{\rho}})$.
\label{RemarkUnnormalizedToNormalized}
%In fact, for $\veps>0$, define the non-local Dirichlet energy $G_\veps:L^2(D) \rightarrow [0, \infty)$ by
%$$  G_\veps(u):= \frac{1}{\veps^2} \int_{D}\int_D \eta_\veps(x-y) (u(x) - u(y))^2 \rho(x)\rho(y) dxdy. $$
%Modifying the proof  of Theorem... in \cite{GammaGraphTV} in a straightforward way, one can obtain the following result.
%
%\begin{proposition}
%Assume that $\eta$ satisfies (K1)-(K3). Then, $G_\veps$, $\Gamma$-converge to $\sigma_\eta G$ as $\veps \rightarrow 0$ with respect to the $L^2(D)$-metric. The surface tension  $\sigma_\eta$ is defined in \ref{sigma_eta} and $G$ is given by \eqref{WeightedDirichlet}.
%Moreover, $\left\{ G_\veps \right\}_{\veps>0}$ satisfies the compactness property, that is, every family of functions $\left\{ u_\veps \right\}_{\veps>0} \subseteq L^2(D)$ bounded in $L^2(D)$ and satisfying
%\begin{equation*}
%\sup_{\veps>0} G_\veps(u_\veps ) < \infty,
%\end{equation*}
%is precompact in $L^2(D)$.
%\end{proposition}
To prove Theorem \ref{DiscreteGammaNormalized} we use the following lemma.
\begin{lemma}
Assume that the sequence $\left\{  \veps_n \right\}_{n \in \N} $ satisfies \eqref{HypothesisEpsilon}. With probability one the following statement holds: a sequence $\left\{u_n \right\}_{n \in \N}$, with $u_n \in L^2(\nu_n)$, converges to $u\in L^2(\rho)$ in the $TL^2$-metric if and only if  $\bar{u}_n  \overset{{TL^2}}{\longrightarrow} \frac{u}{\sqrt{ \beta_\eta \rho}}$, where $\bar{u}_n$ is defined in \eqref{NormalizedFunctionDiscrete} and where $\beta_\eta$ is defined in \eqref{BetaEta}.
\label{LemmaDiscreteGammaNormalized}
\end{lemma}
\begin{proof}
We prove that $u_n \overset{{TL^2}}{\longrightarrow} u$ implies $\bar{u}_n  \overset{{TL^2}}{\longrightarrow} \frac{u}{\sqrt{ \beta_\eta \rho}}$; the converse implication is obtained similarly. Let $\left\{ T_n \right\}_{n \in \N}$ be the transportation maps from Proposition \ref{thm:InifinityTransportEstimate}, which we know exist with probability one. Using the change of variables \eqref{chofvar} we obtain
$$  \frac{\mathcal{D}(X_i)}{ n } =  \int_{D} \eta_{\veps_n}(\x_i- T_n(y)) \rho(y) dy . $$
If $u_n \overset{{TL^2}}{\longrightarrow} u$, in particular from Proposition \ref{EquivalenceTLp} we have $u_n \circ T_n \overset{{L^2(\rho)}}{\longrightarrow} u$. By Proposition \ref{EquivalenceTLp}, in order to prove that $\bar{u}_n \overset{{TL^2}}{\longrightarrow} \frac{u}{\sqrt{ \beta_\eta \rho}}$, it is enough to prove that $\bar{u}_n \circ T_n \overset{{L^2(\rho)}}{\longrightarrow} \frac{u}{\sqrt{ \beta_\eta \rho}}$, which in turn is equivalent to $\bar{u}_n \circ T_n \rightarrow _{L^2(D)}\frac{u}{\sqrt{ \beta_\eta \rho}}$ due to the fact that $\rho$ satisfies \eqref{DensityBound}.
%
%If we define $d_{n}(x) :=  \int_{D} \eta_{\veps_n}(T_n(x)- T_n(y)) \rho(y) dy$ we see that $\bar{u}_n \circ T_n = \frac{u_n\circ T_n(x)}{\sqrt{d_n(x)}}$.
To achieve this, we first find an $L^\infty$-control on  the terms $\frac{1}{\sqrt{\mathcal{D}\circ T_n /n }}$ and then prove that $\frac{1}{\sqrt{\mathcal{D}\circ T_n /n }}$ converges point-wise to $\frac{1}{\sqrt{\beta_\eta \rho}}$. Since $u_n \circ T_n \overset{{L^2(D)}}{\longrightarrow} u$ this is enough to obtain the desired result. 
For that purpose, we fix an arbitrary $\alpha>0$ and define $\underline{\bm{ \eta}}^\alpha: [0, \infty) \rightarrow [0, \infty)$ and $\overline{\bm{\eta}}^\alpha : [0, \infty) \rightarrow [0, \infty)$ to be
\begin{equation} \label{HatEtaAlpha}
\underline{\bm{ \eta}}^\alpha(t) := 
\begin{cases}
 \bm{\eta}(t), & \te{ if }   t>2 \alpha \\
 \bm{\eta}(2 \alpha), & \te{ if }  t \leq 2 \alpha,
\end{cases}
\end{equation}
and
\begin{equation} \label{TildeEtaAlpha}
\overline{\bm{\eta}}^\alpha(t) := 
\begin{cases}
 \bm{\eta}(t),  \; \te{ if }   t>2 \alpha \\
 \bm{\eta}(0), \:  \te{ if }  t \leq 2 \alpha,
\end{cases}
\end{equation}
where we recall that $\bm{\eta}$ is the radial profile of the kernel $\eta$. We let $\underline{\eta}^\alpha$ and $\overline{\eta}^\alpha$ be the isotropic kernels whose radial profiles are $\underline{\bm{\eta}}^\alpha$  and $\overline{\bm{\eta}}^\alpha$ respectively. Note that thanks to assumption (K2) on $\bm{\eta}$, we have $\underline{\eta}^\alpha  \leq \eta \leq \overline{\eta}^\alpha$. 
Set
$$\hat{\veps}_n:=  \veps_n -  \frac{2 \|Id - T_n\|_{\infty}}{\alpha}, $$
$$\tilde{\veps}_n:=  \veps_n + \frac{2 \|Id - T_n\|_{\infty}}{\alpha}. $$
Note that thanks to the assumptions on $\veps_n$ and the properties of the maps $T_n$, for large enough $n$, $\hat{\veps}_n >0$,   $\frac{\hat{\veps}_n}{\veps_n} \rightarrow 1$ and $\frac{\tilde{\veps}_n}{\veps_n} \rightarrow 1$ as $n \rightarrow \infty$. In addition, from assumption (K2) on $\bm{\eta}$  and the definitions of $\underline{\bm{\eta}}^\alpha$, $\overline{\bm{\eta}}^\alpha$, $\hat{\veps}_n$ and $\tilde{\veps}_n$, it is straightforward to check that for large enough $n$ and for Lebesgue almost every $x,y\in D$, 
$$   \eta \left( \frac{T_n(x) -T_n(y)}{\veps_n} \right) \geq \underline{\eta}^\alpha \left(\frac{x-y}{\hat{\veps}_n} \right) ,  $$
and
$$   \eta \left( \frac{T_n(x) -T_n(y)}{\veps_n} \right) \leq \overline{\eta}^\alpha \left(\frac{x-y}{\tilde{\veps}_n} \right).  $$
From these inequalities, we conclude that for large enough $n$ and Lebesgue almost every $x \in D$
\begin{equation}
\int_{D} \eta_{\veps_n}( T_n(x) - T_n(y) ) \rho(y) dy   \geq  \left(\frac{\hat{\veps}_n}{ \veps_n} \right)^{d} \int_{D} \underline{\eta}^\alpha_{\hat{\veps}_n}(x-y ) \rho(y) dy 
\label{auxGammaNormalized0}
\end{equation}
and
\begin{equation}
\int_{D} \eta_{\veps_n}( T_n(x) - T_n(y) ) \rho(y) dy  \leq  \left(\frac{\tilde{\veps}_n}{ \veps_n} \right)^{d} \int_{D} \overline{\eta}^\alpha_{\tilde{\veps}_n}(x-y ) \rho(y) dy.
\label{auxGammaNormalized01}
\end{equation}
Given that $D$ is assumed to be a bounded open set with Lipschitz boundary, it is straightforward to check that exists a ball $B(0, \theta)$, a cone $C$ with nonempty interior and a family of rotations $\left\{ R_x \right\}_{x \in D}$ with the property that for every $x \in D$ it is true that $x+R_x(B(0,\theta) \cap C ) \subseteq D$. For large enough $n$ ( so that $1>\hat{\veps}_n>0$ ), and for almost every $x \in D$ we have:
\begin{align*}
\begin{split}
\int_{D} \underline{\eta}^\alpha_{\hat{\veps}_n}(x-y ) \rho(y) dy    \geq m \int_{D}  \underline{\eta}^\alpha_{\hat{\veps}_n}(x-y ) dy & = m \int_{x + \hat{\veps}_n h \in D }   \underline{\eta}^\alpha(h) dh \\
& \geq m \int_{R_x(B(0,\theta) \cap C)} \underline{\eta}^\alpha(h) dh 
= m \int_{B(0,\theta) \cap C} \underline{\eta}^\alpha(h) dh >0, 
\end{split}
\end{align*}
where in the first inequality we used assumption \eqref{DensityBound} on $\rho$, and we used the change of variables $h= \frac{x-y}{\hat{\veps}_n}$ to deduce the first equality; to obtain the last equality we used the fact that $\underline{\eta}^\alpha$ is radially symmetric. From the previous chain of inequalities and from \eqref{auxGammaNormalized0} we conclude that for large enough $n$ and for almost every $x \in D$ we have
\begin{equation*}
 \int_{D} \eta_{\veps_n}( T_n(x) - T_n(y) ) \rho(y) dy \geq b >0   
 \label{auxGammaNormalized1}
 \end{equation*}
for some positive constant $b$. Form the previous inequality we obtain the desired $L^\infty$-control on the terms $\frac{1}{\sqrt{\mathcal{D}\circ T_n/n}}$. 
It remains to show that for almost every $x \in D$,
\begin{equation}
 \lim_{n \rightarrow \infty} \int_{D} \eta_{\veps_n}( T_n(x) - T_n(y) ) \rho(y) dy = \beta_\eta \rho(x).
 \label{auxGammaNormalized4}
\end{equation}
For this purpose, we use the continuity of $\rho$ to deduce that for every $x \in D$,
\begin{equation}
\lim_{n \rightarrow \infty} \left| \underline{\beta}_{\alpha}  \rho(x) -   \int_{D} \underline{\eta}^{\alpha}_{\hat{\veps}_n}(x-y ) \rho(y) dy \right| =0,
\label{auxGammaNormalized20}
\end{equation}
where $\underline{\beta}_\alpha = \int_{\R^d}  \underline{\eta}^{\alpha}(h)dh $. Similarly, for every $x \in D$,
\begin{equation}
\lim_{n \rightarrow \infty}  \left| \overline{\beta}_{\alpha}  \rho(x) -   \int_{D} \overline{\eta}^{\alpha}_{\hat{\veps}_n}(x-y ) \rho(y) dy \right| =0,
\label{auxGammaNormalized30}
\end{equation}
where $\overline{\beta}_\alpha = \int_{\R^d}  \overline{\eta}^{\alpha}(h)dh $. From \eqref{auxGammaNormalized0}, we deduce that for large enough $n $, and for almost every $x \in D$,  
%Now, $u_n\circ T_n \overset{{L^2(D)}}{\longrightarrow} u $ implies that $\left\{ u_n \circ T_n \right\}_{n \in \N}$ has uniformly integrable second moments. This together with \eqref{auxGammaNormalized1} allows us to conclude that $\left\{  \bar{u}_n \circ T_n \right\}_{n \in \N}$ has uniformly integrable second moments. Thus, to prove that $\bar{u}_n \circ T_n \overset{{L^2(D)}}{\longrightarrow} \frac{u}{ \sqrt{\beta_\eta \rho}} $ , it is enough to show that the convergence holds in $L^2(D')$ for every $D'$ open set compactly contained in $D$. Hence, we fix $D' \subset \subset D$.  
\begin{align*}
\beta_\eta \rho(x) -   \int_{D} \eta_{\veps_n}( T_n(x) - T_n(y) ) \rho(y) dy 
%&  \leq  \beta_\eta \rho(x) - \left(\frac{\hat{\veps}_n}{ \veps_n} \right)^{d} \int_{D} \underline{\eta}^\alpha_{\hat{\veps}_n}(x-y ) \rho(y) dy 
% \\ & \leq \left( 1-  \left(\frac{\hat{\veps}_n}{ \veps_n} \right)^{d}  \right)\beta_\eta \rho(x) + \left(\frac{\hat{\veps}_n}{ \veps_n} \right)^{d}\left( \beta_\eta \rho(x) -  \int_{D} \underline{\eta}^\alpha_{\hat{\veps}_n}(x-y ) \rho(y) dy    \right)
 \leq & \beta_\eta \rho(x) - \left(\frac{\hat{\veps}_n}{ \veps_n} \right)^{d} \int_{D} \underline{\eta}^\alpha_{\hat{\veps}_n}(x-y ) \rho(y) dy 
\\
 \leq & \left(\frac{\hat{\veps}_n}{ \veps_n} \right)^{d}\left( \beta_\eta \rho(x) -  \underline{\beta}_\alpha \rho(x)  + \underline{\beta}_\alpha\rho(x) -    \int_{D} \underline{\eta}^\alpha_{\hat{\veps}_n}(x-y ) \rho(y) dy    \right) 
\\& +  \left( 1-  \left(\frac{\hat{\veps}_n}{ \veps_n} \right)^{d}  \right)\beta_\eta \rho(x). \end{align*}
Analogously, from \eqref{auxGammaNormalized01}, for almost every $x \in D$,
\begin{align*}
 \int_{D} \eta_{\veps_n}( T_n(x) - T_n(y) ) \rho(y) dy - \beta_\eta \rho(x) 
%&  \leq  \beta_\eta \rho(x) - \left(\frac{\hat{\veps}_n}{ \veps_n} \right)^{d} \int_{D} \underline{\eta}^\alpha_{\hat{\veps}_n}(x-y ) \rho(y) dy 
% \\ & \leq \left( 1-  \left(\frac{\hat{\veps}_n}{ \veps_n} \right)^{d}  \right)\beta_\eta \rho(x) + \left(\frac{\hat{\veps}_n}{ \veps_n} \right)^{d}\left( \beta_\eta \rho(x) -  \int_{D} \underline{\eta}^\alpha_{\hat{\veps}_n}(x-y ) \rho(y) dy    \right)
 \leq & \left(\frac{\tilde{\veps}_n}{ \veps_n} \right)^{d}\left(  \overline{\beta}_\alpha \rho(x)  - \beta_\eta \rho(x)  +  \int_{D} \overline{\eta}^\alpha_{\tilde{\veps}_n}(x-y ) \rho(y) dy  - \overline{\beta}_\alpha\rho(x)   \right) 
\\& +  \left(   \left(\frac{\tilde{\veps}_n}{ \veps_n} \right)^{d} -1  \right)\beta_\eta \rho(x).
\end{align*}
From these previous inequalities, \eqref{auxGammaNormalized20} and \eqref{auxGammaNormalized30} we conclude that for almost every $x\in D$,
\begin{equation*}
\limsup_{n \rightarrow \infty}   \left|   \beta_\eta \rho(x) -  \int_{D} \eta_{\veps_n}( T_n(x) - T_n(y) ) \rho(y) dy    \right| \leq \rho(x)( \overline{\beta}_{\alpha}   - \underline{\beta}_\alpha ).
\end{equation*}
Finally, given that $\alpha$ was arbitrary we can take $\alpha \rightarrow 0$ in the previous expression to deduce that the left hand side of the previous expression is actually equal to zero. This establishes \eqref{auxGammaNormalized4} and thus the desired result.
\end{proof}
The proof of Theorem \ref{DiscreteGammaNormalized} is now straightforward.
\begin{proof}[Proof of Theorem \ref{DiscreteGammaNormalized}]
\textbf{Liminf inequality:} Let $u \in L^2(D)$ and suppose that $\left\{u_n \right\}_{n \in \N}$, $u_n \in L^2(\nu_n)$, is such that $u_n \overset{{TL^2}}{\longrightarrow} u $. From Lemma \ref{LemmaDiscreteGammaNormalized}, we know that $\bar{u}_n \overset{{TL^2}}{\longrightarrow}  \frac{u}{\sqrt{\beta_\eta \rho}}$, where $\bar{u}_n$ was defined in \eqref{NormalizedFunctionDiscrete}. From Theorem \ref{DiscreteGamma} and the discussion at the beginning of this section, we obtain
\begin{equation*}
\liminf_{n \rightarrow \infty} \overline{G}_{n , \veps_n}(u_n) = \liminf_{n \rightarrow \infty}G_{n , \veps_n}\left( \bar{u}_n \right) \geq \sigma_\eta G\left(\frac{u}{\sqrt{\beta_\eta \rho }}\right) = \frac{\sigma_\eta}{\beta_\eta} \overline{G}(u),
\end{equation*}
where the inequality is obtained using the liminf inequality from Theorem \ref{DiscreteGamma}.

\textbf{Limsup inequality:} Let $u \in L^2(D)$. Since $\rho$ is bounded below by a positive constant, $\frac{u}{\sqrt{\beta_\eta \rho} }$ belongs to $L^2(D)$ as well. From the limsup inequality in Theorem \ref{DiscreteGamma}, there exists a sequence $\left\{  v_n \right\}_{n \in \N} $, $v_n \in L^2(\nu_n)$, with  $v_n \overset{{TL^2}}{\longrightarrow} \frac{u}{\sqrt{\beta_\eta \rho}}$ and such that
$$\limsup_{n \rightarrow \infty} G_{n , \veps_n}( v_n) \leq \sigma_\eta G\left(\frac{u}{\sqrt{\beta_\eta \rho}}\right) = \frac{\sigma_\eta}{\beta_\eta} \overline{G}(u).$$
Let us consider the function $u_n \in L^2(\nu_n)$ given by  $u_n(\x_i) :=  v_n(\x_i) \sqrt{\mathcal{D}(\x_i)/n }  $ for $i =1 , \dots, n$. Lemma \ref{LemmaDiscreteGammaNormalized} implies that $u_n \overset{{TL^2}}{\longrightarrow} u$. From the discussion at the beginning of this section we obtain
\begin{equation*}
\limsup_{n \rightarrow \infty} \overline{G}_{n , \veps_n}(u_n) = \limsup_{n \rightarrow \infty} G_{n , \veps_n}(v_n) \leq \frac{\sigma_\eta}{\beta_\eta}\overline{G}(u).
\end{equation*}
\textbf{Compactness:} Suppose that $\left\{ u_n  \right\}_{n \in \N}$, $u_n \in L^2(\nu_n)$, is such that 
\begin{equation*}
\sup_{n \in \N} \|u_n\|_{\nu_n}   < \infty, \quad
\sup_{n \in \N} \overline{G}_{n , \veps_n}(u_n) < \infty.
\end{equation*}
From the discussion at the beginning of the section, we deduce that $\sup_{n \in \N} G_{n , \veps_n}(\bar{u}_n) < \infty$ . Also, from the proof of Lemma \ref{LemmaDiscreteGammaNormalized}, the terms $\frac{1}{\sqrt{\mathcal{D}\circ T_n /n} }$ are uniformly bounded in $L^\infty$. This implies that $\sup_{n \in \N} \| \bar{u}_n\|_{L^2(\nu_n)}   < \infty$ as well. Hence, we can apply the compactness property from Theorem \ref{DiscreteGamma} to conclude that $\left\{ \bar{u}_n \right\}_{n \in \N}$ is precompact in $TL^2$. Using Lemma \ref{LemmaDiscreteGammaNormalized}, this implies that $\left\{ u_n \right\}_{n \in \N}$ is precompact in $TL^2$ as well.
\end{proof}

\begin{proof}[Proof of Theorem \ref{ConvergenceSpectrumTheoremNormalized}]
Using Theorem \ref{DiscreteGammaNormalized}, similar arguments to the ones used in the proof of Theorem \ref{ConvergenceSpectrumTheorem} can be used to establish statements 1., 2., and 3. of Theorem \ref{ConvergenceSpectrumTheoremNormalized}.

The proof of statement 4. (consistency of spectral clustering) of Theorem \ref{ConvergenceSpectrumTheoremNormalized} is analogous to the proof of the statement 4. of Theorem \ref{ConvergenceSpectrumTheorem} which is given in Subsection 
\ref{sec:csc}. The reason that the normalization step does not create new difficulties is the following: since the
eigenvectors $\uu^n := (u_1^n, \dots, u_k^n)$ of $\mathcal{N}^{sym}_{n,\veps_n}$ converge in $TL^2$ to eigenfunctions $\uu= (u_1^n, \dots, u_k^n)$ of $\mathcal{N}^{sym}$ along a subsequence, it can be shown that the normalized vectors  $\uu^n / \| \uu^n \|$  converge to 
$\uu / \| \uu \|$ in $TL^2$
provided that the set of $x \in D$ for  which $\uu =0$ is of $\nu$-measure zero. 
 
In fact, assuming that $\nu(\{ x \in D \: : \: \uu(x) = 0\}) =0$ let us show the $TL^2$ convergence. 
 From the assumption on the set of zeroes of $\uu$ follows that  $\lim_{H \to 0^+} \nu(\{ \|\uu(x)\| < H\}) = 0$. Let $U_H = \{(x,y) \in D \times D \::\: \|\uu(x)\| < H \}$.
 Given $n \in \N$ let $\pi_n \in \Pi(\nu_n, \nu)$ be such that 
 \[  \iint |x-y|^2 + \| \uu^n(x) - \uu(y) \|^2 d\pi_n(x,y) \leq 2 d_{TL^2}^2(\uu^n, \uu) . \]
 Then for any $H>0$
 \begin{align*}
 d_{TL^2}\left( \frac{\uu^n}{\|\uu^n\|}, \frac{\uu}{\|\uu\|}  \right) \leq & \iint |x-y|^2 d\pi_n(x,y) +
 \iint_{U_H} 2^2 d \pi_n(x,y)  \\
& + \iint_{D \times D \backslash U_H} \frac{ \lvert \lvert \| \uu^n(y) \| \uu(x) \pm \|\uu^n(y)\| \uu^n(y) - \|\uu(x)\| \uu^n(y)  \rvert \rvert ^2}{\|\uu(x) \|^2 \, \|\uu_n(y)\| ^2} d\pi_n(x,y) \\
\leq & 4  d_{TL^2}^2(\uu^n, \uu) + o(H) + \frac{16}{H^2} d_{TL^2}^2(\uu^n, \uu). 
 \end{align*}
The right hand side can be made arbitrarily small by first picking $H$ small enough and then $n$ 
large enough along the subsequence where $\uu^n$  converges to $\uu$. The convergence of normalized eigenvector $k$-tupples follows.

To show that $\nu(\{ x \in D \: : \: \uu(x) = 0\}) =0$ it suffices to show
 that the set of $x \in D$ for which $u^1(x)=0$ has zero Lebesgue measure. 
To show this, we need the extra technical condition that $\rho \in C^1(D)$. Because of it and the fact that $\rho$ is bounded away from zero, it follows from the regularity theory of elliptic PDEs, that the function $w_1 := \frac{u_1}{\sqrt{\rho}}$ is of class $C^{1, \alpha}(D)$ (for $\alpha \in (0,1)$) and is a solution of 
$$  - \divergence(\rho^2 \nabla w_1 ) - \tau_1 \rho^2 w_1 =0, \quad \forall x \in D.   $$ 
Consider the sets
\[ 
 N(w_1)  := \left\{  x \in D \: : \: w_1(x)= 0    \right\}  \qquad
 S(w_1) := \left\{ x \in N(w_1) \: : \: \nabla w_1(x)=0 \right\}. 
\] 
By the implicit function theorem, it follows that $N(w_1)\setminus S(w_1)$ can be covered by at most countable $d-1$ dimensional manifolds and hence it follows that the Lebesgue measure of $N(w_1)\setminus S(w_1)$ is equal to zero. On the other hand, it follows from the results in \cite{QHan}, that $S(w_1)$ is $(d-2)$-rectifiable, which in particular implies that the Lebesgue measure of $S(w_1)$ is equal to zero. Since $u_1^{-1}(\left\{0 \right\}) = N(w_1)$, we conclude that the set in which $u_1$ is equal to zero has zero Lebesgue measure.
\nc
\end{proof}

\begin{proof}[Proof of Corollary \ref{ConvergenceSpectrumTheoremNormalizedRW}] Given a sequence $\left\{ u_k^n \right\}_{n \in \N}$,  as in the statement of the corollary, we define
$$  w_k^n:= \mathcal{D}^{1/2} u_k^n. $$
From \eqref{ConditionNsymNrw} it follows that $w_k^n$ is an eigenvector of $\mathcal{N}^{sym}_{n,\veps_n}$. We consider a rescaled version of the vectors $w_k^n$, by setting
$$ \tilde{w}_k^n := \frac{w_k^n}{\sqrt{n}}= \frac{1}{\sqrt{n}} \mathcal{D}^{1/2} u_k^n. $$
From the proof of Lemma \ref{LemmaDiscreteGammaNormalized}, it follows that 
$$ \sup_{n \in \N} \| \tilde{w}_k^n \|_{\nu_n}  < \infty.$$
Thus, from Theorem \ref{ConvergenceSpectrumTheoremNormalized}, up to subsequence,
$$ \tilde{w}_{k}^n \overset{TL^2}{\longrightarrow} w, $$
for some $w \in L^2(D)$ which is an eigenfunction of $\mathcal{N}^{sym}$ with eigenvalue $\tau_k$. Hence, up to subsequence, from Lemma \ref{LemmaDiscreteGammaNormalized} it follows that
$$  u_k^n \overset{TL^2}{\longrightarrow} \frac{w}{\sqrt{\beta_\eta} \rho}. $$
By discussion of Subsection \ref{SectionSpectrumL}, it follows that $\frac{w}{\sqrt{\beta_\eta} \rho}$ is an eigenfunction of $\mathcal{N}^{rw}$ with eigenvalue $\tau_k$. 

The proof of convergence of clusters is the same as given in the proof of Theorem \ref{ConvergenceSpectrumTheorem} presented in Subsection \ref{sec:csc}.
\end{proof}

\bigskip
\noindent {\bf Acknowledgments.}
DS is grateful to  NSF (grant DMS-1211760) for its support. The authors
are thankful to Moritz Gerlach,  Matthias Hein, Thomas Laurent,  James
von Brecht, and Ulrike von Luxburg,
  for enlightening conversations. The authors would like to thank the
Center for Nonlinear Analysis of the Carnegie Mellon University for its
support. Furthermore they are thankful to ICERM, where part of the work
was done, for hospitality.

\bibliography{Biblio}
\bibliographystyle{siam}

\end{document}